\documentclass[10pt]{article}

\RequirePackage{amsthm,amsmath,amsfonts,amssymb}
\RequirePackage[numbers]{natbib}
\RequirePackage[citecolor=blue,urlcolor=blue]{hyperref}  
\RequirePackage{graphicx}

\usepackage{bbm}
\usepackage[margin=2.54cm]{geometry}
\usepackage[utf8]{inputenc}
\usepackage[T1]{fontenc}
\usepackage{amsmath} 
\usepackage{amssymb}
\usepackage{paralist}
\usepackage{hyperref}
\usepackage{amsthm}
\usepackage{mathrsfs}
\usepackage{color}
\usepackage{mathtools}
\mathtoolsset{showonlyrefs=true}
\usepackage{accents}
\usepackage{dsfont}
\usepackage{xcolor}
\theoremstyle{plain}

\newtheorem{theorem}{Theorem}[section]
\newtheorem{lemma}[theorem]{Lemma}
\newtheorem{corollary}[theorem]{Corollary}
\newtheorem{proposition}[theorem]{Proposition}

\theoremstyle{definition}

\newtheorem{definition}[theorem]{Definition}
\newtheorem{example}{Example}[section]

\theoremstyle{remark}
\newtheorem{remark}[theorem]{Remark}

\DeclareMathOperator{\Var}{Var}
\DeclareMathOperator{\Cov}{Cov}
\DeclareMathOperator{\supp}{supp}

\newcommand{\extmu}{{\tilde{\mu}}}

\renewcommand{\d}{\mathrm{d}}

\renewcommand{\k}{\kappa}

\newcommand{\R}{{\mathbb{R}}}
\renewcommand{\P}{{\mathbb{P}}}
\newcommand{\E}{{\mathbb{E}}}

\newcommand{\Z}{{\mathbb{Z}}}

\newcommand{\N}{\mathbb{N}}

\newcommand{\mc}{\mathcal}

\newcommand{\bs}{\boldsymbol}

\newcommand{\new}{\mathrm{new}}
\newcommand{\eps}{\varepsilon}

\newcommand{\1}{\mathds{1}}
\newcommand{\vertiii}[1]{{\left\vert\kern-0.25ex\left\vert\kern-0.25ex\left\vert
#1 
    \right\vert\kern-0.25ex\right\vert\kern-0.25ex\right\vert}}

\newcommand{\Fcal}{\mathcal{F}}
\newcommand{\Gcal}{\mathcal{G}}

\newcommand{\Kcal}{\mathcal{K}}
\newcommand{\Lcal}{\mathcal{L}}
\newcommand{\Mcal}{\mathcal{M}}
\newcommand{\Ncal}{\mathcal{N}}

\newcommand{\e}{\mathrm{e}}

\newcommand{\eqdef}{=\vcentcolon}
\def\showauthornotes{1}
\ifnum\showauthornotes=1
\newcommand{\Authornote}[2]{{\sf\small\color{blue}{[#1: #2]}}}
\else
\newcommand{\Authornote}[2]{}
\fi

\allowdisplaybreaks
\newcommand{\zd}{{{\mathbb Z}^d}}
\newcommand{\F}{{\mathcal F}}
\newcommand{\Feta}{\F_{{\text{grad}}}}
\newcommand{\Fkappa}{\F_{{\text{disord}}}}
\newcommand{\RR}{\ensuremath{\mathbb{R}}}
\newcommand{\rmd}{\,\mathrm{d}}
\newcommand{\ormd}{\mathrm{d}}
\newcommand{\bzd}{\ensuremath{{(\zd)^*}}}
\newcommand{\bzdl}{\ensuremath{{\Lambda^*}}}
\newcommand{\bzdlb}{\ensuremath{{\overline{\Lambda^*}}}}
\newcommand{\ubzd}{\ensuremath{{(\zd)^{*,\leftrightarrow}}}}
\newcommand{\ubzdl}{\ensuremath{{\Lambda^{*,\leftrightarrow}}}}
\newcommand{\ubzdlb}{\ensuremath{{\overline{\Lambda}^{*,\leftrightarrow}}}}

\newcommand{\C}{{\mathcal C}}
\newcommand{\A}{{\mathcal A}}

\newcommand{\Vn}{V_\new}
\newcommand{\tVn}{\tilde{V}_{\new}}
\newcommand{\V}{V}
\newcommand{\W}{W}

\makeatletter
\newcommand{\vast}{\bBigg@{4}}
\newcommand{\Vast}{\bBigg@{5}}
\makeatother

\begin{document}

\title{Gradient Gibbs measures with non-convex potentials and the universality class of the Gaussian Free Field}
\author{Simon Buchholz\thanks{Max Planck Institute for Intelligent Systems, \url{simon.buchholz@tuebingen.mpg.de}} \and Codina Cotar\thanks{University College London, \url{c.cotar@ucl.ac.uk}} \and Florian Schweiger\thanks{University of Bath, \url{fms68@bath.ac.uk}}}
\maketitle
\begin{abstract}
We study a general class of gradient interface models with Hamiltonian $H=\beta\sum V(\nabla\phi)$, $\beta>0$, assuming essentially  that the potential $V$ is even, $V'(s)\ge \alpha s$ on $[0,\infty)$ for some $\alpha>0$, and $-M\leq V''\le C$. We establish a Helffer-Sjöstrand representation for these models, and use it to prove that their scaling limits are Gaussian Free Fields (GFFs), and that their covariances decay at the same rate as the GFF.
This extends results for strictly convex potentials
to a large class of non-convex potentials and to arbitrary temperatures. Additionally, we prove Brascamp-Lieb and dimension-free Poincar\'e inequalities for the models.

We obtain these results by representing the interface as a mixture of gradient interface models with strictly convex potentials, extending an idea by Biskup-Spohn who had considered mixtures of Gaussians at moderate inverse temperature $\beta=1$. The construction of such a representation is one of the key new contributions of this work. 

\end{abstract}

\tableofcontents

\section{Introduction}\label{sec:introduction}
The main motivation for this work comes from gradient interface models: these are statistical mechanics models that can be used to model phase transitions in physical systems, or, in the case of vector-valued fields, solid materials. Informally, gradient models can be defined as a random field $(\phi_x)_{x\in \Z^d}\in \R^{\Z^d}$ with the Gibbs distribution informally given by (see Section~\ref{s:definitions} below for the formal definition)
\begin{equation}\label{eq:formal_grad_int_def}
\frac{\exp\left(-\frac{\beta}{2}\sum_{x, y: |x-y|=1} V(\phi(x)-\phi(y))\right)}{Z}\prod_{x\in\Z^d} \d\phi(x),
\end{equation}
where $\beta>0$ is the inverse temperature. It is conjectured that many features of these models are universal in the sense that they do not depend on the specific choice of $V$. One particularly interesting example is the scaling limit of the field, which should always be a multiple of the (continuum) Gaussian free field.

There are very few tools available to study general models like \eqref{eq:formal_grad_int_def}, and most research activity has focused on specific classes of $V$. We will describe what is known in Section \ref{sec:intro_interface}, but will highlight here those that have motivated the present work.

One well-studied class of models are the so-called Ginzburg-Landau interface models, where $V\in C^2(\R)$ with $0<c\le V''(x)\le C<\infty$. These models have a random-walk representation due to Helffer and Sjöstrand, and it has been used to establish their scaling limit \cite{MR1461951} and many other properties.

In another direction, in \cite{MR2322690,MR2778801} potentials $V$ that are log-mixtures of Gaussians in the sense that

\begin{equation}\label{e:def_GM}V(s)=-\log\int_{[0,\infty)}\e^{-\kappa s^2/2}\rho(\d \kappa)
\end{equation}
have been studied. The specific structure of $V$ allows to rewrite the law of the interface as a mixture of Gaussian fields in random environment, and in \cite{MR2778801} that property was used to find the scaling limit of such models.

The main idea of this work is to combine these two approaches. That is, we will study potentials $V$ that are log-mixtures of $V_\kappa\in C^2(\R)$ with $0<c\le V_\kappa''(x)\le C<\infty$, derive a Helffer-Sjöstrand representation in random environment for them, and use it to prove that under suitable assumptions their scaling limit is still a multiple of the Gaussian free field.

But before we discuss those results, we will address the question which potentials $V$ can be represented as such a log-mixture.

\subsection{Decomposition of \texorpdfstring{$\alpha$}{alpha}-monotone potentials}\label{sec:intro_decomp}

\subsubsection*{The decomposition}

We begin by introducing the class of potentials that we will study.
\begin{definition}\label{d:def_V}
A function $V\colon\R\to\R$ is called a \emph{strongly log-concave log-mixture} interaction if the following holds: 
It can be written as
\begin{equation}\label{e:def_V}
V(s)=-\log\int\exp(-V_\kappa(s))\rho(\d\kappa)
\end{equation}
where $(K,\Kcal,\rho)$ is some measure space, the map $(\kappa,s)\mapsto V_\kappa(s)$ is jointly measurable, and for each $\kappa$ the function $V_\kappa$ is in $C^2(\R)$, even, and there exists a constant $0<c_V$ such that $c_V\le V_\kappa''(s)$ for all $s\in\R$, $\kappa\in K$.

Given a function $\mathfrak{g}\colon[0,\infty)\to\R$, we say that the strongly log-concave log-mixture is of $\mathfrak{g}$-controlled growth, if additionally $V_\kappa''(s)\le \mathfrak{g}(s)$ for all $s\ge0$, $\kappa\in K$. If we can choose $\mathfrak{g}$ equal to some constant $C_V$, we say that the strongly log-concave log-mixture is of quadratic growth. 
\end{definition}

It is easy to derive some necessary conditions for a function $V$ to be a strongly log-concave log-mixture (see also Remark~\ref{decompgenrem} below). Clearly $V$ must be even, and (under mild assumptions on $\rho$) in $C^2(\R)$. Moreover, if $c_V\le V_\kappa''(s)$ for all $s\in\R$, then $V_\kappa(s)-\frac{c_Vs^2}{2}$ is a non-decreasing function on $[0,\infty)$. So if \eqref{e:def_V} holds, then $V(s)-\frac{c_Vs^2}{2}$ is non-decreasing as well.
\begin{definition}
    A function $V:\R\to \R$ is called $\alpha$-monotone if $V(s)-\alpha s^2/2$ is a non-decreasing function on $(0,\infty)$.
    If $V$ is $C^1$ and even, then equivalently $V$ is $\alpha$-monotone if $V'(s)\geq \alpha s$ for $s\geq 0$. 
\end{definition}
As we just saw, if $V$ is a strongly log-concave log-mixture, then necessarily $V$ is $c_V$-monotone for some $c_V>0$. 

One of our main results is that this is already close to optimal. Namely we can show that $\alpha$-monotone potentials
with lower bounded second derivative can be written as strongly log-concave log-mixtures, with a growth rate not worse than the one of $V''$. 
\begin{theorem}\label{t:decomp}
Suppose that $V$ is an $\alpha$-monotone $C^2$ even potential such that there is a constant $M>0$ such that $V''(s)\geq -M$. Fix $0<\zeta<1$ arbitrarily. Then $V$ is a  strongly log-concave log-mixture, and in fact in $\eqref{e:def_V}$ one can choose $c_V\ge (1-\zeta)\alpha$, $K$ to be a countable set and $\Kcal$ to be its powerset (so that $\kappa\to V_\kappa(\cdot)$ takes only countably many values). 

Moreover, we can choose the strongly log-concave log-mixture to be of $\mathfrak{g}$-controlled growth, where $\mathfrak{g}(s):=C\max(1,V''(s))$
for some constant $C<\infty$ (that depends on $V$). In particular, if $V''$ is bounded above, we can choose the strongly log-concave log-mixture to be of quadratic growth. 
\end{theorem}

To put this result into context, it is worthwhile to compare strongly log-concave log-mixtures to log-mixtures of Gaussians as in \eqref{e:def_GM}. If $V$ is a log-mixture of Gaussians, then
\[\e^{-V(\sqrt{2s})}=\int_{[0,\infty)}\e^{-\kappa s}\rho(\d \kappa)\]
i.e. $\e^{-V(\sqrt{2s})}$ is the Laplace transform of a positive measure. So by Bernstein's theorem on completely monotone functions, an even $V$ can be represented as in \eqref{e:def_GM} if and only if $s\mapsto\e^{-V(\sqrt{2s})}$ restricted to $(0,\infty)$ is completely monotone. There are some interesting potentials $V$ with this property, for example $V(s)=|s|^p$ with $0<p\le2$, and many results can be extended from Gaussian measures to such $V$, see in particular \cite{MR3846841} for general results, and Section \ref{sec:intro_interface} for applications to interface models.

However, the condition that $\mathbb{R}_+\ni s\mapsto\e^{-V(\sqrt{2s})}$ is completely monotone on $(0,\infty)$ is very restrictive. Indeed, it requires in particular that $V(\sqrt{\cdot})\in C^\infty((0,\infty))$, with restrictions on the signs of all its derivatives. Meanwhile, Theorem \ref{t:decomp} only requires an assumption for $V''$. Additionally, if $V$ is a log-mixture of Gaussians, there is no reason for the same to be true for $\beta V$ for $\beta\in(0,\infty)$. Thus our decomposition is far more general.

\begin{example}\label{ex:function}
For a concrete example, consider the function $V(t)=t^2/8-\cos(t)$. This is an even $C^2$ function that is $\alpha$-monotone
    for $\alpha\leq \alpha_0$ where $\alpha_0\approx0.0327$. Indeed, $V'(t)=t/4+\sin(t)> 0.0327t$, as can be checked numerically. But it is not a log-mixture of Gaussians because it can be checked that the second derivative of $\e^{-V(\sqrt{2s})}=\e^{-s/4+\cos(\sqrt{2s})}$ changes sign.
    In Figure \ref{fig:1} below, we will illustrate what our decomposition looks like for this particular example.
\end{example}

Let us point out that the main difficulty in Theorem \ref{t:decomp} is to ensure the $C^2$-regularity of the $V_\kappa$ and the uniform upper bound on $V_\kappa''$. Without those requirements it would be much easier to obtain a decomposition as a mixture of uniformly convex potentials. Indeed we have the following result.
\begin{proposition}\label{p:trivialdecomp}
Let $V\in C(\R)$ be even and $\alpha$-monotone for some $\alpha>0$. Then there are even functions $V_\kappa\colon\R\to\R\cup\{+\infty\}$ for $\kappa\in[0,\infty]$ such that $V_\kappa\in C^2(I_\kappa)$ for some open interval $I_\kappa\subset\R$, $V_\kappa''(s)\ge\alpha$ for all $s\in I_\kappa$, and $V_\kappa=+\infty$ on $\R\setminus I_\kappa$, as well as a positive Radon measure $\rho$ on $[0,\infty]$ such that \eqref{e:def_V} holds. 
\end{proposition}
Here by $[0,\infty]$ we mean the one-point compactification of $[0,\infty)$. 
The proof being very short, we give it right away.
\begin{proof}
Note that $\tilde V(x)=V(x)-\frac{\alpha x^2}{2}$ defines a continuous even function that is non-decreasing on $[0,\infty)$. So we can define $c=\lim_{|x|\to\infty}\e^{-\tilde V(x)}\in[0,\infty)$ and define $\rho$ by $\rho(\{\infty\})=c$ and 
\begin{align}
\rho([\kappa,\infty)) = \e^{-\tilde V(\kappa)}-c\quad\forall \kappa\in[0,\infty)
\end{align}
Define $\tilde V_\kappa(x)=\infty \1_{[\kappa,\infty)}(|x|)$ for $\kappa\in[0,\infty]$ (where $0\cdot\infty=0$). Then we have
\[
\e^{-\tilde V(x)}=\int_{[0,\infty]} \e^{-\tilde V_\kappa(x)}\,\rho(\d \kappa).
\]
and so, if we define $V_\kappa(x)=\frac{\alpha x^2}{2}+\tilde V_\kappa(x)=\begin{cases}\frac{\alpha x^2}{2}&|x|<\kappa\\+\infty &|x|\ge \kappa
\end{cases}$ then \eqref{e:def_V} follows.
\end{proof}

We are not aware of a place in the literature where a decomposition like this has appeared explicitly, but layer-cake decompositions of this type are of course well-known. In the context of interface models, the interactions $V_\kappa$ in this decomposition have been studied in \cite{MR3395146}. Additionally, in \cite{sellke} it was used that $\alpha$-monotone potentials can be written as the sum of a log-mixture of Gaussians and a monotone potential.

While the decomposition from Proposition \ref{p:trivialdecomp} is sufficient to prove some of the results in Sections \ref{sec:intro_interface} and Section \ref{applicstrength} (such as localisation/delocalisation and a Brascamp-Lieb inequality), for the Helffer-Sjöstrand representation and its corollaries such as the scaling limit, it is essential to have a decomposition into $C^2$-functions with upper bounds on their second derivatives.

\subsubsection*{Outline of the proof}

Unfortunately, the proof of Theorem \ref{t:decomp} is much more technical than the one in Proposition \ref{p:trivialdecomp}.
Our proof is constructive, and we will construct the functions $V_\kappa$ iteratively, ensuring that the remainder becomes more and more convex along the way. The key step (see Lemma~\ref{le:decomp_two} below) is to show that under the stated assumption an even potential $V$ which is convex on $(-t_0,t_0)$ can be decomposed into two potentials $\Vn$ and $W$
so that $\e^{-V} = \e^{-W}+\e^{-\Vn}$ with the following properties.
The potential $W$ is a convex function  and $\Vn$ is convex  on $(-t_1, t_1)$ for some $t_1>t_0$
and with lower bounded second derivative. The intuition why this is possible is that for two affine potentials $V$ and $\Vn$ with $V<\Vn$ and $V'>\Vn'$ (on some bounded interval)
the corresponding potential $W=-\log(\e^{-V}-\e^{-\Vn})$ is strictly convex (see Section~\ref{sec:sing_inter}). With this lemma established, we can then choose $W$ as our first $V_\kappa$, and iterate the procedure with $\Vn$ in place of $\V$. Of course, we will need to carefully keep track of how $V$ and its derivatives evolve under this procedure, to ensure that there is no blow-up and that after countably many iteration steps we have obtained a decomposition as desired.


\subsection{Applications of the decomposition to gradient interface models}\label{sec:intro_interface}

\subsubsection*{Our new results on gradient models}

Let us describe our new results on gradient interface models, made possible by our decomposition. As mentioned, our work is motivated by results for Ginzburg-Landau interfaces in \cite{MR1461951,MR1463032,MR1872740}, and results for Gaussian mixtures in \cite{MR2322690,MR2778801}. The combination of the two methods allows to deal with a much larger class of interaction potentials. Notably, our results hold true at all inverse temperature $\beta>0$.

Our main results for interface models concern the scaling limit of the model. As already observed in \cite{MR2322690,MR2778801}, in this generality it may happen that there are several different Gibbs measures with the same tilt, and that these different Gibbs measures have different scaling limits. In view of this, also here we cannot hope that any Gibbs measure has the same deterministic scaling limit. The best we can hope for is that any shift-invariant ergodic Gibbs measure with a given tilt scales to a multiple of the continuum GFF. Just like \cite{MR2778801}, we will have to restrict ourselves to zero tilt (see Section \ref{applicstrength} for an explanation why). We say that the measure is tempered if $\E_\mu(\eta(b)^2)<\infty$ for all $b\in\bzd$.

\begin{theorem}\label{t:scalinglimit}
Let $d\ge2$. Suppose that $V$ is a strongly log-concave log-mixture of quadratic growth, and let $\mu$ be a shift-invariant ergodic tempered gradient Gibbs measure for $V$ with zero tilt. Let $C_c^\infty(\R^d,\R^d)$ be the set of all infinitely differentiable functions $f:\mathbb{R}^d\to\mathbb{R}^d$ with compact support.

Then there exists a symmetric positive-definite matrix $q\in\R^{d\times d}$ with the following property: For any $f=(f_1,\ldots,f_d)\in C_c^\infty(\R^d,\R^d)$ consider the random variable 
\begin{equation}\label{e:observable}F_\eps(\eta):=\eps^{d/2}\sum_{x\in\zd}\sum_{i=1}^d f_i(\eps x)\eta((x,x+e_i)),
\end{equation}
where $\eta$ is distributed according to $\mu$ and $\eps>0$. Then for any $\lambda\in\R$ we have
\begin{equation}\label{e:scalinglimit}
\lim_{\eps\to0}\E_\mu(\exp(\lambda F_\eps(\eta)))=\exp\left(\frac{\lambda ^2}{2}\mathfrak{Q}_f\right),
\end{equation}
where 
\[\mathfrak{Q}_f:=(\nabla\cdot f,(-Q)^{-1}\nabla\cdot f)_{L^2(\R^d)}=\int_{\R^d}\sum_{i=1}^d\sum_{j=1}^d (\partial_i f_i)(y)((-Q)^{-1}\partial_j f_j)(y)\d y\]
and $Q:=\sum_{i,j=1}^d q_{ij}\frac{\partial^2}{\partial x_i\partial x_j}$. 
In particular, the random variables $F_\eps(\eta)$ converge in law, as $\eps\to0$, to a centred Gaussian random variable with variance $\mathfrak{Q}_f$.
\end{theorem}

Note that we stated Theorem \ref{t:scalinglimit} for $d\ge2$ only, because one step in our proof breaks down for $d=1$. However, the result is still true in $d=1$ (and much more can be shown). Namely in $d=1$ the gradient variables $\eta(x,x+e_1)$ on finite boxes are independent random variables conditioned so that their sum matches the boundary values, and so the result follows from a version of Donsker's theorem for random walk bridges as in \cite{MR238373}. See also \cite[Remarks 4.5 and 8.1]{MR2228384}.

To show that Theorem \ref{t:scalinglimit} is not vacuous, we need to prove the existence of tempered ergodic gradient Gibbs measures with zero tilt. Crucially, for the proof of Theorem \ref{t:scalinglimit} we also need to show that any shift-invariant ergodic gradient Gibbs measure $\mu$ with zero tilt has exponential moment bounds.

\begin{theorem}\label{t:existenceGGMs}
Set $d\ge 1$. Let $V$ be a strongly log-concave log-mixture of quadratic growth with $c_V=\alpha>0$.
\begin{itemize}
\item [(a)]
Then there exists at least one shift-invariant ergodic tempered gradient Gibbs measure $\mu$ for $V$ with zero tilt.

\item [(b)] 
Let $\mu$ be an arbitrary shift-invariant ergodic tempered gradient Gibbs measure for $V$ with zero tilt. Let $f\in C_c^\infty(\R^d,\R^d)$ and take $F_\eps$ as in \eqref{e:observable}.
Then $\mu$ satisfies for all $\lambda\in\R$, $0<\delta<\alpha/2$, $\eps>0$ and all $b\in\bzd$ 
\begin{equation}\label{e:Gaussianmoments}
\E_\mu\left(\exp((\alpha/2-\delta)\eta(b)^2)\right)\le\frac{\sqrt{\alpha}}{\sqrt{2\delta}} ~\mbox{and}~~\mathbb{E}_{\mu}(\exp(\lambda F_\eps(\eta)))\le\exp\bigg(\frac{C(d,f)}{\alpha}\lambda^2\bigg),
\end{equation}
where $C(d,f)>0$ does not depend on $\epsilon$. 
\end{itemize}
\end{theorem}
\begin{remark}
Part (a) holds more generally, and our proof will work for any $V$ that is even and satisfies a Brascamp-Lieb inequality in the spirit of \eqref{hargebl}. 
\end{remark}

In order to emphasise the meaning of these results, let us state explicitly what the combination of Theorems \ref{t:decomp}, \ref{t:scalinglimit} and \ref{t:existenceGGMs} implies for gradient interface models.
\begin{corollary}
\label{scalresalpha}
Let $d\ge 2$. Suppose that $V$ is an $\alpha$-monotone $C^2$ even potential such that there are constants $C,M>0$ such that $-M\le V''(s)\le C$. Then there is at least one shift-invariant ergodic tempered gradient Gibbs measure for that interaction with zero tilt, and each such measure scales to (possibly anisotropic) continuum Gaussian free field (in the sense that  \eqref{e:scalinglimit} holds).
\end{corollary}
\begin{remark}\label{r:BS}
This in particular recovers the results of \cite{MR2778801} for the case of Gaussian mixtures and of \cite{MR1461951} for a Ginzburg-Landau potential, but goes far beyond those results. Let us give two examples of classes of potentials that are covered by Corollary \ref{scalresalpha} but not by previous works.
\begin{itemize}
\item [(a)] The class of potentials that admit the representation \eqref{e:def_GM} is not closed under 
change of temperature, i.e., $\beta V$ may not be representable as in \eqref{e:def_GM} even if $V$ is.
Corollary \ref{scalresalpha} allows us to extend the results of \cite{MR2778801} from (the implicit) $\beta=1$ to the broader class of potentials given by
\begin{equation}\label{e:BK}
V(s)=-\beta\log\int\exp\bigg(-\frac{1}{2}\kappa s^2\bigg)\rho(\rmd\kappa),
\end{equation}
where $\rho$ is a positive measure with finite support in $(\alpha,\infty)$ for some $\alpha>0$, and $\beta>0$ is arbitrary.
\item [(b)] Another example is obtained by considering even log-mixtures of non-centred Gaussians. For example, potentials like  
\[
V(s)=-\beta\log(\e^{-\kappa_1s^2/2}+\e^{-\kappa_2(s-a)^2/2}+\e^{-\kappa_2(s+a)^2/2})
\]
are not covered by \cite{MR2778801}, while for suitable choices of the parameters (e.g. $a=1$, $\kappa_1=10,\kappa_2=1$) the potential $V$ will be $\alpha$-monotone and accordingly satisfy the assumptions of Corollary \ref{scalresalpha} for any $\beta>0$.

\end{itemize}
\end{remark}

Before we discuss potential extensions of our main results, let us describe in more detail relevant previous work on gradient interface models and their scaling limits.

\subsubsection*{State of the art on gradient models} 
There has been a lot of previous research on gradient interface models as in \eqref{eq:formal_grad_int_def} for various types of interactions $V$.
The case where the potential $V$ is uniformly strictly convex, that is, it satisfies $0<C_1\leq V''(s)\leq C_2$ for all $s\in \R$, has been extensively studied and it is well understood. One celebrated result proved in \cite{MR1463032} in this setting is that an ergodic gradient Gibbs measure is uniquely determined by the tilt $u\in \R^d$, where we define the tilt of a shift-invariant gradient measure $\mu$ by $\E_\mu(\nabla \phi(x))=u$
and where $\nabla\phi(x)\in\R^d$ denotes the discrete derivative, i.e., the vector with entries $\nabla_i\phi(x)=\phi(x+e_i)-\phi(x)=\eta((x, x+e_i))$.  Furthermore, the scaling limit of the model is the Gaussian Free Field (GFF) as shown by Naddaf and Spencer \cite{MR1461951} for zero tilt, and generalised to arbitrary tilt $u$ by Giacomin, Olla, and Spohn \cite{MR1872740}. Additionally, the covariance with respect to the unique measure $\mu$ has the same qualitative behaviour as the GFF; more precisely, there holds for $x, y\in \mathbb{Z}^d$, $1\leq i,j\leq d$
\begin{equation}\label{e:decay_covariances}
\left|\Cov_\mu(\nabla_i\phi(x), \nabla_j\phi(y))\right|\le \frac{C}{1\vee |x-y|^d},
\end{equation}
as proved in \cite{MR2198017}.

For non-uniformly strictly convex potentials much less is known and there are very few results, since most of the methods used for uniformly strictly convex potentials are no longer available. Regarding potentials where the uniform lower bound on $V''$ is relaxed, it has been shown in \cite{armstrongwu} for the continuous solid-on-solid model that scaling to the GFF still holds, by employing the convex decomposition of the potential $V$ from \cite{Brydges2012} (where some earlier results were proved for a class of subquadratic potentials). For degenerate convex $C^2$ potentials satisfying $0<\liminf_{s\to\infty}\frac{V''(s)}{|s|^{r-2}}\le\limsup_{s\to\infty}\frac{V''(s)}{|s|^{r-2}}<\infty$, where $r>2$, it was shown in \cite{DegenerateDario} and \cite{hydrodynDario} that localisation and hydrodynamic limit hold, by making use of techniques from quantitative homogenisation as in \cite{MR3648977}.
For strictly convex potentials where the uniform upper bound is relaxed (and in fact $V$ may take the value $+\infty$), in \cite{MR3395146} it was shown that the field in $d=2$ is still delocalised.

When $V$ is not convex, even less is known, and most of the available results are perturbative. For very high temperatures (i.e. very small $\beta$) it has been shown in \cite{Cotar2009, Cotar2012, MR3913274}, by an integration procedure designed to return to the uniformly strictly convex setting, that for all tilts $u$ and potentials of form $V=U+g$, with $U$ uniformly strictly convex and $g''\in L^q(\mathbb{R})$ for some $q\geq 1$ 
with sufficiently small norm,  the model behaves essentially as in the uniformly strictly convex case. That is, one can prove in this setting existence and uniqueness of gradient Gibbs measure for every tilt $u$, scaling limit to the GFF, and strict convexity of the surface tension. 

Meanwhile, for very low temperatures (corresponding to $\beta$ large) in \cite{adams2016strict, adams2019strict} a renormalisation group method in the spirit of \cite{MR3969983} was introduced to study  potentials $V$ which are subexponential, satisfy $V(s)\geq \eps|s|^2 $,$V(0)=0$, and fulfil some regularity assumptions, with equivalent conditions holding for the case of vector-valued fields. Via the methods developed there it was shown in \cite{hilger2016scaling,adams2019strict,adamskoller} that for very small tilts the surface tension is strictly convex and the scaling limit is the Gaussian Free Field. 

For moderate temperatures ($\beta=O(1)$), these perturbative techniques no longer apply. The main existing results in this setting are restricted to the class of potentials $V:\R\to\R$ that admit the log-mixture of Gaussians representation from \eqref{e:def_GM}. It was proved in \cite{MR2322690}, where this class was introduced, that for suitable $\rho$ and for the particular case of moderate temperature $\beta=1$, there exist at least two distinct ergodic gradient Gibbs measures with zero tilt (see also \cite{phasetransitionclass} for an alternative later proof of this result). This lack of uniqueness is unlike the high temperature case for the same class of potentials  \eqref{e:def_GM}, for which uniqueness of gradient Gibbs measures with tilt $u$ is preserved \cite{Cotar2012}. This establishes that new phenomena arise for non-convex potentials, showing that a uniqueness/non-uniqueness phase transition appears for some classes of gradient models with non-convex potentials. Note that the scaling limit to the Gaussian Free Field still holds at $\beta=1$ for the potentials in \eqref{e:def_GM}, as demonstrated in \cite{MR2778801} (cf. also Remark \ref{r:BS}). This was later extended in \cite{MR3982951} to a certain class of subquadratic non-convex potentials, again by means of the convex decomposition from \cite{Brydges2012}.
We also mention here that the uniqueness/non-uniqueness phase transition may disappear if suitable disorder is introduced in the potential, an example of this behaviour is given in \cite{aizenBC}.

Finally, recently Sellke \cite{sellke} has shown that for a wide class of potentials including all $\alpha$-monotone potentials, the interface is localised for $d\ge3$ and all temperatures. His proof technique, while also using decompositions of the potential, is very different from ours. Namely he writes $V$ as a sum of a log-mixture of Gaussians and a monotone remainder,  and uses the Gaussian correlation inequality to deal with that remainder. This works well to prove the upper bounds needed for localisation, but it cannot yield a scaling limit.

\subsubsection*{Outline of the proofs} 

As mentioned, our proof of Theorem \ref{t:scalinglimit} is based on a combination of the methods of \cite{MR2322690,MR2778801} (who were concerned with log-mixtures of Gaussians) and of \cite{MR1461951,MR1872740} (who were concerned with uniformly strictly convex potentials). In more detail, the basic idea of \cite{MR2778801} is to insert the decomposition \eqref{e:def_V} in each term of the Hamiltonian in \eqref{eq:formal_grad_int_def}. This allows to rewrite the probability measure (formally) as
\begin{equation}\label{e:formal_decomp}
\begin{split}&\frac{1}{Z}\Bigg(\prod_{\{x,y\}\in\ubzd}\int\exp\bigg(- V_{\kappa_{\{x,y\}}}(\phi(x)-\phi(y))\bigg)\rho(d\kappa_{\{x,y\}})\Bigg)\prod_{x\in\Z^d} \d\phi(x)\\
&=\int_{K^{\ubzd}}\Bigg(\frac{1}{Z_\kappa}\int_{\R^\zd}\exp\bigg(- \sum_{\{x,y\}\in\ubzd}V_{\kappa_{\{x,y\}}}(\phi(x)-\phi(y))\bigg)\prod_{x\in\Z^d} \d\phi(x)\Bigg)\frac{Z_\kappa}{Z}\prod_{\{x,y\}\in\ubzd}\rho(d\kappa_{\{x,y\}}),
\end{split}
\end{equation}
where we took $\beta=1$ for simplicity and where $\ubzd$ denotes the (undirected) nearest-neighbour edges in $\zd$. One can think of this measure as first sampling the variables $\kappa_{\{x,y\}}$ according to some probability measure 
\begin{equation}\label{e:mubar_formal}
\frac{Z_\kappa}{Z}\prod_{\{x,y\}\in\ubzd}\rho(d\kappa_{(x,y)})
\end{equation} on $K^{\ubzd}$, and then considering a disordered interface model with potential $V_{\kappa_{\{x,y\}}}$ along the edge $\{x,y\}$. This allows to represent the interface as an annealed mixture of interface models with strictly convex potentials, and one can hope to use homogenisation theory to find the scaling limit of the model.

Now \cite{MR2778801} considered mixtures of Gaussians, i.e. their $V_\kappa$ were quadratic, and so the representation \eqref{e:formal_decomp} lets them rewrite the law of the field as a mixture of Gaussian fields in random environment. Given the well-known connection between Gaussian fields and random walks, this let them represent the distribution of observables like \eqref{e:observable} via an annealed random walk in the static environment given by the $\kappa_{(x,y)}$, and standard qualitative annealed homogenisation results allowed them to obtain their scaling limit result.

In our case, we face the difficulty that our $V_\kappa$ are not quadratic. Yet in our situation they are $C^2$ with uniform upper and lower bounds on their second derivatives. That is, \eqref{e:formal_decomp} now becomes a representation of the law of the field as a mixture of Ginzburg-Landau interfaces in random environment.
While these are not Gaussian, they still have a random walk representation, the so-called Helffer-Sjöstrand representation \cite{MR1461951,MR1872740}. This is a representation of the law of observables of the field via an annealed random walk in time-dependent random environment, where the environment itself is given in terms of the natural Langevin dynamics associated with the field.
In our case this means that for each fixed $\kappa\in K^{\bzd}$ we have a representation in terms of an annealed random walk in random environment. Taking now the average over the $\kappa$ this leads to a representation of the entire interface in terms of an ``even more annealed'' random walk in time-dependent random environment, and using annealed homogenisation results for time-dependent random environment will allow us to prove Theorem \ref{t:scalinglimit}.

\medskip

Unfortunately, there are several challenges we need to face when implementing this strategy. First of all, we need to make \eqref{e:formal_decomp} rigorous by suitable finite volume approximations, and ensure that when starting from an ergodic Gibbs measure $\mu$ the marginal probability distribution on the $\kappa$ is ergodic as well. This is very similar to \cite{MR2778801}. Next we need to ensure that if $\mu$ has zero tilt, the same holds true for the disordered Gibbs measure. This part was easy in \cite{MR2778801} because a Gaussian measure is characterised by its mean and covariance. In our case that is not true anymore, so we need a more complicated argument. We will use a uniqueness result by the second author and Külske \cite{MR3383338}, which together with the observation that if the interface $\phi$ has zero tilt the same holds true for $-\phi$ implies that for a.e. $\kappa$ the disordered interface has zero tilt. All this will be done in Section \ref{extendedgradGibbs}.

At this point we have established a rigorous version of \eqref{e:formal_decomp}. The next step is to establish a Helffer-Sjöstrand representation for each fixed $\kappa$. Here we will follow the argument from \cite{MR1872740}. A main ingredient needed for this is that for a.e. $\kappa$ the natural Langevin dynamics is time-ergodic. In our setting this is a by-product of the uniqueness result \cite{MR3383338} (whose proof itself is based on coupling Langevin dynamics, as pioneered in \cite{MR1463032}), and will be shown in Section \ref{s:characterisation}. Taking the average of the Helffer-Sjöstrand representations over $\kappa$, we then obtain a representation of our field via a random walk in time-dependent random environment. The space-ergodicity of the $\kappa$ and the time-ergodicity of the dynamics will then imply that the random environment is space-time ergodic.
All this will be discussed in Section \ref{s:HS}.

Before we can prove the scaling limit, we will also need some a priori bounds on moments of the field. These will follow from so-called Brascamp-Lieb inequalities (to be discussed in more detail in Section \ref{applicstrength}). For each fixed $\kappa$, the disordered Ginzburg-Landau interface satisfies Brascamp-Lieb inequalities, and these carry over to our interface. We will collect all the inequalities we need in Section \ref{BLSection}. For many of our estimates in Section \ref{s:existence}, we will need the bounds to be uniform with respect to the finite volume. A useful statement in this direction will be Lemma \ref{l:expbrasc}, which will follow directly from \eqref{BLweakupperbound} via Green function bounds.

In Section \ref{s:applications} we will then establish Theorem \ref{t:scalinglimit} by using this random walk representation together with homogenisation results from the literature, namely the annealed heat kernel bounds of \cite{MR2198017} together with an invariance principle in probability. In our case the most convenient way to obtain such an invariance principle in probability will be to apply the quenched CLT from \cite{ACDSquenched}. As another application of the Helffer-Sjöstrand representation, we will also show in Theorem \ref{t:decaycovariances} the optimal decay of the covariances as in \eqref{e:decay_covariances}. 

It remains to prove Theorem \ref{t:existenceGGMs}. Here the basic strategy is based on the work \cite{MR2985173} of the second author and Külske, but we have the additional difficulty that we need not just existence of any tempered gradient Gibbs measure, but of one which is ergodic with zero tilt and satisfies the exponential bound \eqref{e:Gaussianmoments}. For that purpose we will want to use the Brascamp-Lieb inequality, but that is naturally a finite volume statement, so we cannot just quote the result from \cite{MR2985173}, but need to modify its proof while keeping track of the bounds like \eqref{e:Gaussianmoments} along the way. Furthermore, we need to show that every ergodic gradient Gibbs measure with zero tilt satisfies \eqref{e:Gaussianmoments} when $V$ is a strongly log-concave log-mixture with quadratic growth, for which purpose we will employ again the uniqueness result \cite{MR3383338}. This will be the topic of Section \ref{s:existence}.

\subsection{Further results and possible extensions}
\label{applicstrength}

\subsubsection*{Scaling limit in other settings}
Our scaling limit result in Theorem \ref{t:scalinglimit} is restricted to the infinite volume, zero tilt case. It is likely that the result also holds in finite volume, say on a torus or on a box with zero boundary data.  However, our method of proof does not easily extend to this case. This is for the same reason as in \cite{MR2778801}. Namely we do not have a good understanding of the properties of the random environment given by the $\kappa$. In infinite volume we know that this environment is ergodic and this is enough to use stochastic homogenisation, but in finite volume we have little control over the distribution of the $\kappa$ (and the fact that there might be several infinite volume Gibbs measures for the $\kappa$, cf. \cite{MR2322690,phasetransitionclass}, makes the study of the finite volume measure even more complicated).

The case of nonzero tilt seems also very hard. In that case one would need to understand the mean of the field conditioned on the $\kappa$. Already in the Gaussian case this was called a ``hard open problem'' in \cite{MR2778801}, and the non-Gaussian case here seems even harder. Additionally, for general non-zero tilt and non-convex $V$ it is not even clear whether there exists an ergodic gradient Gibbs measure with that tilt, and answering this question is another hard open problem in the area.

\subsubsection*{Extensions to other interface models}

While the class of $\alpha$-monotone potentials $V$ with $-M\le V''(s)\le C$ is already very general, it is natural to wonder what further applications our method might have. In this section we will discuss several possible extensions.

\paragraph{More degenerate potentials}
 First of all, the assumption $-M\le V''(s)$ is needed only for the decomposition in Theorem \ref{t:decomp}, and Theorems \ref{t:scalinglimit} and \ref{t:existenceGGMs} hold for general strongly log-concave log-mixtures of quadratic growth. However, the assumption $-M\le V''(s)$ is needed in Section \ref{s:decomp} to ensure that our recursive construction of the decomposition does not blow up in finite time, and so we do not currently see a way to avoid it. As our scaling limit result is in infinite volume only, there is also no easy way to remove this assumption by approximation. 

The assumption of $\alpha$-monotonicity is used in several places. It in particular ensures that the $V_\kappa''$ are bounded away from 0 uniformly. Recently in the study of Ginzburg-Landau models such uniform lower bounds have been relaxed \cite{hydrodynDario,DegenerateDario,armstrongwu}, and it might be possible to relax the assumptions in the uniqueness result of \cite{MR3383338} as well. 
However, a key difficulty is that we would need to have information on the tail behaviour of $\inf_s V_\kappa''(s)$ under the measure in \eqref{e:mubar_formal}. That marginals of the measure are strongly correlated, though, and already in the Gaussian case it is challenging to derive the necessary bounds (see e.g. \cite{MR3982951,armstrongwu}). So at present it seems difficult to relax the assumption of $\alpha$-monotonicity in any way.

Another exciting direction might be to relax the uniform upper bound $V''(s)\le C$. Theorem \ref{t:decomp} can be applied without such a bound, and it would yield a strongly log-concave log-mixture of $\mathfrak{g}$-controlled growth for some $\mathfrak{g}$ related to the growth of $V''$. Several parts of the proof of Theorem \ref{t:scalinglimit} would still apply assuming only such a bound on the strongly log-concave log-mixture: The Brascamp-Lieb inequalities of Section \ref{BLSection} do not require any upper bound on $V_\kappa''$ at all. Regarding the homogenisation results in Section 8, we could replace the results from \cite{MR2198017,ACDSquenched} with the recent bounds from \cite{deuschel2023gradient} that only require a finite first moment of the environment. This could be ensured by a very weak assumption on $\mathfrak{g}$, for example that $\mathfrak{g}$ grows at most exponentially fast. However, there is a major obstacle when trying to actually extend Theorem \ref{t:scalinglimit}. Namely both the uniqueness result in \cite{MR3383338} and the time ergodicity of the dynamics in Lemma \ref{l:langevin_time_mixing} are based on the coupling argument in \cite{MR1463032}. For such a coupling argument, it is essential to have strong solutions of the corresponding Langevin dynamics in infinite volume. When $V_\kappa''$ is not bounded above, the corresponding SDEs no longer have a Lipschitz-continuous drift. These SDEs still have weak solutions (as follows from \cite[Theorem 4.4]{MR1432591}), but proving existence of strong solutions seems very difficult. In the absence of that, it is unclear how to obtain time-ergodicity of the corresponding dynamics. Via \cite[Theorem 5.15]{MR1432591} it would follow if we knew that for a.e. $\kappa$ the disordered Gibbs measure is extremal, but extremality of Gibbs measures is a notoriously hard question as well (and as we need extremality for fixed $\kappa$ we are not in any translation-invariant setting, so the abstract theory of \cite{MR2251117} does not apply).

While we currently do not know how to resolve this problem, the discussion above served as our motivation to point out that if we have a strongly log-concave log-mixture of quadratic growth then the corresponding disordered Gibbs measures are in fact a.s. extremal (see Theorem \ref{th:extremality}). 

\paragraph{Perturbations of $\alpha$-monotone potentials}
As already mentioned in Section \ref{sec:intro_interface}, \cite{Cotar2009, Cotar2012, MR3913274} studied interface models whose interaction is the sum of a uniformly convex term and a small perturbation. The key idea was to integrate out the field on the odd sites, say, and to obtain an effective Ginzburg-Landau model on the even sites.

Our proof of Theorem \ref{t:scalinglimit} relies on rewriting the interface as an annealed mixture of Ginzburg-Landau models in random environment. So it is natural to wonder whether one can use the method of \cite{Cotar2009, Cotar2012, MR3913274} also in our setting, and thereby extend some of our results from $\alpha$-monotone potentials to small perturbations thereof.  If we can guarantee that  the negative part of $V''$ is integrable (but potentially unbounded)
it should be possible to extend the perturbative results in \cite{Cotar2009,Cotar2012}. To see how this could work, we would attempt to write $V=V_0+g$, with $V_0$ being $\alpha$-monotone and $g$ chosen with second derivative having a small $L^p$ norm. Then we would be able to write $V_0$ as a strongly log-concave log-mixture for all $0<\zeta<1$, and 
\[
\e^{-V(s)}=\sum_i \e^{-V_i(s)-g(s)},
\]
which gives rise to a mixture of gradient Gibbs measures, each with potential $V_i+g$. Crucially, if $V$ is of quadratic growth, each resulting gradient Gibbs measure in the mixture would satisfy the odd/even approach from \cite{Cotar2012}, allowing us to control the perturbation $g$ (at least at high/moderate temperature $\beta^{-1}$). The key property we would use is that by our decomposition, $V_i''\ge (1-\zeta)\alpha>0$ uniformly in all $i$; note that without this uniform lower bound assumption on all $V_i$, the argument breaks down.

\paragraph{Adding self-potentials}
We observe here that for some of our results for gradient Gibbs measures, we could consider the more general class of measures of the form
\begin{equation}\label{e:Hamiltonianselfpotential}
\frac{\exp\left(-\frac{\beta}{2}\sum_{x, y: |x-y|=1} V(\phi(x)-\phi(y))-\sum_x U(\phi(x))\right)}{Z}\prod_{x\in\Z^d} \d\phi(x),
\end{equation}
where $U:\R\to\R$ is called a self-potential. 

For instance, if  $U, V$ are $C(\mathbb{R})$ even functions, with one of $U$ or $V$ being $\alpha$-monotone and the other one being monotone (since Proposition \ref{p:trivialdecomp} naturally extends to monotone potentials), we can apply our techniques from Section \ref{s:existence} to show existence of ergodic Gibbs measures in this setting. 

As another application, we can use a result of \cite{AntAnt} to obtain exponential decay of correlations for a class of non-convex potentials $U,V$. Namely, it is proved in \cite[Theorem 1]{AntAnt}, that if $U, V$ are $C^2$ such that $U''\ge\alpha>0$ and is of no more than exponential growth, and $V$ is convex with the ratio 
\begin{equation}\label{e:antant}
\sup_{\phi(x), 
    \phi(y)\in\mathbb{R}}|V''(\phi(x)-\phi(y))|/\sqrt{U''(\phi(x))}\sqrt{U''(\phi(y))}<\infty,
\end{equation}then the correlation between $\phi(x)$ and $\phi(y)$ decays exponentially in $|x-y|$.
The result is stated and proved only in infinite volume, but the method therein works also in finite volume.
One can now apply the decomposition in Theorem \ref{t:decomp} to $V$ and possibly to $U$ as well, and then apply the aforementioned result to the individual potentials in the decomposition. Since by Theorem \ref{t:decomp}, we can choose the strongly log-concave log-mixture to be of $\mathfrak{g}$-controlled growth, it is sufficient for exponential decay of correlations of the fields in \eqref{e:Hamiltonianselfpotential}, both in finite and infinite volume, to have $U,V\in C^2(\R)$ with $U$ being $\alpha$-monotone satisfying the bound \eqref{e:antant}, and either both $U,V$ of quadratic growth, or $U$ convex of no more than exponential growth and $V$ $\alpha$-monotone. An example where this applies is the lattice $\Phi^4_d$ model with a sufficiently large mass (so that the sum of the double-well term and the mass term becomes monotone).

Finally, as pointed out in \cite[Remark 2.3]{MR1759509}, one can extend the Helffer-Sj\"ostrand representation to Ginzburg-Landau fields with a convex $C^2$ self potential. So if one combines this with our decomposition for both $U$ and $V$, one obtains a Helffer-Sj\"ostrand representation for measures of the form \eqref{e:Hamiltonianselfpotential}, provided that both $U,V$ are $\alpha$-monotone with bounded second derivative.

\subsubsection*{Poincar\'e and Brascamp-Lieb inequalities for general $\alpha$-monotone potentials}

A classical result for strongly log-concave distributions is the Brascamp-Lieb moment inequality, which states that for a convex potential $\mathcal{X}:\mathbb{R}^d\to\mathbb{R}$ with $D^2\mathcal{X}\geq \mathrm{Id}$,
the bound 
\begin{align}
\label{hargebl}
    \int_{\mathbb{R}^d} F(v\cdot \left(x-m_{\mathcal{X}}\right))\,\d\mu_{\mathcal{X}} ( x)\leq \int_{\mathbb{R}^d} F(x\cdot v)\,\d\Gamma_d(x), ~~\mbox{where}~~m_{\mathcal{X}}:=\int_{\mathbb{R}^d} x\,\d\mu_{\mathcal{X}} ( x),
\end{align}
holds for all lower bounded convex functions $F:\mathbb{R}\to\mathbb{R}$,
where $\Gamma_d$ denotes the law of a standard Gaussian in $\mathbb{R}^d$ and $\mu_{\mathcal{X}}$ is the probability measure with density $\e^{-\mathcal{X}(x)}/\int_{\mathbb{R}^d}\e^{-\mathcal{X}(x)}\rmd x$. The Brascamp-Lieb inequality was introduced by Brascamp and Lieb in \cite[Theorem 5.1]{MR0450480} for $F(s)=|s|^p, p\ge 1$, and it was later generalised to the form in \eqref{hargebl} and $F$ convex and bounded from below in Theorem 1.1 of Harg\'e \cite{Harge} and in Corollaries 6 and 7 from Section 5 of Caffarelli \cite{MR1800860}. 

An immediate interesting consequence of Proposition \ref{p:trivialdecomp} is Theorem \ref{comparenon-convex} below, stating a generalisation of the Brascamp-Lieb inequality to a certain class of non-log-concave measures on $\R^d$, $d\ge 1$, and a dimension-free Poincar\'e inequality for this class of non-log-concave measures. The proof of Theorem \ref{comparenon-convex} will be given in Section \ref{BLSection} below. 

For (strongly) log-concave measures, such results are well-known and there is a vast literature available due to their versatility in probability theory and their connections to many other areas such as optimal transport and statistics (see, for example \cite{BobLed}, or \cite{Ledoux}, for some of the probabilistic literature, and \cite{SaumWell} for a survey of log-concavity in statistics). For non-log-concave measures, there are fewer results and the constants obtained are in general not optimal. For some of these results, we refer the reader to \cite[ Corollary 1.6]{bak}, to \cite[Theorems 1.1 and 1.3]{fig} and \cite[Theorem 1]{fathi} (all three of which use Caffarelli-type optimal transport arguments as in \cite{MR1800860}, and imply Poincar\'e and log-Sobolev inequalities for the respective measures), to \cite[Corollary 2.1]{CGW}, to \cite{chewi}, and to the references therein. We also observe that the results from \cite{bak}, \cite{CGW}, \cite{fig} and \cite{fathi} would produce dimension-dependent Poincar\'e inequalities if applied in the corresponding statistical mechanics non-convex setting considered in Theorem \ref{genposbeta} and Theorem \ref{genposbetapoincare} below.

While Theorem \ref{comparenon-convex} below seems, to the best of our knowledge, to be new in this generality, we note that a similar  statement to the one in Theorem \ref{comparenon-convex}  (a)(i) and (b)(i) appeared in \cite[Theorem 3.6 and Corollary 4.1]{MR2286104} under more general conditions on $\mathcal{Y}$, but with more restrictions on $F$. Additionally, due to the radial symmetry assumption therein, the result does not extend to gradient-type models. A related result also appeared in \cite[Proposition A.2]{Hariya}, for a more restrictive class of functions $\mathcal{Y}$ than ours and for convex $F$. 

Our $\alpha$-monotonicity assumption on $\mathcal{Y}$ from Theorem \ref{comparenon-convex} corresponds to condition (2.3) from Corollary 2.1 in \cite{CGW}, which our $\alpha$-monotonicity assumption implies. Under that condition, \cite{CGW} obtains a $W_2H$ inequality and as corollaries $L^1$- and $L^2$-Poincar\'e inequalities, yet with dimension-dependent constants. In contrast, the constants in Theorem \ref{comparenon-convex} below are dimension-free.

We observe here also that the proof idea of \cite[Theorem 17]{BartKlar} applies to part (a)(ii), via the decomposition in Proposition \ref{p:trivialdecomp}.

\begin{theorem}
\label{comparenon-convex}
Fix $d\ge 1$ and $\alpha>0$. Let $\mathcal{Y}:\R^d\rightarrow\R$ be a continuous function such that for all $i=1, \ldots,d,$ there holds $\mathcal{Y}(x_1,\ldots, x_i,\ldots,x_d)=\mathcal{Y}(x_1,\ldots,-x_i,\ldots,x_d)~~\mbox{and}~~x_i\to\mathcal{Y}(x_1,\ldots, x_i, \ldots, x_d)-\alpha x_i^2/2$ is a monotone function on $(0,\infty)$.
In other words, $\mathcal{Y}$ is even in each coordinate and $\alpha$-monotone in each coordinate.  Let $\mathcal{X}:\R^d\rightarrow\R$ (potentially $\equiv 0$) be an even, convex function. Set $\mathcal{V}=\e^{-\mathcal{X}-\mathcal{Y}}$ and let $\rmd\mu_{\mathcal{X+Y}}(x):=\mathcal{V}(x)\rmd x/ \int_{\mathbb{R}^d}\mathcal{V}(x)\rmd x$.  
\begin{itemize} 
\item [(a)] Take $\mathcal{X}\equiv 0$ and let $\mathcal{Y}$ be as above. 
\begin{itemize}
\item [(i)] Then there holds for all $v\in {\mathbb R}^d$ and for all convex functions 
$F:\mathbb{R}\to\mathbb{R}$, bounded below, that 
\begin{equation}
\label{BLRn}
\int_{\mathbb{R}^d}F\left(v \cdot x\right)\rmd\mu_{\mathcal{Y}}(x)\leq \int_{\mathbb{R}^d}F\left(\frac{v}{\sqrt{\alpha}} \cdot x\right)\rmd\Gamma_d(x).
\end{equation}
\item [(ii)] Then there exist constants $0<C^1_{PI}\le ((\pi/2\alpha))^{1/2}, 0<C^2_{PI}\le \alpha^{-1}$, such that 
for all smooth functions $F:\R^d\rightarrow\R$ which are odd in at least one coordinate (and which in particular satisfy $\int_{R^d}F(x)\rmd\mu_{\mathcal{Y}}(x)=0$), there holds 
\begin{equation}
\label{poincare1} 
\int_{\R^d} |F|\rmd\mu_{\mathcal{Y}}\le C^1_{PI}\int_{\R^d} |\nabla F|\rmd\mu_{\mathcal{Y}}~~~\mbox{and}~~~\int_{\R^d} F^2\rmd\mu_{\mathcal{Y}}\le C^2_{PI}\int_{\R^d} |\nabla F|^2\rmd\mu_{\mathcal{Y}}.
\end{equation}
\end{itemize}
\item [(b)] Let $\mathcal{X}$ be as above and take $\mathcal {Y}(x):=\mathcal{W}(|x|)$, where $\mathcal{W}:[0, \infty)\rightarrow\R$ is continuous and $\alpha$-monotone.
\begin{itemize}
\item [(i)] Then there holds for all $v\in {\mathbb R}^d$ and for all convex functions 
$F:\mathbb{R}\to\mathbb{R}$, bounded below, that 
\begin{equation}
\label{BLRnb}
\int_{\mathbb{R}^d}F\left(v \cdot x\right)\rmd\mu_{\mathcal{X+Y}}(x)\leq \int_{\mathbb{R}^d}F\left(\frac{v}{\sqrt{\alpha}} \cdot x\right)\rmd\Gamma_d(x).
\end{equation}
\item [(ii)] 
Then there exist constants $0<C^1_{PI}\le (\pi/2\alpha)^{1/2}, 0<C^2_{PI}\le \alpha^{-1}$, such that 
for all smooth odd functions $F:\R^d\rightarrow\R$ (which in particular satisfy $\int_{R^d}F(x)\rmd\mu_{\mathcal{X+Y}}(x)=0$), there holds 
\begin{equation}
\label{poincare2} 
\int_{\R^d} |F|\rmd\mu_{\mathcal{X+Y}}\le C^1_{PI}\int_{\R^d} |\nabla F|\rmd\mu_{\mathcal{X+Y}}~~~\mbox{and}~~~\int_{\R^d} F^2\rmd\mu_{\mathcal{X+Y}}\le C^2_{PI}\int_{\R^d} |\nabla F|^2\rmd\mu_{\mathcal{X+Y}}.
\end{equation}
\end{itemize}
\end{itemize}
\end{theorem}
We conclude this section with a remark.
\begin{remark}
\label{poinkls}
Theorem \ref{comparenon-convex} can be extended via optimal transport tools to more general functions than $\alpha$-monotone, such as of form ${\mathcal W}+g$, where $\mathcal{W}$ is $\alpha$-monotone, $g$ is even and compactly supported on some set $K\subset\R^d$ and  $\nabla^2g\ge -C_g \mbox{Id}$, where $C_g>0$. To show how it works for the  Poincar\'e inequality, we write
\[
\e^{-\mathcal{X}(x)-\mathcal{Y}(x)}=\int \exp(-\mathcal{X}(x)-\mathcal{W}_\kappa(|x|)-g(x))\rho(\d\kappa)=:\int \exp(-\mathcal{X}(x)-\mathcal{Y}_\kappa(x))\rho(\d\kappa),
\]
where $\mathcal{W}''_\kappa\ge \alpha$ for all $\kappa$. We now apply Theorem 1.1 from \cite{fig} to each $\rmd\mu_{\kappa, g}(x):= \exp(-\mathcal{X}(x)-\mathcal{Y}_\kappa(x))/\int \exp(-\mathcal{X}(x)-\mathcal{Y}_\kappa(x))\rmd x$. Thus we write $\mu_{\kappa, g}=\Gamma_d\circ T_\kappa^{-1}$, where $T_\kappa$ is $C(K, \alpha, C_g)$-Lipschitz for some constant $C(K, \alpha, C_g)>0$.  Therefore the derivative of $T_\kappa$ is bounded, with $|D T_\kappa|_{\mbox{op}}\le C(K, \alpha, C_g)$, where $|\cdot|_{\mbox{op}}$ is the operator norm. Thus, we get for all $F$ odd
\[
\mathbb{E}_{\mu_\kappa, g}(F^2)=\Var_{\Gamma_d}(F\circ T_\kappa)\le \mathbb{E}_{\Gamma_d}\left(|\nabla(F\circ T_\kappa)|^2\right)\le \int \!\! |DT_\kappa|^2_{\text{op}}(|\nabla F|\circ T_\kappa)^2\rmd\Gamma_d\le C^2(K, \alpha, C_g)\int\!\!|\nabla F|^2\rmd\mu_{\kappa, g},
\]
where for the first inequality we applied the Poincar\'e inequality to the Gaussian measure $\Gamma_d$, and where $|\cdot|_{\mbox{op}}$ is the operator norm.  This concludes the argument. 
\end{remark}

\section{Decompositions of monotone potentials}\label{s:decomp}

The goal of this section is to show that a wide class of potentials can be written as strongly log-concave log-mixtures, and in particular to prove Theorem~\ref{t:decomp}. 
We will first state a slightly more concrete version 
of Theorem~\ref{t:decomp} which requires the following notation.
For a function $f:I\to \R$ we will use the standard notation $(f)_+=\max(f, 0)$ and $(f)_- = -\min(f,0)$   for the positive and negative parts and we will frequently use that $f=f_+-f_-$ and that $f_+$ and $f_-$ are continuous if $f$ is continuous.

\begin{theorem}\label{th:decomp}
Let $V$ be an even and $\alpha$-monotone $C^2$-potential. Suppose that $V''\geq -M$
for some $M>0$. Then, there exists for each $0<\zeta<1$ a collection of at most countably many  even $C^2$-functions $V_i$ such that 
\begin{equation}\label{eq:decomp_V}
\e^{-V(t)}=\sum_i \e^{-V_i(t)}
\end{equation}
and the $V_i$ are strictly convex 
and there is a constant $C>0$ (depending on $\alpha$,
$\zeta$, $M$, and 
$\inf \{t>0: V''(t)-(1-\zeta)\alpha <0\}$) such that
\begin{equation}\label{eq:estderV}
(1-\zeta)\alpha\leq V_i''(t)\leq C(V'')_+(t)+C.
\end{equation}
\end{theorem}

\begin{figure}
    \centering
    \includegraphics[width=1.\linewidth]{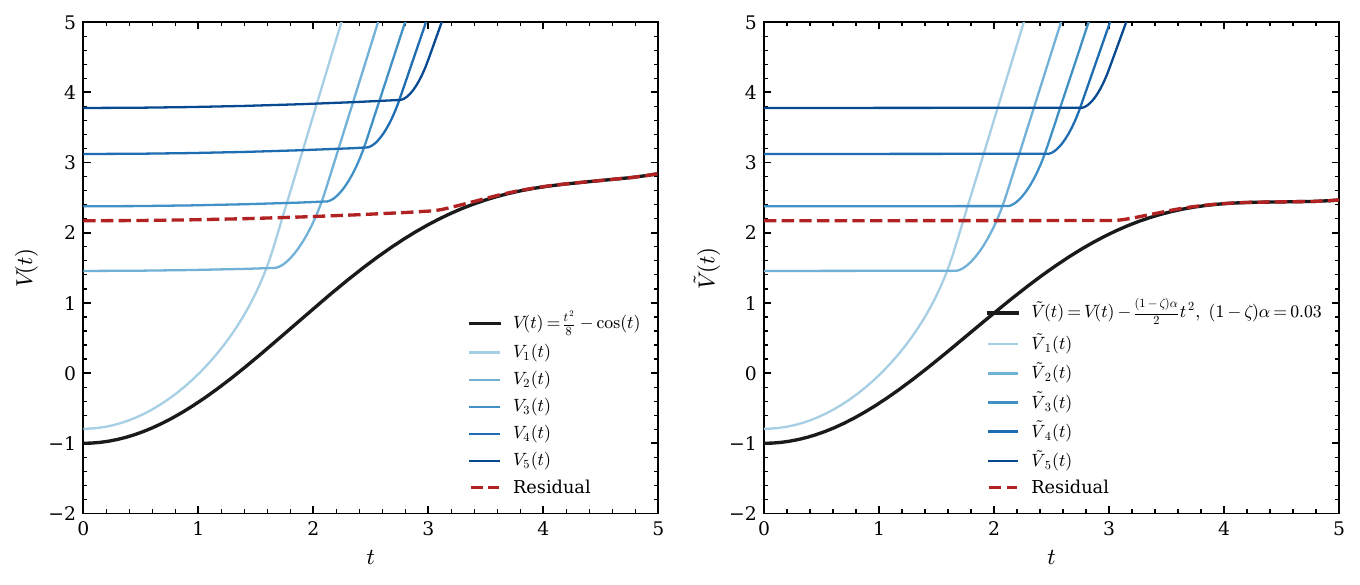}
    \caption{(left) We consider 
    the potential $V(t)=t^2/8-\cos(t)$ introduced in Example~\ref{ex:function} and sketch  the first strictly convex potentials $V_i$ for $1\leq i\leq 5$ of the decomposition constructed in Theorem~\ref{th:decomp}. 
    It can be seen that the residual (dashed red) is convex
    on a larger domain than the initial potential (black).
    (right)
    Decomposition of the potential
    $\tilde{V}(t)=V(t)-(1-\zeta)\alpha t^2/2$  
    (with $(1-\zeta)\alpha=0.03$)  into convex potentials $\tilde{V}_i$ as constructed in Theorem~\ref{th:decom_general}  (cf. the proof of Theorem \ref{th:decomp} for the introduction of $\tilde{V}$).
    }
    \label{fig:1}
\end{figure}
We note in passing that the theorem is trivial if the infimum is $\infty$ because we can then use a decomposition with the single term $V_1=V$.
Let us add a few remarks regarding this theorem.
\begin{remark}
\label{decompgenrem}
    \begin{enumerate}
\item Monotonicity of $V$ on $[0,\infty)$ is clearly necessary for such a decomposition to exist because all the convex and even potentials $V_i$ are non-decreasing on $(0,\infty)$.
\item Recall that in addition, $\alpha$-monotonicity is necessary if we want the $V_i$ to be strictly convex with $V_i''\geq \alpha$. 
Indeed, note that $V_i''(s)\geq \alpha$ implies that $V_i'(s)\geq \alpha s$ for $s\geq 0$. 
Then we find by differentiating \eqref{eq:decomp_V}
 that (to justify the exchange of sum and derivative we use difference quotients and  $V_i'(t)\geq 0$ for $t\geq 0$)
\begin{align}
-V'(t)\e^{-V(t)}
\leq \lim_{N\to \infty}\sum_{i=1}^N -V_i'(t)\e^{-V_i(t)}
\leq \lim_{N\to \infty}\sum_{i=1}^N -\alpha t \e^{-V_i(t)}=-\alpha t\e^{-V(t)}.
\end{align}
Thus we find $V'(t)\geq \alpha t$ which implies $\alpha$-monotonicity. Thus,  up to the factor $1-\zeta$ in Theorem~\ref{th:decomp}
 the $\alpha$-monotonicity is necessary.
\item 
The result in Theorem~\ref{th:decom_general} below shows that  a similar decomposition into convex potentials also exists for general strictly monotone even potentials such that $-V''(t)$ and $-V''(t)/V'(t)$  are upper bounded for $t> 0$. 
However, the bounds for $V_i''$ are weaker in this case. 
\item If $V$ is a potential such that $\{t\in \R\, :\, V''(t)<0\}$ is bounded, then
the decomposition contains only a finite number of terms.

\item The decomposition is essentially constructed 
for $t\geq 0$ and is then extended to even functions. One could similarly construct decompositions for not necessarily even functions where the only challenge is to ensure smoothness at the origin. However, we do not discuss this further as we are not aware of any applications. Indeed, we crucially rely on even potentials to control the mean of the gradient measures. 
\end{enumerate}
\end{remark}

We will next deduce Theorem \ref{th:decomp}
from a slightly more general and  technical version of it.
\begin{theorem}\label{th:decom_general}
Let $V$ be an even $C^2$ function which is  convex in a neighbourhood $(-t_0,t_0)$ of $0$ and
assume $V'(t_0)\geq \eps$ for some $0<\eps\leq 1/2$
and $V'(t)>0$ for $t>0$. Moreover we assume that 
$V(t)\to\infty$ as $t\to\infty$ and that there is an $M\geq 1$ such that for $t>0$
\begin{align}
    -V''(t)\leq M\min(1,V'(t)).
\end{align}
Then there is a potentially infinite sequence of convex and even $C^2$ potentials $V_i$ for $i=1,\ldots$ such that 
\begin{align}
    \e^{-V}=\sum_i \e^{-V_i}
\end{align}
and they satisfy the following upper bound on their second derivative
\begin{align}\label{eq:Vpp_final1}
    V_i''(t)\leq \frac{C_2M}{\eps}\e^{C_1Mt}(V'')_+(t)
    +\frac{C_3M^2}{\eps}\e^{C_1Mt}
\end{align}
where $C_1,C_2,C_3$ are absolute positive constants.
If, moreover, 
\begin{align}\label{eq:as_eps_lower}
    V'(t)\geq \eps
\end{align}
for all $t\geq t_0$
then such a decomposition exists with the stronger upper bound
\begin{align}\label{eq:Vpp_final2}
    V_i''(t)\leq C_2'\frac{M}{\eps}(V'')_+(t)+C_3'\frac{M^2}{\eps}
\end{align}
where $C_2'$ and $C_3'$ are again absolute constants.
\end{theorem}

Our proof of Theorem \ref{th:decom_general} is constructive.
The general idea is to
find for a given potential $\V$ 
two other potentials $\Vn$ and $\W$
such that $\e^{-\V}=\e^{-\Vn}+\e^{-\W}$. Here $\W$ will be convex and can be added to the collection of $V_i$. Meanwhile $\Vn$ will have slightly better properties than $\V$, and in particular be convex in a larger neighbourhood of $0$ than $\V$. If we combine this with some bounds on the behaviour of $\Vn(s)$ for large $s$, we can hope to iterate the procedure, with $\Vn$ taking the place of $\V$. 

We will give details of this in the next two subsections. For now, let us show how Theorem \ref{th:decomp} (and thus Theorem~\ref{t:decomp}) follows from Theorem \ref{th:decom_general}.

\begin{proof}[Proof of Theorem~\ref{th:decomp}]
    First, we note that for an $\alpha$-monotone 
    potential $V$ the potential $\tilde V$ defined by $\tilde V(t)=V(t)-(1-\zeta)\alpha t^2/2$
    is $\zeta\alpha$-monotone, i.e., $\tilde V'(s)\geq \zeta\alpha s$. Since $\tilde{V}$ is even
    we know that $\tilde{V}'(0)=0$
    and therefore $\tilde{V}''(0)\geq \zeta \alpha$.
    Using continuity we conclude that $\tilde V$ is convex in a neighbourhood $(-t_0,t_0)$ of $0$
    and we can set 
    \begin{align}
    t_0=\inf \{t>0: V''(t)-(1-\zeta)\alpha <0\}.
    \end{align}
    If $t_0=\infty$ the decomposition with the single term $V$ can be used so we assume $t_0<\infty$ from now on.
    We  define $\eps =\min(1/2,\zeta\alpha t_0)$.
    Then,  $\tilde V'(t)\geq \zeta\alpha t\geq \zeta\alpha t_0\geq  \eps$ for $t\geq t_0$. 
    We also conclude that $\min(\tilde V'(t), 1)\geq \eps$ is lower bounded for $t\geq t_0$. Since $V''$ is lower bounded the same is true for $\tilde V''$. Hence, there is a constant $M$ 
    (depending on $\inf V''$, $\alpha$, $\zeta$, and $t_0$)
    such that 
    $-\tilde V''(t)\leq M\min(\tilde V'(t), 1)$.
    Now we apply Theorem~\ref{th:decom_general}
    and obtain a decomposition
    \begin{equation}\label{e:dec_Vtilde}
        \e^{-\tilde V}=\sum_i \e^{-\tilde V_i}
    \end{equation}
    where $\tilde V_i$ are $C^2$ and convex, and
    satisfy
    \begin{equation}\label{eq:estdertildeV}
        \tilde V_i''(t)\leq C(\tilde V'')_+(t)+C
    \end{equation}
    for some constant $C>0$ that depends on 
    $\eps$ and $M$ and thus $t_0$, $\alpha$, $\zeta$, and $\inf V''$.

    Recalling that $\tilde V(t)=V(t)-(1-\zeta)\alpha t^2/2$, \eqref{e:dec_Vtilde} implies that
    \begin{align}
        \e^{-V(t)}=\sum_i \exp\left(-\left(\tilde V_i(t)+\frac{(1-\zeta)\alpha t^2}{2}\right)\right).
    \end{align}
   Therefore, we define $V_i$ by $V_i(t)=\tilde V_i(t)+(1-\zeta)\frac{\alpha t^2}{2}$. This ensures that \eqref{eq:decomp_V} holds. Moreover, since the potentials $\tilde V_i$ are convex, we have $V_i''\geq (1-\zeta)\alpha$, and the upper bound on $V_i''$ in \eqref{eq:estderV} follows immediately from \eqref{eq:estdertildeV}.
\end{proof}

\subsection{A single iteration step}\label{sec:sing_inter}
As mentioned, we will prove Theorem \ref{th:decom_general} by an iterative construction. In this subsection we will describe one such iteration step. That is, given $\V$ we want to construct a function $\W$ such that $\Vn$, defined implicitly by $\e^{-\V}=\e^{-\Vn}+\e^{-\W}$ is convex on a larger interval around 0 than $\V$, while staying well-behaved outside that interval (in a sense to be made precise). It is convenient  to restrict potentials $V$ to $[0,\infty)$ here.
This is sufficient because those can be extended uniquely to an even function on $\R$. Therefore, we only consider potentials $V:[0,\infty)\to  \R$.

For the construction of the decomposition into two potentials it is helpful to think about this slightly differently, and to consider $\W$ as a function of $\V$ and $\Vn$. We will need to choose $\Vn$ such that in particular the function $\W$ will be convex.

We will assume right away that $\Vn>\V$ pointwise, and define
\begin{align}\label{eq:def_Delta}
    \Delta(s)=\Vn(s)-\V(s)>0.
\end{align} 
Next we define $\W$ such that 
\[
\e^{-\W}=\e^{-\V}-\e^{-\Vn}
\]
holds. Explicitly,
\begin{equation}\label{eq:decomp_vnew_delta}
\begin{split}
    \W(\V,\Vn)(s)&= - \log\left(\e^{-\V(s)}-\e^{-\Vn(s)}\right)=\V(s) - \log\left(1-\e^{\V(s)-\Vn(s)}\right)\\
&=\V(s)-\log\left(1-\e^{-\Delta(s)}\right).
\end{split}
\end{equation}
Clearly $\W$ is $C^2$ if $\V$ and $\Vn$ are twice continuously differentiable, and we can find the following expressions for its derivatives 
\begin{align}\label{eq:Wp}
    W'(s) &= V'(s) + \frac{-\Delta'(s)}{\e^{\Delta(s)}-1}\approx V'(s) - \frac{\Delta'(s)}{\Delta(s)},
    \\
    \label{eq:Wpp}
    W''(s)&= V''(s)-\frac{\Delta''(s)}{\e^{\Delta(s)}-1} + \frac{(\Delta')^2(s)}{(\e^{\Delta(s)}-1)(1-\e^{-\Delta(s)})}
    \approx 
    V''(s)-\frac{\Delta''(s)}{\Delta(s)} + \frac{(\Delta')^2(s)}{\Delta^2(s)},
\end{align}
where the approximations hold for small $\Delta$.

\begin{figure}
    \centering
    \includegraphics[width=0.5\linewidth]{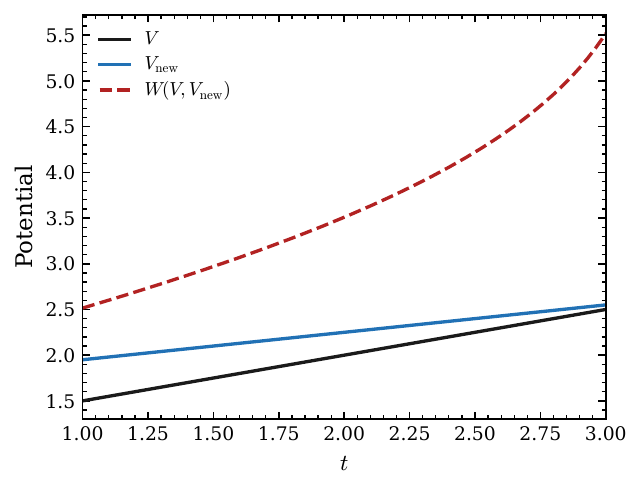}
    \caption{Sketch of affine $V$, $\Vn$
    with the resulting strictly convex $W(V,\Vn)$.}
    \label{fig:elementary}
\end{figure}
Note that if $V''=\Vn''$ (so $\Delta''=0$) 
we find $W''\approx V''+\Delta'^2/\Delta^2$, i.e., $W$ is more convex than $V$ (for $\Delta'\neq 0$) and $\Vn$
is  as convex as $V$. 
To illustrate this, we show in Figure~\ref{fig:elementary} the function $W$
for the case where $V$ and $\Vn$ are affine on an interval where we note that $W''$
increases as $V$ and $\Vn$ approach each other. Qualitatively, the behaviour is similar for general $V$ and $\Vn$.

We now provide a high-level overview of our 
construction, which is a bit technical and then implemented in the rest of the section.
We will assume that there is $t_0>0$ such that  $V$ is convex on $(0,t_0)$ 
and $M>0$ such that $V''\geq -M$
and want to show that we can find $\Vn$ and $W(V,\Vn)$ such that:
\begin{itemize}
    \item The potential $W$ is globally convex.
    \item The potential $\Vn$ is convex
    on $(0,t_1)$ for some $t_1>t_0$ and satisfies $\Vn''(t)\geq V''(t)$ for $t\geq t_0$. 
\end{itemize}
To achieve this, we set $\Vn(t)=\Vn(t_0)>V(t_0)$ for $0\leq t\leq t_0$ and note that then on $(0,t_0)$
all $V$, $\Vn$ and $W$ are convex. Indeed, for $V$ this was assumed, for  $\Vn$ it is clear and for $W$ we note that all terms in \eqref{eq:Wpp} are non-negative because $\Delta>0$ and $(-\Delta'')=V''\geq 0$.
Let us consider the second derivative of $W$ for $t\ge t_0$ and $t$ close to $t_0$ and we bound for $\Vn''(t)= 0$
\begin{align}
    W''(t)\gtrsim -M -\frac{M}{\Delta(t)}
    +\frac{(\Delta')^2(t)}{\Delta^2(t)}.
\end{align}
Moreover,  $\Delta(t)\approx \Vn(t_0)-V(t_0) $
and $\Delta'(t)\approx -V'(t_0)$ so if we choose $\Vn(t_0)$ such that $\Delta(t_0)$ is sufficiently small, $\Vn$ and $W$ will both be convex in an open neighbourhood of $t_0$. 
It remains to motivate that we can find $\Vn$ so that $W$ is globally convex. Here the idea is that $s\to \log(1-\e^{-s})$ is a bijection,  and we can therefore just define a function $G(s)$
suitably such that $V''(s)+G''(s)\geq 0$
for $s\geq t_0$, which then implicitly defines a function $\Delta$ so that $G(s)=-\log(1-\e^{-\Delta(s)})$, and then $W''=V''+G''$ by \eqref{eq:decomp_vnew_delta}. We will see below that 
we can do this such that in addition $\Vn''\geq - M$.  We construct $\Delta$ next, but let us first outline a couple of additional challenges that need to be addressed.
\begin{itemize}
    \item One main challenge is that we need to ensure that $t_1-t_0$ is not too small 
    so that we continue to make progress.
    While this is not generally true we can find a quantity (basically $V'(t_0)$
    and $\Vn'(t_1)$) so that either $t_1-t_0$
    is not too small or the quantity increases
    (see \eqref{eq:two_cases} below).
        \item We also need to keep track of the upper bounds on $\Vn''$ and $W''$.
    This is challenging because the number of terms in the decompositions can be infinite, but we are interested in uniform bounds. This is addressed by the tight bound \eqref{eq:Wpp_upper1b} in Lemma~\ref{le:decomp_two} which is an important ingredient when proving Theorem~\ref{th:decom_general}.
    \item While a key challenge of our construction is to 
    define the potentials in 
    $(t_0,\infty)$ to ensure that the terms at the next step are convex in $(t_0,t_1)$
    there are a couple of additional challenges related to defining the potentials in $(0,t_0)$ where the 
    initial potential was already convex.
    This construction is done in Steps 3 and 4 in the proof of Lemma~\ref{le:decomp_two} and, in particular, addresses:
    \begin{itemize}
    \item In order to control $\Vn'(t_1)$ we
    cannot use $\Vn$ constant on $(0,t_0)$.
    \item The upper bounds need to hold on $(0,t_0)$.
    \item We need to ensure that all functions are $C^2$ at the transition point $t_0$ of the two regimes.
    \end{itemize}
\end{itemize}

We now construct $\Vn$ where we initially define $\Vn$ on an interval $(t_0,\infty)$ (later $t_0$ will be chosen as the largest $t_0$
such that $V$ is convex on $(0,t_0)$).

As before, we denote the subtrahend $-\log\left(1-\e^{-\Delta(s)}\right)$ on the right-hand side of \eqref{eq:decomp_vnew_delta} by $G(s)$. Note that if we have chosen $G$, then $\Delta$, $\Vn$, and $\W$ are determined by this choice.

It is clear from \eqref{eq:decomp_vnew_delta} that $W$ will be convex if $G''(s)$ is pointwise bigger than $(\V'')_-$. Actually, this even suggests that an optimal choice of $G$ might satisfy $G''(s)=(\V'')_-$. 
This motivates the following definition of $G$. 
\begin{align}
\begin{split}\label{eq:def_G}
    G(t_0)= G_0>0, \quad G'(t_0)=G'_0>0,\\
    G''(t)=(\V'')_-(t), \quad \text{for $t\geq t_0$.}
    \end{split}
\end{align}
If $\V''$ is continuous this equation has a unique global solution given by
\begin{align}\label{eq:sol_Gp}
    G'(t)&=G'_0+\int_{t_0}^t (\V'')_-(s)\, \d s
    \\
    \label{eq:sol_G}
    G(t) &= G_0+(t-t_0) G'_0+\int_{t_0}^t
   \left(\int_{t_0}^s (\V'')_-(\tau)\, \d \tau\right)\,\d s.
\end{align}
Clearly, $G$ is a $C^2$ function on $(t_0,\infty)$, 
$G(t),G'(t)>0$ and $G''(t)=(\V'')_-(t)\geq 0$.

Untangling our definitions, we then have 
\begin{align}\label{eq:def_Delta2}
    \Delta(t)=
    -\log\left(1-\e^{-G(t)}\right)> 0,
\end{align}
and $\Delta(t)$ is a $C^2$ function on $(t_0,\infty)$. 

Moreover, if we now set $\Vn=\V+\Delta$ on $(t_0,\infty)$ then \eqref{eq:decomp_vnew_delta} becomes
\begin{equation}\label{V_{new}}
\W(\V,\Vn)(s)=\V(s) - \log\left(1-\e^{-\Delta(s)}\right)=\V(s) + G(s),
\end{equation}
and therefore 
\begin{align}
\label{V_{new}_convex}
\W''(\V,\Vn)(s)
=\V''(s)+G''(s) =
(\V'')_+(s)-(\V'')_-(s)+(\V'')_-(s)=(\V'')_+(s)\geq 0.
\end{align}
We can also find expressions for the derivatives of $\Delta$, namely
\begin{align}\label{eq:Deltap}
    \Delta'(t)=
    \frac{-G'(t)\e^{-G(t)}}{1-\e^{-G(t)}}
    =
    -\frac{G'(t)(1-\e^{-\Delta(t)})}{\e^{-\Delta(t)}}
    =
    -G'(t)(\e^{\Delta(t)}-1),
\end{align}
(so that in particular $\Delta'(t)<0$), 
and
\begin{align}\label{eq:Deltapp}
    \Delta''(t)=-G''(t)(\e^{\Delta(t)}-1)
    -G'(t)\Delta'(t)\e^{\Delta(t)}
    =-(\V'')_-(t)(\e^{\Delta(t)}-1)
    +\frac{\Delta'^2(t)}{1-\e^{-\Delta(t)}}.
\end{align}

In the following lemma we summarise the properties of $\Delta$ and $W$ that we have achieved with this construction.
\begin{lemma}\label{le:ode}
    Assume that $\V\colon[t_0,\infty)\to\R$ is a $C^2$ function such that $\V''(t)\geq -M$ for some constant $M>0$.
    Let $\Delta_0>0$ and $\Delta'_0<0$
    be two constants and set
    \begin{align}\label{eq:def_G0_G0prime}
        G_0=-\log(1-\e^{-\Delta_0})>0,\quad
        G_0'=\frac{-\Delta_0'}{\e^{\Delta_0}-1}>0.
    \end{align}
    Define $G$ as in \eqref{eq:def_G}
    and $\Delta$ as in \eqref{eq:def_Delta2}
    and set $\Vn=\V+\Delta$ and $\W=\W(\V,\Vn)$.
    Then:
    \begin{itemize}
        \item The function $\W$ is convex on $(t_0,\infty)$.
        \item The function $\Delta(t)$ decays exponentially, i.e., for $t\geq t_0$
        \begin{align}\label{eq:decay_delta}
           \Delta(t)
        \leq \Delta_0\e^{-G_0'(t-t_0)}.
        \end{align}
        \item The derivatives of $\Delta(t)$ decay exponentially and satisfy for $t\geq t_0$ the bounds
        \begin{align}\label{eq:deltap_bound}
            (-\Delta_0')\exp\left({ -\e^{\Delta_0}(t-t_0)G_0'-\e^{\Delta_0}(t-t_0)^2M/2}\right)\leq -\Delta'(t)\leq (-\Delta_0')\e^{-(t-t_0)G_0'}\left(1 +\frac{M(t-t_0)}{G'_0} \right).
        \end{align}
        \item Finally, the following bound holds for $t\geq t_0$
        \begin{align}
        \label{eq:key_bound_ode}
             G'_0&\leq  \frac{-\Delta'(t)}{\e^{\Delta(t)}-1}\leq G'_0+M(t-t_0).
        \end{align}
    \end{itemize}

\end{lemma}

\begin{proof}
    We have seen above that the first item holds by \eqref{V_{new}_convex}.
    Moreover, combining the assumption $\V''\geq -M$ 
     with \eqref{eq:sol_Gp} and \eqref{eq:sol_G} we directly get
    \begin{align}\label{eq:boundGp}
      G'_0&\leq   G'(t)=\frac{-\Delta'(t)}{\e^{\Delta(t)}-1}\leq G'_0+M(t-t_0)\\
      \label{eq:boundG}
      G_0+(t-t_0)G'_0&\leq G(t)\leq 
       G_0+(t-t_0)G'_0+\frac{(t-t_0)^2}{2}M.
    \end{align}
    We conclude in particular that \eqref{eq:key_bound_ode} holds. It remains to translate this bound into bounds for $\Delta$.
    To bound $\Delta $ we observe that convexity 
    of $x\to -\log(1-x)$ on $[0,1)$ then implies that
    for $0<y\leq x<1$ the bound 
    \begin{align}
        -\log(1-y)=-\log \left(1-\frac{y}{x}x-\left(1-\frac{y}{x}\right) 0\right)\leq \frac{y}{x}(-\log(1-x))+\left(1-\frac{y}{x}\right)\log 1=\frac{y}{x}(-\log(1-x))
    \end{align}
    holds. Combining this with the 
     relation between $G$ and $\Delta$
    we obtain (setting $x=\e^{-G_0}$ and $y=\e^{-G(t)}$) 
    \begin{align}
        \Delta(t)=-\log(1-\e^{-G(t)})
        \leq -\frac{\e^{-G(t)}}{\e^{-G_0}} \log(1-\e^{-G_0})
        \leq \Delta_0\e^{-G_0'(t-t_0)},
    \end{align}
  where for the last inequality we applied \eqref{eq:sol_G}. This proves \eqref{eq:decay_delta}. We move next to the proof of \eqref{eq:deltap_bound} and we start by bounding $\Delta'(t)$. Differentiating the relation 
    $\e^{-\Delta(t)}=1-\e^{-G(t)}$, we find
    \begin{align}\label{eq:deltap_first}
        -\Delta'(t)=G'(t)\e^{-G(t)+\Delta(t)}
        \geq G'_0\e^{-G(t)+\Delta(t)}
        =\frac{-\Delta'_0}{\e^{\Delta_0}-1}\e^{-G(t)+\Delta(t)}
        =-\Delta'_0\e^{\Delta(t)-G(t) - (\Delta_0-G_0)},
    \end{align}
 where for the inequality we used \eqref{eq:boundGp} and for the second equality we used the definition of $G_0'$ from \eqref{eq:def_G0_G0prime}.
    Next we observe that using \eqref{eq:Deltap}
    \begin{align}
        \Delta'(t)-G'(t)=-G'(t)\left(\e^{\Delta(t)}-1\right)-G'(t)
        =-\e^{\Delta(t)}G'(t)\geq -\e^{\Delta_0} G'(t),
    \end{align}
    where for the last step we used that $\Delta$ is monotonically decreasing by \eqref{eq:Deltap} and $G'(t)\ge 0$ by \eqref{eq:boundGp}.
    This implies that 
    \begin{align}
    \begin{split}\label{eq:aux_bound_Delta_G}
       \Delta(t)-G(t) - (\Delta_0-G_0) 
      & =\int_{t_0}^t\left(\Delta'(s)-G'(s)\right)\;\d s
       \geq -\e^{\Delta_0} \int_{t_0}^t G'(s) \d s
       =-\e^{\Delta_0}(G(t)-G(t_0))
       \\
       &\geq 
       -\e^{\Delta_0}(t-t_0)G'_0 -\e^{\Delta_0}\frac{(t-t_0)^2}{2}M,
       \end{split}
    \end{align}
  where for the inequality we applied \eqref{eq:boundG}.  Combining \eqref{eq:deltap_first} and \eqref{eq:aux_bound_Delta_G}
    we find the lower bound in \eqref{eq:deltap_bound}.
    On the other hand, we obtain the upper 
    bound as follows
    \begin{align}
    \begin{split}
         -\Delta'(t)&=\frac{G'(t)}{\e^{G(t)}-1}
         \leq \frac{G'_0+M(t-t_0)}{\e^{G_0+(t-t_0)G_0'}-\e^{(t-t_0)G_0'}}
         =
         \e^{-(t-t_0)G_0'}\frac{\e^{-G_0}}{1-\e^{-G_0}}(G_0'+M(t-t_0))
         \\
         &=
         -\Delta_0'\e^{-(t-t_0)G_0'}\left(1 +\frac{M(t-t_0)}{G'_0} \right),
         \end{split}
    \end{align}
where for the inequality we used the upper bound from \eqref{eq:boundGp} for the numerator, and for the denominator we used the lower bound from \eqref{eq:boundG} together with the bound $\e^{(t-t_0)G'_0}>1$ due to $G'_0>0$. For the last equality, we plugged in \eqref{eq:def_G0_G0prime}.    
\end{proof}

The previous lemma is the crucial ingredient to decompose a non-convex potential into two potentials, one of which is convex and the other slightly less non-convex than the initial potential. The next lemma extends the definition of the functions from $(t_0,\infty)$ to $(0, \infty)$ and provides careful bounds on the derivatives of $\W$ and $\Vn$.

\begin{lemma}\label{le:decomp_two}
    Let $\V:[0,\infty)\to\R$ be a  $C^2$ function with $V'(0)=0$
    and let $t_0$ be the largest value such that $V$ is convex on $(0,t_0)$.
    Assume that $0<t_0<\infty$
     and $\V'(t)>0$ on 
    $[t_0,\infty)$. 
    Assume that there is $M\geq 1$ such that $-\V''\leq M$
    and $-\V''/\V'\leq M$ on $(t_0,\infty)$. 
    Then there is an increasing $C^2$-function $\Vn\colon[0,\infty)\to \R$ 
 with $\Vn'(0)=0$ and $\Vn>V$ satisfying the following properties:
    \begin{itemize}
        \item There is $t_1>t_0$ such that
        $\Vn$ is convex on $[0,t_1)$ and (at least) one of the following two cases is true 
        \begin{align}\label{eq:two_cases}
           \left( t_1\geq t_0+(16M)^{-1}
            \text{ and } \Vn'(t_1)\geq \frac{\V'(t_0)}{2}
            \right)
            \quad \text{ or }\quad 
           \Big( \Vn'(t_1)\geq 2\V'(t_0)\text{ and }
           \V'(t_0)< 1\Big).
        \end{align}
        \item The function $\Vn$ is strictly monotone
        for $t\geq t_0$ and satisfies 
        \begin{align}\label{e:upperboundVn''Vn'}
            -\Vn''(t)\leq M\quad \text{and}
            \quad-\Vn''(t)/\Vn'(t)\leq M.
        \end{align}
        \item We can upper bound 
        \begin{align}\label{eq:Wpp_upper1}
            \Vn''(t) &\leq \V''(t)\quad& \text{for $t\leq t_0-(16M)^{-1}$ if $\V'(t_0)\geq 1$},\\
            \label{eq:Wpp_upper1b}
            \Vn''(t) &=0\quad& \text{for $t\leq t_0-(16M)^{-1}$ if $\V'(t_0)< 1$},\\
            \label{eq:Wpp_upper2}
            \Vn''(t) &\leq \V''(t)+48M\e^{-4M(t-t_0)}
            \min(\V'(t_0),1)
           \quad & \text{for $t\geq t_0-(16M)^{-1}$}.
        \end{align}
        \item The following estimates hold
        \begin{align}\label{eq:W0}
            \Vn(0)-\V(0)&\geq \frac{1}{64M}
            \quad &\text{if $\V'(t_0)\geq 1$}
            \\
            \label{eq:W1}
            \Vn(0)&\geq \V(t_0)-\frac{1}{32M}\quad 
            &\text{if $\V'(t_0)< 1$}
            \\
            \label{eq:W2}
            \Vn(t)-\V(t)&\leq \frac{\min(1,\V'(t_0))}{32M} \e^{-16M(t-t_0)}\quad &\text{for $t\geq t_0$}
        \end{align}
    \end{itemize}
    In addition, the function $\W=\W(\V,\Vn)$
    is a convex $C^2$-function with $W'(0)=0$ such that 
    \begin{align}\label{eq:Vpp_upper1}
       0 &\leq \W''(t)= (\V'')_+(t) \quad
        &\text{for $t\geq t_0$}
        \\ 
        \label{eq:Vpp_upper2}
        0&\leq \W''(t)\leq \frac{64M}{\min(\V'(t_0),1)}\V''(t)+(64M)^2
        \quad
       & \text{for $t\leq t_0$}.
    \end{align} 
    Moreover, if $\V$ is constant on $(0,t')$ for some $t'$
    then $\W''(t)=0$ for $t\leq t'$.
\end{lemma}
\begin{proof}
    We will define a suitable function $\Vn$.
    In the first step, we define it on the interval
    $(t_0,\infty)$  and show the claimed bounds on this interval and in the second step we extend it to 
     $(0,t_0)$.
\vspace{1mm}

     \textbf{Step 1: Definition of $\Vn$ on $(t_0,\infty)$ and of $t_1$}

\vspace{1mm}
    Here we essentially use the  construction from Lemma~\ref{le:ode} and derive additional bounds.
    We need to specify the initial values in \eqref{eq:def_G}. 
    Let 
    \begin{align}\label{eq:def_delta_0p}
    \Delta_0'=-\tfrac{1}{2}\min(1, \V'(t_0)) .
    \end{align}
    This definition ensures that $\Delta_0'$ is bounded (which is crucial for the upper bounds on $\Vn''$) and we can bound $\Vn'(t_0)=V'(t_0)+\Delta'(t_0)\geq V'(t_0)/2$ which allows us to control $\Vn'(t_1)$ later.
    Define $\Delta_0$ to be the unique solution of the equation 
    \begin{align}\label{eq:def_Delta0}
        \e^{\Delta_0}-1 = \frac{-\Delta_0'}{16M},
    \end{align}
    i.e., $\Delta_0=\log(1-(16M)^{-1}\Delta_0')$. This definition ensures that the last term in \eqref{eq:Wpp}
    is large for $s=t_0$ thereby ensuring convexity at $t_0$.
    This definition implies (recall $-\Delta_0'\leq 1/2$, $M\geq 1$) that 
    \begin{align}\label{eq:aprior_delta0}
        \e^{\Delta_0}\leq 1+\frac{1}{32}\leq 2,
    \end{align}
    which estimate will be needed later (for example in \eqref{112} below). Set $G_0$ and $G_0'$ as in \eqref{eq:def_G0_G0prime}, i.e., 
    \begin{align}
    \label{eq:G0_G0prime_concrete}
         G_0=-\log(1-\e^{-\Delta_0})>0,\quad
        G_0'=\frac{-\Delta_0'}{\e^{\Delta_0}-1}>0,
    \end{align}
    set $G$ as in \eqref{eq:def_G}, and finally define $\Delta$
    by $\Delta(t)=-\log(1-\e^{-G(t)})$ for $t\geq t_0$.
    We then define 
    \begin{align}
        \Vn(t) = \V(t)+\Delta(t)=
        \V(t)-\log(1-\e^{-G(t)}).
    \end{align}
    Note that then
    \begin{equation}\label{e:G0'}
        G_0'={16M}.
    \end{equation}

    We claim that $\Vn''(t_0)>0$. Assuming this claim for a moment, we can define 
    \[t_1=\sup\{t\ge t_0\colon \Vn''(s)>0 \text{\; for all $s\in [t_0,t)$} \}\in\R\cup\{\infty\},\]
    and we note that because $\Vn''$ is continuous, we must have $t_1>t_0$ and either $t_1=+\infty$ or $\Vn''(t_1)=0$.

Let us now supply the proof that $\Vn''(t_0)>0$.
    Using $\V''\geq -(\V'')_{-}$ and \eqref{eq:Deltapp}, we have for $t\geq t_0$ that 
    \begin{equation}
        \Vn''(t)=\V''(t)+\Delta''(t)\ge-(\V'')_-(t)+\frac{\Delta'(t)^2}{1-\e^{-\Delta(t)}} 
        -(\e^{\Delta(t)}-1) (\V'')_-
        =\e^{\Delta(t)}\left(\frac{\Delta'(t)^2}{\e^{\Delta(t)}-1} - (\V'')_-\right).
    \end{equation}
    Using the lower bound in \eqref{eq:key_bound_ode} as well as \eqref{e:G0'}, together with the assumption $-(\V'')_-\ge -M\min(1,\V'(t))$, we obtain the estimate
    \begin{equation}
    \label{W''convex}
           \Vn''(t) \geq \e^{\Delta(t)} \left(G_0'(-\Delta'(t))
        -M\min(1,\V'(t))\right)=\e^{\Delta(t)} \left(16M(-\Delta'(t))
        -M\min(1,\V'(t))\right).
        \end{equation}
This, together with the definition of $\Delta'_0$ in \eqref{eq:def_delta_0p},  already implies that $\Vn''(t_0)$ is strictly positive.
   \vspace{1mm}

\textbf{Step 2: Verification of lemma estimates on $(t_0,\infty)$ }

\vspace{1mm}
After having defined $\Vn$ on $(t_0,\infty)$, we will now verify all claimed estimates on that interval. 

We begin by establishing \eqref{eq:two_cases}. To do so, we  begin with the observation that convexity of $\Vn$ on $(t_0,t_1)$ implies that 
    \begin{equation}\label{e:estVnew'}
        \Vn'(t_1)\geq \Vn'(t_0)\geq \V'(t_0)/2,
    \end{equation}
    where we also used \eqref{eq:def_delta_0p}. This still holds if $t_1=+\infty$ if we use the convention $\Vn'(+\infty)=\lim_{t\to\infty}\Vn'(t)$.

Next, we need a quantitative lower bound for $t_1-t_0$.
    We apply the lower bound of \eqref{eq:deltap_bound} and find that for 
    $t\leq t_0+(16M)^{-1}=t_0+G_0'^{-1}$ the bound 
    \begin{align}
    \label{Delta'bd}
        -\Delta'(t)\geq -\Delta'_0\exp\left(-\e^{\Delta_0}\left(1+\frac{1}{2\cdot 16^{2}M}\right)\right)
        \geq \frac{-\Delta'_0}{3}
    \end{align}
    holds (here the last step can be checked by explicit calculation by using $\e^{\Delta_0}\le 1+1/32$ and $M>1$). 
    
    Consider first the case that $\V'(t_0)\geq 1$. In this case, by definition, 
    $\Delta'_0=-\tfrac{1}{2}$ and thus for $t \leq t_0+(16M)^{-1}$ we have from \eqref{W''convex} and \eqref{Delta'bd}
    \begin{align}
        \Vn''(t)\geq \e^{\Delta(t)} \left(16M\frac{-\Delta'_0}{3}
        -M\right)=\e^{\Delta(t)} \left(\frac83M
        -M\right)>0,
    \end{align}
    and so we must have $t_1>t_0+(16M)^{-1}$. Therefore, in view of \eqref{e:estVnew'} the first case of \eqref{eq:two_cases} occurs.
    
    Next consider the case that $\V'(t_0)\leq 1$. Then
    $-\Delta'_0=\V'(t_0)/2$ and thus for $t \leq t_0+(16M)^{-1}$
     \begin{equation}\label{e:lowerboundVnew''}
      \Vn''(t)\geq \e^{\Delta(t)} \left(16M \frac{-\Delta_0'}{3}
        -M\V'(t)\right)
        =\e^{\Delta(t)}M\left(\frac{8}{3}\V'(t_0)-\V'(t)\right).
    \end{equation}
    If $t_1\ge t_0+(16M)^{-1}$ then again the first case of \eqref{eq:two_cases} occurs (recall \eqref{e:estVnew'}).
    If $t_1<t_0+(16M)^{-1}$, though, then $\Vn''(t_1)=0$ and by \eqref{e:lowerboundVnew''} we have    
    \begin{align}\label{eq:increaseV'}
        \V'(t_1)\geq \frac{8}{3}\V'(t_0).
    \end{align}
    Our goal is to translate this into a similar bound for $\Vn'$. Because $\Vn=\V+\Delta$, \eqref{eq:increaseV'} implies that
    \begin{align}
         \Vn'(t_1)=\V'(t_1)+\Delta'(t_1)
         \geq \frac{8}{3}\V'(t_0)-\frac{\V'(t_0)}{2}\geq 2\V'(t_0).
    \end{align}
    and so the second case of \eqref{eq:two_cases} holds.

    We have now shown that always at least one of the cases in \eqref{eq:two_cases} occurs. Next we will verify that $\Vn$ is strictly monotone for $t\ge t_0$, and that \eqref{e:upperboundVn''Vn'} holds.

For the former statement, we will use that $\Vn=\V+\Delta$,
    where we control the derivative of both terms.
    Note that Gronwall's inequality combined with
    the bound $-\V''/\V'\leq M$ implies that
    \begin{align}
    \label{eq:increaseV'1}
        \V'(t)\geq \e^{-M(t-t_0)}\V'(t_0).
    \end{align}
    
    On the other hand, we observe that the upper bound of \eqref{eq:deltap_bound} combined with $G_0'=16M$ and $M\ge1$ implies that
    \begin{align}\label{eq:bound_deltap_special}
      -\Delta'(t)\leq -\Delta_0'\e^{-16M(t-t_0)}\left(1+\frac{1}{16}(t-t_0)\right)\leq -\Delta_0'\e^{-16M(t-t_0)}\e^{\frac{t-t_0}{16}}
      \leq -\Delta_0'\e^{-15M(t-t_0)}.
    \end{align}
    We conclude from \eqref{eq:increaseV'1} and \eqref{eq:bound_deltap_special} that for $t\geq t_0$
    \begin{equation}\label{e:Vn'}
        \Vn'(t)=\V'(t)+\Delta'(t)
        \geq \e^{-M(t-t_0)}\V'(t_0) +\Delta'_0\e^{-15M(t-t_0)}
        \geq \frac{\V'(t_0)}{2}\e^{-M(t-t_0)}.
    \end{equation}
    This implies that $\Vn$ is strictly monotone on $(t_0,\infty)$.
   
    Next, we show \eqref{e:upperboundVn''Vn'}.
    First we note that for $t> t_0$ by \eqref{eq:Deltapp} and the fact that $(\V'')_-\le M$ we have 
     \begin{equation}
    \begin{split}\label{eq:lower_boundDelta''}
    \Delta''(t)&=\frac{\Delta'(t)^2}{1-\e^{-\Delta(t)}} 
        -(\e^{\Delta(t)}-1) (\V'')_-\\
        &
        \geq 
        (-\Delta'(t))\left(\e^{\Delta(t)}\frac{-\Delta'(t)}{\e^{\Delta(t)}-1} -M\frac{\e^{\Delta(t)}-1}{-\Delta'(t)}\right)
       \\
       &\geq (-\Delta'(t))\left(\e^{\Delta(t)}G_0'-\frac{M}{G_0'}\right)\\
       &\geq  (-\Delta'(t))\left(16M-1\right)\\
       &\ge -M\Delta'(t)\ge0.
       \end{split}
    \end{equation}
    Here in the third step we used
    that $ax-b/x$ is monotone in $x$ for positive $a,b$ and $x$, combined with the lower bound of \eqref{eq:key_bound_ode}, and in the fourth step we used $\Delta(t)\ge0$ and $G_0'=16M$. This then implies that
    \begin{align}
       \Vn''(t)=\V''(t)+\Delta''(t)\geq -M.
    \end{align}

    In order to bound $-\Vn''/\Vn'$ we note that we have already seen in \eqref{e:Vn'} that $\Vn'$ is positive. Thus the bound we want to show is equivalent to 
    \begin{align}
        M\Vn'=M(\V'+\Delta')\geq -\V''-\Delta''=-\Vn''.
    \end{align}
    Combining \eqref{eq:lower_boundDelta''} with the assumption
    \begin{align}
        M\V'\geq -\V''
    \end{align}
    the claim follows. This completes the proof of \eqref{e:upperboundVn''Vn'}.
    
    Finally, we upper bound $\Vn''$ and establish \eqref{eq:Wpp_upper2} (for $t\ge t_0)$).
    We find for $t\geq t_0$ (using that the first summand in the equation \eqref{eq:Deltapp} for $\Delta''$ is negative since $-(V'')_{-}(t)\le 0$), by combining the upper bounds of
    \eqref{eq:key_bound_ode}, \eqref{eq:aprior_delta0} and \eqref{eq:bound_deltap_special}, that 
    \begin{align}
    \begin{split}\label{eq:Vnpp_right}
        \Vn''(t)=\V''(t)+\Delta''(t)&\leq \V''(t) +\frac{\Delta'(t)^2}{1-\e^{-\Delta(t)}}
        \\
&= \V''(t) + \e^{\Delta(t)}(-\Delta'(t))
        \frac{-\Delta'(t)}{\e^{\Delta(t)}-1}
     \\
     &\leq
         \V''(t) + 2(-\Delta'_0)\e^{-15M(t-t_0)}
         (G_0'+M(t-t_0))
     \\
     &\leq 
          \V''(t) + 32M(-\Delta'_0)\e^{-15M(t-t_0)}
         \left(1+\frac{1}{16}(t-t_0)\right)
         \\
         &\leq \V''(t) + 32M(-\Delta'_0)\e^{-14M(t-t_0)}.
        \end{split}
    \end{align}
In the fourth step, we used that $G_0'=16M$.  This proves the bound \eqref{eq:Wpp_upper2}
    for $t\geq t_0$ (the larger constants in \eqref{eq:Wpp_upper2} will be necessary later on so that the bound will also hold for $t<t_0$). 

    We also remark that combining the definition of $\Delta_0$ in \eqref{eq:def_Delta0}
    (implying $\Delta_0\leq -\Delta_0'/(16M)$) with the bound \eqref{eq:decay_delta}
    we directly find \eqref{eq:W2}. 
    
    To finish the estimates for the interval $(t_0,\infty)$  it remains to prove the bounds for $\W$ on this interval, where we recall here the definition of $\W$ from \eqref{V_{new}}.
    Then by  
    \eqref{V_{new}_convex} we have $0\leq  \W''(t)=(\V'')_+(t)$ for $t\geq t_0$, and so \eqref{eq:Vpp_upper1} follows. 
    This completes the verification of all properties of $\Vn$ and $\W$ for $t\ge t_0$.

In the next two steps we will extend $\Vn$ to $(0,t_0]$ and verify the estimates on that interval as well. We distinguish  two cases for the definition of $\Vn$ on $(0,t_0)$, depending on whether $\V'(t_0)\ge 1$ or $\V'(t_0)<1$, and we first focus on the case $\V'(t_0)\ge 1$.

\textbf{Step 3: Extending $\Vn$ to $(0,t_0)$ if $\V'(t_0)\geq 1$}

\vspace{1mm}

 Recall that in \eqref{eq:def_delta_0p} we defined $\Delta_0'=-1/2$
and thus \eqref{eq:def_Delta0} becomes now
\begin{align}\label{eq:edelta0case1}
    \e^{\Delta_0}-1=\frac{1}{32M}.
\end{align}
In this case, the key idea is to essentially define
 $\Vn$ on $(0,t_0)$ by the (backward) ODE
    \begin{align}\label{eq:defWpplargeVp}
    \begin{split}
        \Vn''(t) = \V''(t) \frac{\Vn'(t_0)}{\V'(t_0)}
        \\
            \Vn(t_0)=\V(t_0)+\Delta_0,
            \quad \Vn'(t_0)=\V'(t_0)+\Delta_0'.
        \end{split}
    \end{align}
    This would ensure that $0\leq \Vn''\leq V''$
    and it is also sufficient to derive the bounds for $W$.
    However, this definition is not $C^2$ in $t_0$ so we make the following minor modification. We define $\Vn$ on 
    the interval $(t'_0,t_0)$ 
    where $t'_0$ will be specified below by the backward ODE 
    \begin{align} \begin{split}  \label{eq:ode_discont}  \Vn(t_0)=V(t_0)+\Delta_0,\quad\Vn'(t_0)=V'(t_0)+\Delta_0'
        \\
        \Vn''(t)=V''(t_0)+\Delta''(t_0)-C(t_0-t)
        =\Delta''(t_0)-C(t_0-t)
        \end{split}
    \end{align}
    using $V''(t_0)=0$
    (this follows because $t_0$ is the largest value so that $V$ is convex on $(0,t_0)$). 
    Here $C>0$ is a large constant 
    and $t'_0$ is the largest solution of $\Vn''(t)=V''(t)\Vn'(t)/V'(t)$.
    On $(0,t'_0)$ we finally define
    $\Vn$ by the backward ODE 
    \begin{align}
    \label{extendVn}
        \Vn''(t)=V''(t)\frac{\Vn'(t'_0)}{V'(t'_0)},
    \end{align}
    where the initial values at $t'_0$ here are chosen such that $\Vn$ is $C^2$. 
    We first claim that $t'_0$ exists for $C$ large enough. Note that $
    \Vn''(t_0)=V''(t_0)+\Delta''(t_0)=\Delta''(t_0)>0$ (see, e.g.,
    \eqref{eq:lower_boundDelta''}).  Solving the ODE, we find on $(t'_0,t_0)$ from \eqref{eq:ode_discont}
    \begin{align}
       \frac{ \Vn''(t)}{\Vn'(t)}
       =\frac{V''(t_0)+\Delta''(t_0)-C(t_0-t)}
       {V'(t_0)+\Delta_0'-(V''(t_0)+\Delta''(t_0))(t_0-t)+C(t_0-t)^2/2}.
    \end{align}
    For $\bar{t}=t_0-C^{-2/3}$ and as $C\to\infty$ we find for $t\in (\bar{t}, t_0)$
    \begin{align}\label{eq:lim_C}
     \Vn'(t)=  V'(t_0)+\Delta_0'-(V''(t_0)+\Delta''(t_0))(t_0-t)+C(t_0-t)^2/2\to  V'(t_0)+\Delta_0'
    \end{align}
    and in particular
    the denominator eventually stays positive in $(\bar{t}, t_0)$. 
    We then find that $\Vn''(\bar{t})/\Vn'(\bar{t})\to -\infty$.
    By continuity and the intermediate value theorem we can indeed find a solution of $\Vn''(t)/\Vn'(t)=V''(t)/V'(t)$ in 
    $(\bar{t},t)$ for $C$ sufficiently large.
    Moreover, the solution $t'_0=t'_0(C)$ will satisfy $t'_0(C)\to t_0$ as $C\to \infty$
    so that $V'(t'_0(C))\to V'(t_0)$ and also
    $\Vn'(t'_0(C))\to \Vn'(t_0)$ and similarly for $V$ and $\Vn$  by \eqref{eq:lim_C}.
    Therefore, $\Vn'(t'_0)/V'(t'_0)\to \Vn'(t_0)/V'(t_0)< 1$ as $C\to \infty$.
    Let us now fix $C>0$ such that on
    $(t'_0,t_0)$
    \begin{align}\label{eq:Vn'}
    \begin{split}
        \Vn'(t)/V'(t)&<1,\quad
     |   V''(t)- V''(t_0)|\leq M,\quad
        \e^{\Delta(t)}-1\geq
        \frac{\e^{\Delta_0}-1}{2^{1/4}},
        \quad
        1- \e^{-\Delta(t)}\geq
        \frac{1- \e^{-\Delta_0}}{2^{1/4}},
        \\
        |\Delta'(t)|&\leq 2^{1/4}|\Delta'(t_0)|
        \quad
         \text{and moreover }
        t'_0\geq t_0-(16M)^{-1}.
        \end{split}
    \end{align}
   Here the second condition can be ensured by the  continuity of $V''$.
    We now show the claimed bounds for $\Vn$
    and then for $W$.
    Convexity of $\Vn$ follows because $V''(t)\geq 0$ and $\Vn'(t'_0)/V'(t'_0)> 0$. Next we show the upper bound on $\Vn''$
   Using \eqref{eq:Vnpp_right} (for $t=t_0$) we find on $(t'_0,t_0)$
    \begin{align}
        \Vn''(t)\leq \Vn''(t_0)
        \leq V''(t_0)+16M\leq 
        V''(t)+17M\leq V''(t)+48M\e^{-4M(t-t_0)}\min(V'(t_0),1).
    \end{align}
    For $t<t_0'$ we have $\Vn''(t)=V''(t)\Vn'(t'_0)/V'(t'_0)\leq V''(t)$
    so that \eqref{eq:Wpp_upper2} also holds
    on $(t_0-(16M)^{-1},t'_0)$
    and \eqref{eq:Wpp_upper1} holds. 
    The relation $\Vn''(t)=V''(t)\Vn'(t'_0)/V'(t'_0)$
    implies 
    \begin{align}\label{eq:vnp}
        \Vn'(t)=V'(t)\frac{\Vn'(t'_0)}{V'(t'_0)}
    \end{align}
    and therefore $\Vn'(0)=0$ and $\Vn'(t)\leq V'(t)$ on $(0,t'_0)$. Since this is also true on $(t'_0,t_0)$ by \eqref{eq:Vn'}
    we find that $\Vn(t)-V(t)$ is decreasing
    on $(0,t_0)$ and thus by \eqref{eq:edelta0case1} (for $M\geq 1$)
    \begin{align}
        \Vn(0)-V(0)\geq \Vn(t_0)-V(t_0)
        =\Delta_0=\log(1+(32M)^{-1})\geq \frac{1}{64M}
    \end{align}
    which is \eqref{eq:W0}.
    
    Finally, we show the bounds for $W$.
   Recall that by \eqref{eq:Wpp}
    \begin{align}
         W''(t)&= V''(t)-\frac{\Delta''(t)}{\e^{\Delta(t)}-1} + \frac{(\Delta')^2(t)}{(\e^{\Delta(t)}-1)(1-\e^{-\Delta(t)})}.
    \end{align}
We now show convexity of $W$ on $(t'_0,t_0)$. We can bound the above by using $V''(s)\geq 0$ and that $\Vn''(s)=\V''(s)+\Delta''(s)\ge\Delta''(s)$ followed by \eqref{eq:ode_discont} in the second step
\begin{align}
     W''(t)&\geq  -\frac{\Vn''(t)}{\e^{\Delta(t)}-1} + \frac{(\Delta')^2(t)}{(\e^{\Delta(t)}-1)(1-\e^{-\Delta(t)})}
     =\frac{1}{\e^{\Delta(t)}-1}
     \left( -\Delta''(t_0)+C(t_0-t)
     + \frac{(\Delta')^2(t)}{1-\e^{-\Delta(t)}}
     \right).
\end{align}
By construction the term in brackets vanishes 
at $t=t_0$ because $W''(t_0)=V''(t_0)=0$.
Thus, it is sufficient to show that the derivative of the bracket with respect to $t$ is negative on $(t'_0,t_0)$. We bound
\begin{align}
\begin{split}
  \left| \left(\frac{(\Delta')^2(t)}{1-\e^{-\Delta(t)}}\right)'\right|
   &=\left| 2\Delta''(t)\frac{(\Delta')(t)}{1-\e^{-\Delta(t)}}
   -\frac{(\Delta')^3(t)}{(\e^{\Delta(t)}-1)(1-\e^{-\Delta(t)})}\right|
   \\
   &\leq 4(|V''(t_0)|+M+|\Delta''(t_0)|+C|t-t_0|)\frac{|\Delta'(t_0)|}{1-\e^{-\Delta(t_0)}}
   +2\frac{|(\Delta')^3(t_0)|}{(\e^{\Delta(t)}-1)(1-\e^{-\Delta(t)})}.
   \end{split}
\end{align}
Note that $|t-t_0|\leq C^{-2/3}$ for $t\in (t_0',t_0)$.
So the expression in the last display scales as $C^{1/3}$ and we  conclude that on $(t_0',t_0)$
\begin{align}
 \left( -\Delta''(t_0)+C(t_0-t)
     + \frac{(\Delta')^2(t)}{1-\e^{-\Delta(t)}}
     \right)'\leq -C+K_1+K_2C^{1/3}\leq 0
\end{align}
where $K_1$ and $K_2$ are two constants.
This implies $W''(t)\geq 0$ on $(t_0',t_0)$ for $C>0$ sufficiently large.
On $(0,t'_0)$ we infer $W''(t)\geq 0$ directly from
$V''(t)\geq 0$ and $\Delta''(t)=\Vn''(t)-V''(t)\leq 0$. For the upper bound, we observe
that on $(t'_0,t_0)$ using $\Vn''(t)\geq 0$
\begin{align}
\begin{split}
    W''(t)-V''(t)
    &\leq \frac{V''(t)}{\e^{\Delta(t)}-1}
    +\frac{(\Delta'(t))^2}{(\e^{\Delta(t)}-1)(1-\e^{-\Delta(t)})}
    \\
    &\leq 
    2\frac{V''(t)}{\e^{\Delta(t_0)}-1}
    +\frac{2^{2/4}(\Delta'(t_0))^2}{2^{-2/4}(\e^{\Delta(t_0)}-1)(1-\e^{-\Delta(t_0)})}
    \\
    &\leq 
     2^{1/4}(32M)V''(t)
    +\frac{2\cdot 2^{-2}\e^{\Delta_0}}{(\e^{\Delta(t_0)}-1)^2}
    \leq 40MV''(t)+(32M)^2
    \end{split}
\end{align}
where we used $\e^{\Delta_0}\leq 2$ in the last step (and \eqref{eq:edelta0case1}).
Moreover, using $\Vn''(s)\leq V''(s)$ and
$\Vn'(s)\leq V'(s)$ for $0\leq s\leq t'_0$
we find that $\Delta$ and $\Delta'<0$ are decreasing on this interval and we can
bound 
\begin{align}\begin{split}\label{eq:Wpp_case1}
    W''(t)-V''(t)
    &\leq \frac{V''(t)}{\e^{\Delta(t)}-1}
    +\frac{(\Delta'(t))^2}{(\e^{\Delta(t)}-1)(1-\e^{-\Delta(t)})}
    \\
    &\leq \frac{V''(t)}{\e^{\Delta(t'_0)}-1}
    +\frac{(\Delta'(t'_0))^2}{(\e^{\Delta(t'_0)}-1)(1-\e^{-\Delta(t'_0)})}
    \\
    &
    \leq 40MV''(t)+(32M)^2.
    \end{split}
\end{align}
This shows \eqref{eq:Vpp_upper2} in this case.
\vspace{1mm}

\textbf{Step 4: Extending $\Vn$ to $(0,t_0)$ if $\V'(t_0)< 1$ as a piecewise $C^2$ function}

\vspace{1mm}
    Finally, we consider the case where $\V'(t_0)< 1$ holds. Here we again split the interval $(0,t_0)$ into two intervals
    $(0,t_0')$ and $(t_0',t_0)$
    where we define $\Vn$ on $(t_0',t_0)$
    so that $\Vn'(t'_0)=0$ and then let $\Vn$ be constant on $(0,t'_0)$. We will first give a version of the construction that might not be $C^2$ at $t_0'$, and in the next step we will explain how to modify it to ensure that $\Vn$ is globally $C^2$.

    To define $\Vn$ on $(t_0',t_0)$ we, as before, define it through $G$.
    We extend $G$ in the backward direction as in \eqref{eq:def_G} by
    \begin{equation}
    \label{extGback}
        G''(t)=(\V'')_-=0,\quad
        G(t_0)=G_0>0,\quad G'(t_0)=G_0'>0,
    \end{equation}
    with $G_0, G_0'$ as in \eqref{eq:def_G0_G0prime}. Explicitly, $G(t)=G_0-(t_0-t)G_0'$. 
    For $t> t_0-G_0/G_0'$ the function $G$ is positive and so, given any $t_0'\geq t_0-G_0/G_0'$ (fixed below), we can define $\Delta=-\log(1-\e^{-G})$ and
    $\Vn=\V+\Delta$ on the interval $(t_0',t_0)$. 
    By \eqref{eq:Deltapp} 
    the function $\Vn$
    satisfies on the interval $(t_0',t_0)$ the backward ODE
    \begin{align}
    \label{backode1}
       \Vn''(t) = V''(t)+\Delta''(t)=\V''(t)+\frac{\Delta'(t)^2}{1-\e^{-\Delta(t)}} \\
       \Vn(t_0)=\V(t_0)+\Delta_0,
       \quad \Vn'(t_0)=\V'(t_0)+\Delta_0'.
    \end{align}
    We claim that we can find 
    \begin{align}\label{eq:t0prime2}
        t_0'
    > \max(t_0-(16M)^{-1},0) \quad (>t_0-G_0/G'_0)
    \end{align} 
    such that $\Vn'(t_0')=0$.
    We note that  $t_0-(16M)^{-1}>t_0-G_0/G'_0$
    which follows from $G_0'=16M$ and $G_0=\Delta_0-\log(-\Delta_0'/(16M))\geq\log(2)+\log(16)\geq 1$
    (see \eqref{eq:G0_G0prime_concrete}).
    Therefore 
     $\Vn$ is well defined for $t_0-(16M)^{-1}\leq t\leq t_0$. Using \eqref{eq:Deltap}, we get for 
$t_0-G_0/G_0'\leq t\leq t_0$
\begin{align}
    \frac{-\Delta'(t)}{-\Delta'(t_0)}
    = \frac{\frac{G'(t)}{\e^{G(t)}-1}}{\frac{G'(t_0)}{\e^{G(t_0)}-1}}
    =
    \frac{\e^{G(t_0)}-1}{\e^{G(t_0)-(t_0-t)G'_0}-1}
    \geq \frac{\e^{G(t_0)}}{\e^{G(t_0)-(t_0-t)G'_0}}=\e^{(t_0-t)G'_0}=\e^{16M(t_0-t)},
\end{align}
where for the second equality we applied the definition of the extended $G$, and for the last equality we used that $16M=G_0'$.
We conclude (using convexity of $V$ on $(0,t_0)$ and $-\Delta'_0=V'(t_0)/2$)
\begin{align}
    \Vn'(t)=V'(t)+\Delta'(t)
    \leq V'(t_0)+\e^{16M(t_0-t)}\Delta'(t_0)
    =V'(t_0)\left(1-\frac12\e^{16M(t_0-t)}\right).
\end{align}
This is negative for $t<t_0-(16M)^{-1}$
so by continuity we can indeed find $t_0'$
satisfying \eqref{eq:t0prime2} such that $\Vn'(t_0')=0$.
    It remains to be shown that $t_0'>0$. Note that $\Vn'(t_0)<V'(t_0)$ and $\Vn''(t)>V''(t)$
    on $(t_0',t_0)$ so $\Vn'(t)<V'(t)$
    and then $V'(0)=0$ implies $t_0'>0$.

    Extend $\Vn$
     to $(0,t_0')$  as a  constant function.
     The function $\Vn$ defined like this is $C^1$ and twice differentiable  everywhere except for the point $t_0'$.

    Note that $\Vn$  is clearly convex in $(0,t_0)$.
    Similarly, $\W$ is convex. 
    Indeed, by \eqref{V_{new}}
    we get 
    $\W''=\V''+G''=\V''$ on $(t_0',t_0)$
    (see \eqref{extGback} for the second step). On   $(0,t_0')$ we use $\Vn''=0\leq V''$ and therefore $\Delta''<0$
    to conclude from \eqref{eq:Wpp} that $\W''\geq \V''\geq 0$. 
    
    It remains to show the  bounds
    for $\Vn$ and $W$ on $(0,t_0)$, more concretely we need to prove 
    \eqref{eq:Wpp_upper1b}, \eqref{eq:Wpp_upper2}, \eqref{eq:W1}, and
    \eqref{eq:Vpp_upper2}. 
    Since $t_0'\geq t_0-(16M)^{-1}$ we conclude that $\Vn$ is constant on $(0, t_0-(16M)^{-1})$ and we directly find that \eqref{eq:Wpp_upper1b} is satisfied. 
   By convexity and using $\Vn'(t_0)=\V'(t_0)/2<1/2$ and
   $t_0-t_0'\leq (16M)^{-1}$
    \begin{align}
    \label{delta'bd1}
        \Vn(0)=\Vn(t_0')\geq
        \Vn(t_0)-(t_0-t_0')\Vn'(t_0)\geq \V(t_0)- (32M)^{-1}
    \end{align}
    which proves \eqref{eq:W1}.
    The next step is to upper bound $\Vn''$ on the interval    $(t_0',t_0)$
    in order to prove \eqref{eq:Wpp_upper2}. First, we find in $(0,t_0)$ 
    \begin{align}\label{eq:deltap_trivial}
         -\Delta'(t)=\V'(t)-\Vn'(t)
        \leq \V'(t_0),
    \end{align}
    where we used convexity of $\V$ and $\Vn'\geq 0$ in the last step.
    This implies (by integration and using the bounds $t_0-t_0'\le (16 M)^{-1}$, $V'(t_0)\le 1$ and $\Delta_0\le \log(1+(32)^{-1})$ by \eqref{eq:def_Delta0}) that $\Delta(t)\leq \Delta_0+(16M)^{-1}\V'(t_0)\leq 1/2+1/16\leq  2/3$, so that 
    \begin{align}
        \e^{\Delta(t)}\leq \e^{2/3}\leq 2.
    \end{align}
    Then we find using \eqref{backode1}  together with the last display,  \eqref{eq:deltap_first}, and \eqref{eq:Deltap}
    \begin{align}
        \Vn''(t)\leq \V''(t)+(-\Delta'(t))\e^{\Delta(t)}\frac{-\Delta'(t)}{\e^{\Delta(t)}-1} 
        \leq \V''(t)+ \V'(t_0)\cdot 2\cdot G_0'
        \leq \V''(t)+32M \V'(t_0).
    \end{align}
    This finishes the proof of \eqref{eq:Wpp_upper2}.
    The last step is to bound $\W''$ by showing \eqref{eq:Vpp_upper2}.
    On $(t_0',t_0)$ we have $\W''=\V''$
    so it remains to prove \eqref{eq:Vpp_upper2}
    on $(0, t_0')$.
    For $t<t_0'$ we note that $\Delta(t)>\Delta_0$, 
    and thus (starting from \eqref{eq:Wpp} with 
    $\Vn''=\Vn'=0$) 
    \begin{align}
    \label{112}
    \begin{split}
       \W''(t)&=V''(t)-\frac{\Delta''(t)}{\e^{\Delta(t)}-1} + \frac{(\Delta')^2(t)}{(\e^{\Delta(t)}-1)(1-\e^{-\Delta(t)})}
       \\
       &=\frac{1}{1-\e^{-\Delta(t)}}\left(\V''(t)+\frac{(\V'(t))^2}{\e^{\Delta(t)}-1}\right)
       \leq \frac{1}{1-\e^{-\Delta_0}}
       \left(\V''(t)+\frac{(\V'(t_0))^2}{\e^{\Delta_0}-1}
       \right)
       \\
       &= \frac{\e^{\Delta_0}}{\e^{\Delta_0}-1}
       \left(\V''(t)+2\V'(t_0)\frac{-\Delta_0'}{\e^{\Delta_0}-1}
       \right)
       \leq \frac{64M}{\V'(t_0)}\left(\V''(t)+32M\V'(t_0)\right)
       \\
       &\leq \frac{64M}{\V'(t_0)}\V''(t)+(64M)^2,
       \end{split}
    \end{align}
    where for the penultimate inequality we applied \eqref{eq:def_delta_0p}, \eqref{eq:def_Delta0}, and \eqref{eq:aprior_delta0}. The last display together with \eqref{eq:Wpp_case1} imply that \eqref{eq:Vpp_upper2} holds in all cases.
    We  also remark that $\W''(t)=0$
    if $\V''(t')=0$ for all $0\leq t'\leq t\leq t_0'$.
\vspace{1mm}

\textbf{Step 5: Extending $\Vn$ to $(0,t_0)$ if $\V'(t_0)< 1$ as a $C^2$ function}
\vspace{1mm}

 The functions $\Vn$, $W$ from the previous step have all required properties with the one exception that they might not be $C^2$ at $t_0'$. In order to fix this, we will modify the construction of $\Vn$ near $t_0'$. To avoid confusion, we will denote the modified objects with a tilde.
    Our goal is to find $\tilde t_0'$ close to $t_0'$, and $\tVn$ close to $\Vn$ such that $\tVn'(\tilde t_0')=\tVn''(\tilde t_0')=0$. Then we can extend $\tVn$ piecewise constant on $(0,\tilde t_0')$ and the resulting function will be genuinely $C^2$. This modification is slightly technical, as we need to preserve all properties of $W,\Vn$. However, the estimates are similar to the previous step, and so we will be brief here.
    
    Let $\tilde t_0'>0$ and $\eps>0$ be parameters to be fixed later, and let $\theta\in C^\infty(\R)$ be such that $0\le\theta(t)\le1$, $\theta(t)=1$ for $t\le0$ and $\theta(t)=0$ for $t\ge1$. Define functions $\tilde G$ and $\tilde\Delta$ on $(\tilde t_0',t_0)$ such that 
    \begin{equation}\label{e:tildeGDelta}
        \e^{-\tilde\Delta(t)}+\e^{-\tilde G(t)}=1
    \end{equation}
     and such that they satisfy
    \begin{align}
        \tilde G''(t)&=\theta\left(\frac{t-\tilde t_0'}{\eps}\right)\left(\frac{V''(t)}{\e^{\tilde \Delta(t)}-1}-\frac{\tilde G'(t)\tilde \Delta'(t)}{1-\e^{-\tilde \Delta(t)}}\right),\label{e:ODEtildeG}
        \\
        \tilde G(t_0)&=G_0>0,\quad \tilde G'(t_0)=G_0'>0.
    \end{align}
    As $\tilde\Delta$ is a function of $\tilde G$, \eqref{e:ODEtildeG} is a backward ODE for $\tilde G$, explicitly
    \begin{equation}\label{e:ODEtildeG2}
         \tilde G''(t)=\theta\left(\frac{t-\tilde t_0'}{\eps}\right)\left((\e^{\tilde G(t)}-1)V''(t)+\frac{\e^{\tilde G(t)}}{\e^{\tilde G(t)}-1}(\tilde G'(t))^2\right). 
    \end{equation}
    This ODE has locally Lipschitz coefficients and hence a unique solution as long as $\tilde G(t)>0$ on $(\tilde t_0',t_0)$. 
    For $\tilde t_0'$ close to $t_0'$ and $\eps\ll1$, \eqref{e:ODEtildeG2} is a perturbation of \eqref{extGback}, and the solution $G$ of \eqref{extGback} 
is strictly positive on $(t_0',t_0$ by \eqref{eq:t0prime2}. So for $(\eps,\tilde t_0')$ close enough to $(0,t_0')$, also \eqref{e:ODEtildeG2} has a well-defined, strictly positive solution on $(\tilde t_0',t_0)$. Moreover, for fixed $\tilde t_0'$, $\tVn$ and $\tVn'$ converge to $\Vn$ and $\Vn'$ uniformly on $(\tilde t_0',t_0)$.

Having defined $\tilde G$, define $\tilde\Delta$ via \eqref{e:tildeGDelta}, and set $\tVn=V+\tilde\Delta$, $\tilde W=V+\tilde G$.
The key point now is that using the analogue of \eqref{eq:Deltapp} for $\tilde G,\tilde\Delta$ as well as \eqref{e:ODEtildeG} and the fact that $\theta=1$ near 0, we compute that
\[\tVn''(\tilde t_0')=V''(\tilde t_0')+\tilde\Delta''(\tilde t_0')=V''(\tilde t_0')-\tilde G''(\tilde t_0')(\e^{\tilde\Delta(\tilde t_0')}-1)
    -\tilde G'(\tilde t_0')\tilde\Delta'(\tilde t_0')\e^{\tilde \Delta(\tilde t_0')}=V''(\tilde t_0')-V''(\tilde t_0')=0.
    \]
We also know that $\Vn'$ has a zero at $t_0'$ and that zero is non-degenerate by \eqref{backode1}, and so we can find $(\eps,\tilde t_0')$ arbitrarily close to $(0,t_0')$ such that $\tVn'(\tilde t_0')=0$. Then if we extend $\tVn$ as a constant to $(0,\tilde t_0')$, the resulting function will be in $C^2$. We claim that for $\eps$ small enough $\tVn$ will be the desired modification of $\Vn$, and it has all required properties.
First of all, convexity of $\tVn$ on $[0,t_0)$ follows from
\begin{align*}
    \tVn''(t)&=V''(\tilde t)-\tilde G''(t)(\e^{\tilde\Delta(t)}-1)
    -\tilde G'(t)\tilde\Delta'(t)\e^{\tilde \Delta(t)}\\
    &=\left(1-\theta\left(\frac{t-\tilde t_0'}{\eps}\right)\right)\left(V''(t)-\tilde G'(t)\tilde\Delta'(t)\e^{\tilde\Delta(t)}\right)\\
    &=\left(1-\theta\left(\frac{t-\tilde t_0'}{\eps}\right)\right)\left(V''(t)+(\tilde\Delta'(t))^2\frac{\e^{\tilde\Delta(t)}}{\e^{\tilde\Delta(t)}-1}\right)\ge0,
\end{align*}
for $t\in(\tilde t_0',t_0)$, where we used \eqref{e:ODEtildeG} and the analogues of \eqref{eq:Deltap}, \eqref{eq:Deltapp}. Next, the uniform convergence ensures that for $\eps$ small enough we have the analogue of \eqref{eq:t0prime2}, and hence also the analogue of \eqref{eq:Wpp_upper1b}. Estimate \eqref{eq:W1} for $\tVn$ follows from the corresponding (strict) inequalities for $\Vn$ shown above and the uniform convergence. The estimate \eqref{eq:W2} directly carries over from $\Vn$ to $\tVn$. Estimate \eqref{eq:Wpp_upper2} follows from the same argument as in the previous step. Finally, convexity of $\tilde W$ follows from $\tilde W''=V''+\tilde G''$ on $(\tilde t_0',t_0)$, and for \eqref{eq:Vpp_upper2} we can use an explicit computation very similar to \eqref{112}.

    
\end{proof}

\subsection{Proof of the general decomposition result}

To prove Theorem \ref{th:decom_general}, we have to show how the previous lemma can be used to construct a decomposition of a non-convex potential into convex functions. Here we essentially iterate the construction from the lemma. 
The main difficulty, however, is to ensure that we continue to make progress, i.e., that the remainder terms $\W$ from the application of Lemma~\ref{le:decomp_two}
do not converge to some finite function and to 
show that the bounds do not blow up too much under the iteration.

\begin{proof}[Proof of Theorem \ref{th:decom_general}]

    The strategy to prove this theorem is to repeatedly apply Lemma~\ref{le:decomp_two}, thereby subtracting off a convex potential in each step.
    Let us set $U_0=V$.
    Given a non-convex potential $U_i$ we define $U_{i+1}$
    such that the statement in Lemma~\ref{le:decomp_two} holds where we set $\Vn$ equal to $U_{i+1}$ and $V$ equal to $U_i$ and we  define 
    \begin{align}
        V_{i+1}=W(U_i, U_{i+1}).
    \end{align}
    If $U_i$ is convex we set $V_{i+1}=U_i$ and terminate the decomposition.
    Then the definition of $\Vn$ and a simple induction implies that
    \begin{equation}\label{e:dec_uptoN}
        \e^{-U_0} = \sum_{i=1}^{N} \e^{-V_i}+\e^{-U_N}.
    \end{equation}
\textbf{Step 1: Inductive construction, and properties of the $U_i$}\\

    Our goal is to show that we can define $U_i$ for all $i$ (or there is a convex $U_i$) and that we have bounds on their derivatives. More precisely, we claim that the functions $U_i$ have the following 
    properties which we prove by induction.
    \begin{itemize}
        \item The functions $U_i$ are well-defined, increasing on $(0,\infty)$, and
        satisfy $U_{i+1}\geq U_i$.
         \item The functions $U_i$ satisfy for $t\geq t_i$
        the bounds
        \begin{align}
            -U''_i(t)\leq M \min(U_i'(t),1).
        \end{align}
        \item Let $t_i$ be the largest real number such that the function $U_i$ is convex in $(0,t_i)$.
        Recall that $U_0=V$ is convex on $(-t_0,t_0)$ and $U_0'(t_0)\geq \eps$
        with $1/2\geq \eps>0$. We claim $t_i> t_{i-1}>0$ for $i\geq 1$ and 
        \begin{align}\label{eq:lower_ti}
            t_i\geq \frac{i-1+\log_2(\eps)}{32M}.
        \end{align}
        \item The potential $U_i$ satisfies $U_i'(t)>0$ on $(t_i,\infty)$ and 
        \begin{align}\label{eq:lowerWip}
            U'_i(t_i)\geq \eps\e^{-\log(2)32Mt_i}.
        \end{align}
       
        \item The functions $U_i$ satisfy the upper bounds
        \begin{align}\label{eq:Wipp_upper}
            U_i''(t)\leq U_0''(t)+1000M.
        \end{align}
    \end{itemize}
    First we note that these conditions hold
    for $i=0$. Suppose they hold for some $i$. Then $U_i$ satisfies the assumptions of Lemma~\ref{le:decomp_two} and we define
    $U_{i+1}$ as $\Vn$ in the statement.
    Then we directly find that the first and second item hold for $U_{i+1}$.
    
    Next we show the bounds for $t_i$ and $U_i'(t_i)$ where we assume the bounds for $j<i$. 
    Recall that by \eqref{eq:two_cases} in  Lemma~\ref{le:decomp_two} 
    we find that each pair $(U_j, U_{j+1})$ for $j\geq 0$ satisfies at least one of the two cases, i.e.,
       \begin{align}\label{eq:two_cases2}
           \left( t_{j+1}\geq t_{j}+(16M)^{-1}
            \text{ and } U_{j+1}'(t_{j+1})\geq \frac{U_{j}'(t_{j})}{2}
            \right)
            \quad \text{ or }\quad 
           \Big( U_{j+1}'(t_{j+1})\geq 2U_{j}'(t_{j})\text{ and }
           U_{j}'(t_{j})\leq 1\Big).
        \end{align}
     If for each $0\leq j < i$ the first case occurs for $(U_j, U_{j+1})$, then we have $t_i\ge \frac{i}{16M}$ and $U_i'(t_i)\ge 2^{-i}U_0'(t_0)\ge 2^{-i}\eps$, which easily imply \eqref{eq:lower_ti} and \eqref{eq:lowerWip}. So the interesting case is when there is some $j< i$ for which the second case (and not the first) occurs for $(U_j,U_{j+1})$. Let $j_0< i$ be the largest among these indices $j$. 

Then for indices $j$ with $j_0<j< i$ the first case from \eqref{eq:two_cases} occurs, and hence 
\begin{align}
    t_i-t_{j_0}\ge \frac{i-j_0-1}{16M}
\end{align}and $U_i'(t_i)\ge 2\cdot2^{-(i-j_0-1)}U_{j_0}'(t_{j_0})=2^{-(i-j_0-2)}U_{j_0}'(t_{j_0})$. Moreover, as the second case occurs for $U_{j_0}$, we know $U'_{j_0}(t_{j_0})\leq 1$. Suppose that out of the pairs $(U_j,U_{j+1})$ with $j<j_0$, the first case occurs for $k$ of them
(so that $U'_{j+1}(t_{j+1})\geq U_{j}'(t_{j})/2$ for those $k$ values)
and  only the second case for the remaining $j_0-k$ indices (so that $U'_{j+1}(t_{j+1})\geq 2U_{j}'(t_{j})$ for those).  Therefore, we get
\begin{equation}\label{e:estU_j0'}
     1\geq    U'_{j_0}(t_{j_0})\geq 2^{j_0-k}2^{-k}U_0'(t_0)\ge 2^{j_0-2k}\eps,
 \end{equation}
and hence 
\[k\ge \frac{j_0+\log_2(\eps)}{2}.\]
Since the first case occurs for $k$ indices below $j_0$, we have $t_{j_0}-t_0\ge k/16M$ and there holds
\begin{equation}\label{e:estt_i}
t_i=(t_i-t_{j_0})+(t_{j_0}-t_0)\ge \frac{i-j_0-1}{16M}+\frac{k}{16M}=\frac{i+k-j_0-1}{16M}.
\end{equation}
The last two displays now on the one hand imply
\[t_i\ge\frac{2i-j_0-2+\log_2(\eps)}{32M}\ge\frac{i-1+\log_2(\eps)}{32M},\]
    which is \eqref{eq:lower_ti}, and on the other hand together with \eqref{e:estU_j0'} they imply
\[U_i'(t_i)\ge 2^{-(i-j_0-2)}U_{j_0}'(t_{j_0})\ge 2^{-(i-j_0-2)}2^{j_0-2k}\eps=2^{2j_0+2-i-2k}\eps\ge 2^{-2(i+k-j_0-1)}\eps\ge 2^{-32Mt_i}\eps,\]
    which yields \eqref{eq:lowerWip}. 

    Finally, we consider the upper bound on $U_j''(t)$
    for some fixed $j$. To provide some intuition, we first note that using the telescoping expansion and \eqref{eq:Wpp_upper2} we get
    \begin{align}\label{eq:Wtelescope_mot}
        U_j''(s)-U_0''(s)&=
        \sum_{i=0}^{j-1} U''_{i+1}(s)-U''_i(s)
        \leq 48M \sum_{i=0}^{j-1} \e^{-4M(s-t_i)} \min(U_i'(t_i),1) \bs{1}_{s\geq t_i-(16M)^{-1}} .
    \end{align}
    Now the sum would be summable due to the exponential decay if the $t_i$ do not accumulate too much close to $s$ (e.g., if they are well separated). While this is not necessarily true, we can show that if the $t_i$ accumulate then the factors $U_i'(t_i)$ become exponentially small and the sum can still be controlled. We now make this rigorous.

    We fix an $s>0$ and then we consider the intervals 
    $I_k$ for $k\geq 0$, where 
    \begin{align}\label{eq:def_Ik}
        I_k=(s+(16M)^{-1}-(k+1)(16M)^{-1},
        s+(16M)^{-1}-k(16M)^{-1}]
        \eqdef (s_{k+1},s_k].
    \end{align}
    Denote by $i_k$ the largest $i$ such that $t_{i_k}\leq s_k$, i.e., the last $t_{i_k}\in I_k$ (or the largest $i_k$ such that $t_{i_k}$ is to the left of $I_k$
    if no $t_i$ is in $I_k$). Note that $I_{k+1}$ is to the left of $I_k$ and thus 
    $t_{i_k}>t_{i_{k+1}}$ is decreasing but this is irrelevant in the following.
    We claim next that 
    \begin{align}\label{eq:bound_by3}
        \sum_{i:t_i\in I_k} \min(U_i'(t_i),1)
        \leq 3\min(U'_{i_k}(t_{i_k}),1)\leq 3.
    \end{align}
    If $I_k$ has at most one element, this estimate is trivial. Otherwise, note that the definition of $I_k$ and the case distinction in \eqref{eq:two_cases2} imply that the first case can only occur for $(U_{i_k},U_{{i_k}+1})$ but not for any smaller $i\in I_k$ (otherwise $t_{i+1}\geq t_i+(16M)^{-1}\notin I_k$).  In particular, this implies
    $U'_i(t_i)\leq 2^{-1}U'_{i+1}(t_{i+1})$ if $i\in I_k$ and $i<i_k$. Moreover, because the second case occurs for $i_k-1$, we know $U_{i_k-1}'(t_{i_k-1})\leq 1$ and then
    \begin{align}
   U'_i(t_i)\le 2^{-(i_k-1-i)}U'_{i_k-1}(t_{i_k-1})\leq 2^{-(i_k-1-i)}\min(2^{-1}U'_{{i_k}}(t_{i_k}),1)\leq 2^{-(i_k-1-i)}\min(U'_{{i_k}}(t_{i_k}),1)
    \end{align}
     if $i\le i_k-1$ is in $I_k$. 
     We conclude from the above that 
     \begin{align}
         \sum_{i:t_i\in I_k} \min(U_i'(t_i),1)
         \leq  \min(U_{i_k}'(t_{i_k}),1)+ \sum_{j=0}^\infty 2^{-j}\min(U_{i_k}'(t_{i_k}),1)=3\min(U_{i_k}'(t_{i_k}),1)\leq 3.
     \end{align}
     We now start from \eqref{eq:Wtelescope_mot},
     use that $s\geq t_i-(16M)^{-1}$ iff
     $t_i\in I_k$ for some $k$ to rewrite as a sum over $I_k$ 
     and then apply \eqref{eq:bound_by3} and  find the bound 
    \begin{align}
    \begin{split}\label{eq:telescopeW}
        U_j''(s)-U_0''(s)
        &
        \leq 48M \sum_{i=0}^{j-1} \e^{-4M(s-t_i)} \min(U_i'(t_i),1) \bs{1}_{s\geq t_i-(16M)^{-1}}
        \\
        &\leq 
        48M \sum_{k=0}^\infty 
         \e^{-4M(s-s_k)} \sum_{i\in I_k}\min(U_i'(t_i),1) 
         \leq 3\cdot 48M \sum_{k=0}^\infty \e^{4M((16M)^{-1}-(16M)^{-1}k}
         \\&
         \leq 144 M \e^{1/4}\sum_{k=0}^\infty \e^{-k/4}\leq 1000 M.
        \end{split}
    \end{align}
     This ends the proof of \eqref{eq:Wipp_upper}, and completes the inductive construction of the $U_i$.

     \textbf{Step 2: Bounds for $V_i''$}\\
     The next step is to bound the derivatives of the functions $V_i$.  If $U_i'(t_i)\geq \eps$
     (recall that we assumed $U'_0(t_0)=V'(t_0)\geq \eps$ and $\eps\leq 1/2$ in the statement of the theorem), then 
     we find from \eqref{eq:Vpp_upper2} the bound (and noting that by \eqref{eq:Vpp_upper1} the bound holds for all $t$ when considering the positive part of the second derivative) 
     \begin{align}
     \label{121bound}
     0\leq   V_{i+1}''(t) \leq 
     \frac{64M}{\eps} (U_i'')_+(t)+
    {(64M)^2}\leq \frac{64M}{\eps} (V'')_+(t) + C\frac{M^2}{\eps}
     \end{align}
     where the last step used \eqref{eq:Wipp_upper} and $C= 64^2+64000$ is a constant.
     Now suppose that $U'_i(t_i)<\eps\leq 1/2$ so by assumption $i\neq 0$. 
     Then $U_{i-1}'(t_{i-1})<2U_i'(t_i)< 1$ and thus by \eqref{eq:Wpp_upper1b} 
     \begin{align}\label{eq:Wipp_zero}
     U''_i(t)=0\quad \text{for $t<t_{i-1}-(16M)^{-1}$.}
     \end{align}
    Now we can estimate  using $\eqref{eq:lowerWip}$ for $t\geq t_{i-1}-(16M)^{-1}$
    \begin{align}
        U'_{i}(t_i)\geq \frac{U'_{i-1}(t_{i-1})}{2}
        \geq \frac{\eps}{2}\e^{-\log(2)32Mt_{i-1}}    
        \geq \frac{\eps}{2}\e^{-\log(2)32M(t+(16M)^{-1})}
        \geq \frac{\eps}{8}\e^{-\log(2)32Mt}.
    \end{align}
     Combining this with the bound \eqref{eq:Vpp_upper2} followed by \eqref{eq:Wipp_upper} we obtain 
     for $t\geq t_{i-1}-(16M)^{-1}$
     \begin{align}
     \begin{split}
      V_{i+1}''(t)&\leq \frac{ 64M}{U'_{i}(t_{i})}U_i''(t)+(64M)^2
      \\
      &\leq \frac{ 8\cdot 64M}{\eps} \e^{\log(2)32Mt}(U_0''(t)+1000M)+(64M)^2
         \\
         &\leq 
         \frac{C_2M}{\eps}\e^{C_1Mt}(V'')_+(t)
         +\frac{C_3M^2}{\eps}\e^{C_1Mt}.
         \end{split}
     \end{align}
     But since $V_{i+1}''(t)=0$ for $t\leq t_{i-1}-(16M)^{-1}$ the last display is actually true for all $t$ and together with \eqref{121bound} we thus proved \eqref{eq:Vpp_final1}.

\textbf{Step 3: Convergence}\\
     Next, we prove that the identity 
     $\e^{-V}=\e^{-U_0}=\sum_i \e^{-V_i}$ holds, i.e. that we can pass to the limit $N\to\infty$ in \eqref{e:dec_uptoN}. For that purpose we need to show that $U_i(t)\to \infty$ as $i\to \infty$ pointwise for all $t$.
     Since by the first property of the $U_i$, we know that  $U_i$ is monotone increasing on $(0,\infty)$ and $U_{i+1}\geq U_i$
     so 
     it is sufficient to show that $U_i(0)$ is unbounded. 
     Assume that $U_i'(t_i)\geq 1$ holds infinitely often.
     Then along a subsequence $i_j$ of such $i$ we find by \eqref{eq:W0} that
     \begin{align}
         U_{i_j+1}(0)\geq U_{i_j}(0)+\frac{1}{64M}.
     \end{align}
     This implies (by monotonicity of $U_i(0)$) that $U_i(0)\to \infty$.
     On the other hand,  if $U_i'(t_i)< 1$
     holds infinitely often 
     we can apply  \eqref{eq:W1}  along this subsequence $i_j$ to find
     \begin{align}
         U_{i_j+1}(0)\geq U_{i_j}(t_{i_j})-\frac{1}{32M}
         \geq U_0(t_{i_j})-\frac{1}{32M}.
     \end{align}
     This (together with the monotonicity of $U_i(0)$) implies that $U_{i+1}(0)\to \infty$
     because \eqref{eq:lower_ti} implies that $t_i\to\infty $, and we know that $\lim_{t\to \infty} U_0(t)=\lim_{t\to \infty} V(t)=\infty$.

\textbf{Step 4: Improved bounds for $V_i''$}\\   
    Finally, we prove the stronger upper bound on $V_i''$ under the additional  assumption \eqref{eq:as_eps_lower} on $V$.
    The key idea is to show that $U'_j(t_j)\ge c\eps$
    which then implies the desired bound.
    We
    first control the difference $U_{j+1}(t_j)-U_0(t_j)$. This is very similar to our derivation of the bound for $U_j''$, i.e., we decompose into a telescoping sum, apply \eqref{eq:W2} to bound all terms and then group the terms to control the sum. 
    Similar to \eqref{eq:def_Ik}  we now consider
    for $k\geq 0$ the intervals 
    \begin{align}
        J_k=(t_j-(k+1)(16M)^{-1},
        t_j-k(16M)^{-1}].
    \end{align}
    As before we define $i_k$ to be the largest index such that $t_{i_k}\leq \max(J_k)$
    and again $t_{i_k}$ is decreasing since $J_{k+1}$ is to the left of $J_k$. 
    We now claim that for $l\geq k$
    \begin{align}\label{eq:claimWi}
        U'_{i_{l}}(t_{i_{l}})\leq 
        2^{l-k}U'_{i_k}(t_{i_k}).
    \end{align}
    It is sufficient to show the claim for $l=k+1$ because the general case follows by induction.
    Note that for $i_{k+1}<i<i_k$ we have $t_{i+1},t_i\in J_k$ and thus $t_{i+1}-t_i<(16M)^{-1}$ and therefore \eqref{eq:two_cases2} implies that
    $U_i'(t_i)\leq U_{i+1}'(t_{i+1})/2$.
    Combining this with $ U'_{i_{k+1}}(t_{i_{k+1}})
    \leq 2 U'_{i_{k+1}+1}(t_{i_{k+1}+1})$ (this holds in either  case of \eqref{eq:two_cases2})
    we find  $ U'_{i_{k+1}}(t_{i_{k+1}})\leq 
        2U'_{i_k}(t_{i_k}) $ and \eqref{eq:claimWi} follows.
    Note that \eqref{eq:claimWi} implies using $i_0=j$ that
    \begin{align}\label{eq:claimWi2}
        U'_{i_k}(t_{i_k})\leq 2^k U'_{i_0}(t_{i_0})=2^k U'_{j}(t_j).
    \end{align}
        
    As in \eqref{eq:bound_by3} we find that
    \begin{align}\label{eq:bound_by3_2}
      \sum_{i:t_i\in J_k} \min(U_i'(t_i),1)
        \leq 3\min(U'_{i_k}(t_{i_k}),1).
    \end{align}
    The same approach as in \eqref{eq:telescopeW}
    implies using the bound \eqref{eq:W2}  that
    \begin{align}\label{eq:boundWj-W0}
    \begin{split}
        U_{j+1}(t_j)-U_0(t_j)
        &=\sum_{i=0}^{j} U_{i+1}(t_j)-U_{i}(t_j)
        \leq \sum_{i=0}^{j} \frac{\min(1,U_i'(t_i))}{32M}\e^{-16M(t_j-t_i)}
        \\
        &\leq (32 M)^{-1}
        \sum_{k\geq 0}\e^{-16M(t_j-(t_j-k/(16M))}\sum_{i: t_i\in J_k} {\min(1,U_i'(t_i))}
        \\
        &\leq (32 M)^{-1}
        \sum_{k\geq 0}\e^{-k}3\min(U_{i_k}'(t_{i_k}),1)
        \\
        &\leq \frac{3}{32 M}
        \sum_{k\geq 0}\e^{-k}2^k U_{j}'(t_j)
        \leq \frac{1}{2M}U_{j}'(t_j).
    \end{split}
    \end{align}
    Here we used \eqref{eq:bound_by3_2} to remove the sum over $J_k$ in the third inequality, and then applied \eqref{eq:claimWi2}.
    In the last step we evaluated the geometric sum and bounded the resulting constant.
    
    We now show how this estimate can  be used to control $U_{j}'(t_{j})$, more concretely, we show by induction that
    \begin{align}\label{eq:Up_lower}
        U_{j}'(t_j)\geq \frac{\eps}{17}.
    \end{align}

    By assumption this bound holds for $j=0$
    establishing the induction base. 
    Now we assume the bound holds for $i\leq j$
    and show it for $j+1$.
    If the second case in \eqref{eq:two_cases2}
    holds for $(U_j,U_{j+1})$, we infer that then
    $U_{j+1}'(t_{j+1})\geq 2U_j'(t_j)\geq \eps/17$ which is the desired bound.
    
    It remains to consider the case where the first statement of \eqref{eq:two_cases2} holds.
     Then convexity of $U_{j+1}$ on $(0,t_{j+1})$ for the first inequality below, followed by $U_{j+1}>U_0$ for the second inequality, 
     and the monotonicity assumption \eqref{eq:as_eps_lower} applied to $U_0$ for the last inequality, imply that
    \begin{align}\label{eq:lower_boundWp2}
    \begin{split}
        U_{j+1}(t_j)&\geq U_{j+1}(t_{j+1})
        -U'_{j+1}(t_{j+1})(t_{j+1}-t_j)
        \geq U_0(t_{j+1}) -U'_{j+1}(t_{j+1})(t_{j+1}-t_j)
        \\
        &\geq U_0(t_{j}) -U'_{j+1}(t_{j+1})(t_{j+1}-t_j)
        +\eps(t_{j+1}-t_j).
        \end{split}
    \end{align}
Combining \eqref{eq:boundWj-W0} and \eqref{eq:lower_boundWp2} we obtain
\begin{align}
\begin{split}\label{eq:dtbound}
  \eps(t_{j+1}-t_j)&\leq   U'_{j+1}(t_{j+1})(t_{j+1}-t_j)+U_{j+1}(t_{j})-U_0(t_j)
   \\
   &\leq U'_{j+1}(t_{j+1})(t_{j+1}-t_j)+\frac{1}{2M}U_j'(t_j).
    \end{split}
\end{align}

Since the first case in \eqref{eq:two_cases2} holds for $(U_j,U_{j+1})$, we find $t_{j+1}-t_j\geq (16M)^{-1}$ and $U'_j(t_j)\leq 2U_{j+1}'(t_{j+1})$.
Dividing \eqref{eq:dtbound} by $(t_{j+1}-t_j)$, we get 
\begin{align}
\begin{split}
    \eps &\leq U'_{j+1}(t_{j+1})+\frac{1}{2M(t_{j+1}-t_j)}U_j'(t_j)
    \leq U'_{j+1}(t_{j+1})+\frac{2}{2M(t_{j+1}-t_j)}U'_{j+1}(t_{j+1})
    \\
    &\leq
     U'_{j+1}(t_{j+1})+16U_{j+1}'(t_{j+1}).
    \end{split}
\end{align}
   We conclude that indeed $U'_{j+1}(t_{j+1})\geq \eps/17$
   in this case finishing the proof of \eqref{eq:Up_lower}.
   
    Proving \eqref{eq:Vpp_final2} with \eqref{eq:Up_lower} at hand is a direct consequence of \eqref{eq:Vpp_upper2}
    and \eqref{eq:Wipp_upper},
    which imply that
    \begin{align}
        V_{i+1}''(t)\leq
        \frac{17\cdot 64M}{\eps}(U_{i}'')_+(t)+
        (64M)^2
         \leq  \frac{17\cdot 64M}{\eps}(U_{0}'')_+(t)+
         \frac{CM^2}{\eps}.
    \end{align}
    This establishes \eqref{eq:Vpp_final2} and ends the proof.
\end{proof}

\section{Extended gradient Gibbs measures}
\label{extendedgradGibbs}
In this section we introduce the required background and notation for interface models. In particular, Section~\ref{s:definitions} introduces formal definitions of gradient Gibbs measures, while Section~\ref{sec:extended_Gibbs} shows how decompositions of potentials as in Theorem \ref{t:decomp} give rise to so-called extended Gibbs measures. Additionally, we relate these extended Gibbs measures to disordered Gibbs measures in Section~\ref{sec:disordered_Gibbs}.
\subsection{General definitions and notation for gradient models}\label{s:definitions}

Let $\Lambda$ be a finite set in $\zd$ with boundary 
\[
\partial\Lambda:=\{x\notin\Lambda,~|x-y|=1~\mbox{for some}~y\in\Lambda\},~\mbox{where}~|x-y| =\sqrt{\sum (x_i-y_i)^2}~\mbox{for}~x,y\in\Z^d,
\]
and with given boundary condition $\psi$ such that $\phi(x)=\psi(x)$ for $x\in\partial\Lambda$; a special case of such boundary conditions is the {\it tilted} boundary conditions, with $\psi(x)=x\cdot u$ for all $x\in\partial\Lambda$, and where $u\in\RR^d$ is fixed. 

Let $V\in C^2(\R)$ be even (i.e. $V(-s)=V(s)\ \forall s\in\R)$. We define the Hamiltonian $H_{\Lambda,\psi}$ on $\Lambda$, of gradient type, by 
\begin{equation}
\label{eqn00}
H_{\Lambda,\psi}(\phi)=\frac{1}{2}\sum_{\substack{x,y\in\Lambda\\ |x-y|=1}}V(\phi(x)-\phi(y))+\sum_{\substack{x\in\Lambda, y\in\partial\Lambda\\ |x-y|=1}}V(\phi(x)-\psi(y)).
\end{equation}
The factor $\frac12$ in front of the bulk term accounts for the double-counting of the internal edges.

Other normalisations are common in the literature as well (e.g. summing just once over each undirected edge), but these just correspond to multiplying $V$ by a factor.
\medskip

Let $\bar\Lambda:=\Lambda\cup\partial\Lambda$ and let $\d\phi_{\Lambda}=\prod_{x\in\Lambda}\d\phi(x)$ be the Lebesgue measure over $\RR^{\Lambda}$. 
For $A\subset\zd$, we shall denote by $\F_{A}$ the $\sigma$-field generated by $\{\phi(x):x\in A\}$. 

\begin{definition}  {\bf (Finite volume $\phi$-Gibbs measure on $\zd$)}
\label{finitegibbs}
For a finite region $\Lambda\subset\zd$, {\it the finite volume Gibbs measure $\nu_{\Lambda,\psi}$ on $\RR^{\zd}$ with boundary condition $\psi\in \RR^{\zd}$ }for the field of \textit{height variables} $(\phi(x))_{x\in\zd}$ over $\Lambda$ is defined for $\beta>0$ by
\begin{equation}
\label{finitevolgibbs}
\nu_{\Lambda,\psi}(\ormd\phi)=\frac{1}{Z_{\Lambda,\psi}}
\exp\left(-\beta H_{\Lambda,\psi}(\phi)\right)\d
\phi_\Lambda\delta_\psi(\d\phi_{\zd\setminus\Lambda}),
\end{equation}
where
\[Z_{\Lambda,\psi}=\int_{\RR^\zd}\exp\left\{-\beta H_{\Lambda,\psi}(\phi)\right\}\d\phi_\Lambda\delta_\psi(\d\phi_{\zd\setminus\Lambda}),\]
and where $\delta_\psi(\d\phi_{\zd\setminus\Lambda})=\prod_{x\in\zd\setminus\Lambda}\delta_{\psi(x)}(\d\phi(x))$ determines the boundary condition.
\end{definition}

\begin{definition}  {\bf (Infinite volume $\phi$-Gibbs measure on $\zd$)}
\label{infgibbs}
A probability measure $\nu\in P(\RR^{\zd})$ is called a Gibbs measure for the $\phi$-field with given Hamiltonian $H:=(H_{\Lambda, \psi})_{\Lambda\subset\zd, \psi \in\RR^{\zd}}$ ($\phi$-Gibbs measure for short), if its
conditional probability given $\F_{\Lambda^c}$ satisfies the DLR equation
$$\nu(\,\cdot\,|\F_{\Lambda^c})(\psi)=\nu_{\Lambda,\psi}(\cdot),~~\nu-\mbox{a.e. }\psi,$$
for every finite $\Lambda\subset\zd$.
\end{definition}

It is known that an infinite volume $\phi$-Gibbs measure exists under condition $0<c_V<V''(s)$ for every $s\in\mathbb{R}$ in dimension $d\ge3$,
but not for $d=1,2$, where the field ``delocalises'' as $\Lambda\nearrow\zd$. For a proof of this statement for strictly convex potentials satisfying $V''\le C$, see for example Section 1.2.3 from \cite{MR2216964}; for a much more general delocalisation result, see Theorem 1.1 from \cite{MR3395146}. One way to bypass this problem is to consider the so-called gradient Gibbs measures instead, which exist in all dimensions $d\ge 1$. To define them, we need to set up more notation.

\subsubsection*{Notation for the Bond Variables on $\zd$}

For most of the paper we will work with directed bonds. Accordingly, we define \[\bzd:=\{b=(x_b,y_b)~|~x_b,y_b\in\zd,|x_b-y_b|=1,b~\mbox{directed from}~x_b~\mbox{to}~y_b\};\] 
note that each undirected bond appears twice in $\bzd$. Let 
\[\bzdl:=\bzd\cap (\Lambda\times\Lambda),~\partial\bzdl:=\{b=(x_b,y_b)~|~x_b\in\zd\setminus\Lambda,y_b\in\Lambda,|x_b-y_b|=1\}\]
and 
\[\bzdlb:=\{b=(x_b,y_b)\in\bzd~|~x_b\in\Lambda~\mbox{or}~y_b\in\Lambda\}.\]

At one point (when introducing the $\kappa$-disorder variables) it will be more convenient to work with undirected bonds. So we also define \[\ubzd=\left\{\{x_b,y_b\}\mid b=(x_b,y_b)\in\bzd\right\}\]
and analogously for $\ubzdl,\ubzdlb$.

For $\phi=(\phi(x))_{x\in\zd}$ and $b=(x_b,y_b)\in\bzd$, we define the \textit{height differences} $\nabla\phi(b):=\phi(y_b)-\phi(x_b)$.
The height variables $\phi=\{\phi(x);x\in\zd\}$ on $\zd$ automatically determine a field of height differences
$\nabla\phi=\{\nabla\phi(b);b\in\bzd\}$. One can therefore consider the distribution $\mu$ of the $\nabla\phi$-field
under the $\phi$-Gibbs measure $\nu$. We shall call $\mu$ the $\nabla\phi$-Gibbs measure. In fact, it is possible to define
the $\nabla\phi$-Gibbs measures directly by means of the DLR equations and, in this sense, $\nabla\phi$-Gibbs measures exist for
all dimensions $d\ge1$.

A sequence of bonds $\C=\{b^{(1)},b^{(2)},\ldots,b^{(n)}\}$ is called a \textit{chain} connecting $x$ and $y$, $x,y\in\zd$, if $x_{b_1}=x,y_{b^{(i)}}=x_{b^{(i+1)}}$ for $1\le i\le n-1$ and $y_{b^{(n)}}=y$. The chain is called a \textit{closed loop} if $y_{b^{(n)}}=x_{b^{(1)}}$. A \textit{plaquette} is a closed loop $\A=\{b^{(1)},b^{(2)},b^{(3)},b^{(4)}\}$ such that $\{x_{b^{(i)}},i=1,\ldots,4\}$ consists of four different points. 

The field $\eta=\{\eta(b)\}\in\RR^{\bzd},b\in\bzd,$ is said to satisfy \textit{the plaquette conditions} if
\[
\eta(b)=-\eta(-b)~\mbox{for all}~b\in\bzd~\mbox{and}~\sum_{b\in\A}\eta(b)=0~\mbox{for all plaquettes}~\A~\mbox{in}~\zd,
\]
where $-b$ denotes the reversed bond of $b$. Let 
\begin{equation}
\chi=\{\eta\in\RR^{(\zd)^*}~\mbox{which satisfy the plaquette 
condition}\}
\end{equation}
Then $\nabla\phi=\{\nabla\phi(b)\}$ satisfies the plaquette condition. Conversely, the heights $\phi^{\eta,\phi(0)}\in\RR^\zd$ can be constructed from height differences $\eta$ and the
height variable $\phi(0)$ at $x=0$ as 
\begin{equation}
\label{19}
\phi^{\eta,\phi(0)}(x):=\sum_{b\in\C_{0,x}}\eta(b)+\phi(0),
\end{equation} 
where $\C_{0,x}$ is an arbitrary chain connecting $0$ and $x$. Note that $\phi^{\eta,\phi(0)}$ is well-defined if
$\eta=\{\eta(b)\}\in\chi$.

Given some fixed $r>0$ let $L_r^2$ be the set of all $\eta\in\RR^{\bzd}$ such that $|\eta|^2_r:=\sum_{b\in\bzd}|\eta(b)|^2\e^{-2r|x_b|}<\infty$. Let $\chi_r:=\chi\cap L_r^2$ equipped with the norm $|\cdot|_r$. For $\phi=(\phi(x))_{x\in\zd}$ and $b\in\bzd$, we define $\eta^\phi(b):=\nabla\phi(b)$. We also define $C_b(\chi)$ as the space of continuous bounded functions on $\chi$, and $C^2_{b,\text{loc}}$ as the space of local twice continuously differentiable functions with bounded derivatives up to order 2.

\subsubsection*{Definition of $\nabla\phi$-Gibbs measures}

We next define the finite volume $\nabla\phi$-Gibbs measures. For every $\xi\in\chi$ and finite $\Lambda\subset\zd$ the space of all
possible configurations of height differences on $\bzdlb$ for given
boundary condition $\xi$ is defined as
$$\chi_{\bzdlb,\xi}=\{\eta=(\eta(b))_{b\in\bzdlb};\eta\vee\xi\in\chi\},$$
where $\eta\vee\xi\in\chi$ is 
determined by $(\eta\vee\xi)(b)=\eta(b)$ for $b\in\bzdlb$ and $=\xi(b)$ for
$b\not\in\bzdlb$.

\begin{remark}
\label{equivzd}
When $\zd\setminus\Lambda$ is connected, $\chi_{\bzdlb,\xi}$ is an affine space such that $\dim \chi_{\bzdlb,\xi}=|\Lambda|$. Indeed, fixing a point $x_0\notin\Lambda$, we consider the map $\chi_{\bzdlb,\xi}\rightarrow\RR^{\Lambda}$, such that $\eta\rightarrow\phi=\{\phi(x)\}\in\RR^{\Lambda}$, with $\phi(x)$ defined by
$$\phi(x)=\sum_{b\in C_{x_0,x}}(\eta\vee\xi)(b)$$
for a chain $C_{x_0,x}$ connecting $x_0$ and $x\in\Lambda$. This map is then well-defined and an invertible affine transformation. 
\end{remark}

\begin{definition}{\bf(Finite volume $\nabla\phi$-Gibbs measure)}
\label{finvolgrad}
The finite volume $\nabla\phi$-Gibbs measure in $\bzdlb$ with associated Hamiltonian $H:=(H_{\Lambda, \xi})_{\Lambda\subset\zd,\,\xi \in
\chi}$ and 
with boundary condition $\xi$ is defined by
\[
\mu_{\Lambda,\xi}(\ormd\eta)=\frac{1}{Z_{\Lambda,\xi}}\exp\bigg(-\frac{\beta}{2}\sum_{b\in\bzdlb}V(\eta(b))\bigg)\ormd\eta_{\Lambda,\xi}
\in P(\chi_{\bzdlb,\xi}),
\]
where $\ormd\eta_{\Lambda,\xi}$ denotes the Hausdorff measure on the affine space $\chi_{\bzdlb,\xi}$ and $Z_{\Lambda,\xi}$ is
the normalisation constant.
\end{definition}

 Let $P({\chi})$ be the set of all probability measures on ${\chi}$ and let $P_2({\chi})$ be those 
$\mu\in P({\chi})$ satisfying $\E_{\mu}[|\eta(b)|^2]<\infty$ for each $b\in\bzd$.
\begin{remark}
\label{equivzd1}
For every $\xi\in\chi$ and $a\in\RR$, let $\psi=\phi^{\xi,a}$ be defined by \eqref{19} and consider the measure $\nu_{\Lambda,\psi}$. Then $\mu_{\Lambda,\xi}$ is the image measure of $\nu_{\Lambda,\psi}$ under the map $\{\phi(x)\}_{x\in\Lambda}\rightarrow\{\eta(b):=\nabla(\phi\vee\psi)(b)\}_{b\in\bzdlb}$ and where we defined $(\phi\vee\psi)(x):=\phi(x)$ for $x\in\Lambda$ and $(\phi\vee\psi)(x):=\psi(x)$ for $x\notin\Lambda$.
Note that the image measure is determined only by $\xi$ and is independent of the choice of $a$. Let ${K}^{\psi}_{\Lambda}:\{\phi(x)\}_{x\in\zd}\rightarrow\{\eta(b)\}_{b\in\bzd}$, with $\eta(b):=\nabla(\phi\vee\psi)(b)$.
\end{remark}

\begin{definition} {\bf (Infinite volume $\nabla\phi$-Gibbs measure on $(\zd)^*$)} 
\label{nablaphigib}
The probability measure $\mu\in P(\chi)$ is called a Gibbs measure for the height differences with given Hamiltonian $H:=(H_{\Lambda, \xi})_{\Lambda\subset\zd, \xi \in\chi}$ ($\nabla\phi$-Gibbs measure for short),
if it satisfies the DLR equation
\begin{equation}
\label{dlr}
\mu(\,\cdot\,|\F_{\bzd\setminus\bzdlb})(\xi)=\mu_{\Lambda,\xi}(\cdot),~~\mu-\mbox{a.e. }\xi,
\end{equation}
for every finite $\Lambda\subset\zd$, where $\F_{\bzd\setminus\bzdlb}$ stands for the $\sigma$-field of $\chi$ generated by
$\{\eta(b),b\in\bzd\setminus\bzdlb\}$.
\end{definition}

\begin{remark}
Proving the DLR equation \eqref{dlr} is equivalent to proving that for every finite $\Lambda\subset\zd$ and for all $F\in C_b(\chi)$ we have
\begin{equation}
\label{dlrgrad}
\int_\chi\mu(\d\xi)\int_{\chi_{\bzdlb,\xi}} \mu_{\Lambda,\xi}(\d\eta)F(\eta)=\int_\chi\mu(\d\eta)F(\eta).
\end{equation}
In practice, it is often easier to work with simple functions in \eqref{dlrgrad}. 
\end{remark}

\begin{remark}
Throughout the rest of the paper, we will use the notation $\phi,\psi$ to denote height variables and $\eta,\xi$ to denote height differences.
\end{remark}

\subsubsection*{Temperedness, shift-invariance and ergodicity}

We say that a gradient Gibbs measure is tempered if it is in $P_2(\chi)$, i.e. if $\int\eta(b)^2\mu(\d\eta)<\infty$, $\forall b\in\bzd$. This in particular implies that if $\mu$ is a shift-invariant tempered gradient Gibbs measure, it satisfies $\mu(\chi_r)=1$.

For $x\in\zd$, we define by abuse of notation the shift operators: $\tau_{x}:\RR^{\zd}\rightarrow\RR^{\zd}$ for the heights by $\tau_{x}\phi(y)=\phi(y-x)~\mbox{for}~y\in\zd~\mbox{and}~\phi\in\RR^{\zd}$, and $\tau_{x}:\RR^{\bzd}\rightarrow\RR^{\bzd}$ for the bonds by $(\tau_{x}\eta)(b)=\eta(b-x)$, for $b\in\bzd~\mbox{and}~\eta\in\chi$, where $(b-x)=(x_b-x, y_b-x)\in (\zd)^*$. Then \textbf{shift-invariance} and \textbf{ergodicity} for $\mu$ (with respect to $\tau_x$ for all $x\in \zd$) is defined in the usual way (see for example Definition 2.3 on page 122 in \cite{MR2228384}). We say that the shift-invariant $\mu\in P_2({\chi})$ has a given \textbf{tilt} $u\in\RR^d$ if 
$\mathbb{E}_{\mu}(\eta(b))= \langle u, y_b-x_b \rangle$ for all bonds $b=(x_b,y_b)\in\bzd$.

\begin{remark}
To simplify the notation we will take $\beta=1$ for the rest of the paper, or equivalently replace $V$ by $\beta V$. This is no loss of generality: Note that
if $V$ is $\alpha$-monotone, then $\beta V$ is $\alpha\beta$-monotone, and so $V$ satisfies the assumptions of Theorem~\ref{t:decomp} if and only if $\beta V$ does.

We emphasize, however, that when applying Theorem~\ref{t:decomp} to $\beta V$, the resulting decomposition has a non-trivial $\beta$ dependence, and we are not aware of a construction where a change of $\beta$ results in tractable transformations of the decomposition terms $V_i$.

\end{remark}

\subsection{Definition and basic properties of extended gradient Gibbs measures}
\label{sec:extended_Gibbs}

The key property of mixtures is that, given a (gradient) Gibbs measure, we can treat the parameters $\kappa$ as random variables in their own right, and define an extended (gradient) Gibbs measure whose first marginal is the original Gibbs measure, and whose second marginal is some measure on the $\kappa$. This will allow us, in Sections \ref{s:HS} and \ref{SectScaling} below, to use stochastic homogenisation techniques to obtain results for the extended measure, which apply also to the marginal gradient Gibbs measure. This key idea was introduced in \cite{MR2322690,MR2778801}, and we base our presentation on these works. 

We therefore fix a strongly log-concave log-mixture potential $V$ throughout this section. We want the $\kappa$ to satisfy $\kappa(b)=\kappa(-b)$, and to achieve this, it is convenient to define the extended gradient Gibbs measure in terms of the undirected bonds first. 

Given a gradient Gibbs measure $\mu$ on $\chi$, we can define the extended gradient Gibbs measure $\extmu$ on $\chi\times K^{\ubzd}$ given by 
\begin{align}\label{e:def_ext}
\extmu_\Lambda(A\times B)
=
\int_B \prod_{b\in \ubzdlb}\rho(\d\kappa(b)) \int_A \prod_{b\in \ubzdlb} \exp\left(-V_{\kappa(b)}(\eta(b))+V(\eta(b))\right)\mu(\d\eta),
\end{align}
for $A \in \F_{\ubzdlb}$, $B\in \Kcal^{\otimes \ubzdlb}$ and $\Lambda\subset\zd$ finite, where we recall that $\ubzd$ and $\ubzdlb$ denote the undirected bonds. 

One can check that this is a consistent family of measures in the sense of Kolmogorov's extension theorem, and hence can be extended to a measure $\extmu$ on $\chi\times K^{\ubzd}$. 
Having defined $\extmu$ as a gradient Gibbs measure on $\chi\times K^{\ubzd}$, we can trivially extend it to a gradient Gibbs measure on $\chi\times K^{\bzd}$, which we will continue to denote by $\extmu$. This gradient Gibbs measure is then supported on configurations $\eta,\kappa$ that satisfy $\kappa(b)=\kappa(-b)$ for all $b\in\bzd$. 

The name ``extended gradient Gibbs measure'' is justified by the fact that $(\eta,\kappa)$ is Gibbsian for the (formal) Hamiltonian 
\[H_{\Lambda, \xi}[\kappa](\eta)=\frac{1}{2}\sum_{b\in\bzdlb}V_{\kappa(b)}(\eta(b)).\]
More precisely, similarly to the explanation given in Remark 2.1 from \cite{MR2322690},  $\tilde{\mu}$ is itself a gradient Gibbs measure with associated Hamiltonian $(H_{\Lambda, \xi, \bar\kappa})_{\Lambda\subset\zd, \xi\in\chi, \bar\kappa\in K^{\bzd}}$, and corresponding finite volume gradient Gibbs measure 
  $\tilde \mu_{\Lambda, \xi,\bar\kappa}$ defined 
  for ${\xi} \in \chi$ and $\bar{\k}\in K^{\ubzd}$
  by 
  \begin{equation}\label{eq:specification_gamma_tilde}
  \tilde{\mu}_{\Lambda, \xi, \bar\kappa}(\d \eta,\d \kappa)
 =\frac{\exp\left(-\sum_{b\in\ubzdlb} V_{\kappa(b)}(\eta(b)\right)}{Z_{\Lambda,{\xi}, \bar\kappa}} \d\eta_{\Lambda, \xi}\prod_{b\in\ubzdlb}\rho(\d\kappa(b))\prod_{b\in \ubzd\setminus\ubzdlb}\delta_{\bar\kappa(b)}(\d\kappa(b)),
   \end{equation}
and extended to $\kappa,\bar\kappa\in K^\bzd$ by again requiring $\kappa(b)=\kappa(-b)$ for all $b\in\bzd$.

  Note that the distribution $(\d\eta,\d\kappa)_{\bzdlb}$  in fact only depends on $\xi_{\bzdlb}$ and is independent of $\bar\kappa$,
and therefore the partition function $Z_{\Lambda, \xi,\bar{\k}}$ also only depends on $\xi$, it is in fact equal to $Z_{\Lambda, \xi}$ from Definition \ref{finvolgrad}.
 Observe also that integrating out the $\kappa$ variables, we indeed recover the finite volume Gibbs measure of the gradient field with potential $V$.  We also observe that $\mu$ is the $\eta$-marginal of $\tilde\mu$.

Let $\Fkappa=\sigma(\{\kappa(b)\colon b\in\bzd\})$ and $\Feta=\sigma(\{\eta(b)\colon b\in\bzd\})$. We also denote by $\bar\mu$ the $\kappa$-marginal of $\extmu$. 

As it turns out, many properties of $\mu$ carry over to $\extmu$ and $\bar\mu$. To begin with, we have the following easy lemma.

\begin{lemma}\label{l:extension_tinv_erg}
Let $\mu$ be a shift-invariant ergodic gradient Gibbs measure. Then the associated extended gradient Gibbs measure $\extmu$ is shift-invariant ergodic as well. The same holds true for its $\kappa$-marginal $\bar\mu$.
\end{lemma}
\begin{proof}
Shift-invariance of $\tilde\mu$ is clear from the definition since by \eqref{e:def_ext}, it is uniquely defined in terms of $\mu$, which is shift-invariant with respect to $\eta$. Furthermore, if we take the infinite product over the $\kappa$ then this is shift-invariant with respect to $\kappa$, which suffices to conclude. We leave the details to the interested reader.

To see that $\extmu$ is ergodic, one can proceed as in \cite[Lemma 3.2]{MR2778801}. The crucial observation is that one can show as in \cite[Lemma 3.1]{MR2778801} that $\extmu(\cdot\mid\Feta)$ defines a regular conditional probability, with
\[\extmu(\d\kappa\mid\Feta)(\eta)=\prod_{b\in \ubzd}\exp\left(V(\eta(b))-V_{\kappa(b)}(\eta(b))\right)\rho(\d\kappa(b))\]
(extended to directed bonds as before).
This is a product measure and so it is in particular mixing. From this and using the fact that $\mu$ is ergodic, ergodicity of $\extmu$ easily follows by the same arguments as in \cite[Lemma 3.2]{MR2778801}. 

Now, if $\extmu$ is shift-invariant ergodic, the same holds true for its marginal $\bar\mu$.
\end{proof}

As part of the proof of Lemma \ref{l:extension_tinv_erg}, we computed the law of $\extmu$ conditional on the $\eta$. Next, we will try to characterise the law of $\extmu$ conditional on $\kappa$. However, in order to be able to state the result, we first need to introduce disordered gradient Gibbs measures, relying on the definitions in Section \ref{s:definitions}.

\subsection{Disordered gradient Gibbs measures}
\label{sec:disordered_Gibbs}
Given $\kappa\in K^{\bzd}$ arbitrary, consider the family of Hamiltonians 
\begin{equation}
\label{e:disordered_H}
H_{\Lambda,\psi}[\kappa](\phi)=\frac{1}{2}\sum_{\substack{x,y\in\Lambda\\|x-y|=1}}V_{\kappa((x,y))}(\phi(x)-\phi(y))+\frac{1}{2}\sum_{\substack{x\in\Lambda, y\in\partial\Lambda\\ |x-y|=1}}V_{\kappa((x,y))}(\phi(x)-\psi(y))+\frac{1}{2}\sum_{\substack{x\in\Lambda, y\in\partial\Lambda\\ |x-y|=1}}V_{\kappa((y,x))}(\phi(x)-\psi(y)),
\end{equation}
where
\[Z_{\Lambda,\psi}[\kappa]=\int_{\RR^\zd}\exp\left\{-H_{\Lambda,\psi}[\kappa](\phi)\right\}\d\phi_\Lambda\delta_\psi(\d\phi_{\zd\setminus\Lambda}),\]
and where $\delta_\psi(\d\phi_{\zd\setminus\Lambda})=\prod_{x\in\zd\setminus\Lambda}\delta_{\psi(x)}(\d\phi(x))$ determines the boundary condition.

\begin{definition}\label{d:finitevoldisgibbs}
For a finite region $\Lambda\subset\zd$, {\it the finite volume disordered Gibbs measure $\nu_{\Lambda,\psi}[\kappa]$ on $\RR^{\zd}$ with boundary condition $\psi$ } outside $\Lambda$ is defined by
\begin{equation}
\label{finitevoldisgibbs}
\nu_{\Lambda,\psi}[\kappa](\ormd\phi)=\frac{1}{Z_{\Lambda,\psi}[\kappa]}
\exp\left(-H_{\Lambda,\psi}[\kappa](\phi)\right)\d
\phi_\Lambda\delta_\psi(\d\phi_{\zd\setminus\Lambda}).
\end{equation}
As in Section \ref{s:definitions}, this finite volume disordered Gibbs measure allows us to define infinite volume Gibbs measures $\nu[\kappa]$, finite volume disordered gradient Gibbs measures and infinite volume gradient Gibbs measures $\mu[\kappa]$. 
This allows us to define shift-covariant random/disordered (gradient) Gibbs measures.
\end{definition}
 
\begin{definition} {\bf(shift-covariant disordered (gradient) Gibbs measures)}
\label{shiftcov1}
\begin{itemize}
\item [(a)] Let $\P$ be a probability measure on $(K^{\bzd},\Kcal^{\otimes\bzd})$. 
A measurable map $\kappa\rightarrow\nu[\kappa]$ is called {\em a translation-} or {\em shift-covariant disordered Gibbs
measure} if $\nu[\kappa]$ is a $\phi$-Gibbs measure for $\P$-almost every $\kappa$, and for all $v\in\zd$ and for all $F\in C_b(\R^{\zd})$, we have
$$\int\nu[\tau_v\kappa](\ormd\phi)F(\phi)=\int\nu[\kappa](\ormd\phi)F(\tau_v\phi).$$

To define the notion of measurability for a measure-valued function we use the 
evaluation sigma-algebra in the image space, which is the smallest sigma-algebra such that
the evaluation maps $\mu\mapsto \mu(A)$ are measurable for all events $A$ (for details, see \cite[p. 129]{MR2807681}).

\item [(b)] Let $\P$ be a probability measure on $(K^{\bzd},\Kcal^{\otimes\bzd})$. A measurable map $\kappa\rightarrow\mu[\kappa]$ is called {\em a shift-covariant disordered gradient Gibbs measure} if $\mu[\kappa]$ is a $\nabla\phi$-Gibbs measure for $\P$-almost every $\kappa$, and for all $v\in\zd$ and for all $F\in C_b(\chi)$, we have
$$\int\mu[\tau_v\kappa](\ormd\eta)F(\eta)=\int\mu[\kappa](\ormd\eta)F(\tau_v\eta).$$
\end{itemize}
\end{definition}
\begin{remark}
We pause here to observe that we will sometimes use the shorter term $\kappa$-disordered gradient Gibbs measures for measures satisfying Definition \ref{shiftcov1}. Additionally, if $\P$ is shift-invariant, then the averaged measure $\nu$ defined by $\nu^{\mbox{av}}(\d\phi):=\int\nu[\kappa](\d\phi)\d\mathbb{P}(\kappa)$, respectively the averaged measure $\mu$ defined by $\mu^{\mbox{av}}(\d\eta):=\int \mu[\kappa](\d\eta)\d\mathbb{P}(\kappa)$, are shift-invariant probability measures by Definition \ref{shiftcov1}.  

Finally, we note that if $\kappa\to\mu[\kappa]$ is a shift-covariant disordered gradient Gibbs measure such that $\mu^{\mbox{av}}$ is tempered, we have $\mu^{\mbox{av}}(\chi_r)=1$ and hence also $\mu[\kappa](\chi_r)=1$ for $\mathbb{P}$- a.e. $\kappa$.
\end{remark}
We introduce one final definition before we proceed with the proofs.
\begin{definition}
 Fix $u\in\mathbb{R}^d$. Let $\P$ be a probability measure on $(K^{\bzd},\Kcal^{\otimes\bzd})$, and let $\mathbb{E}$ be the corresponding expectation. We say that a shift-covariant gradient Gibbs measure $\kappa\rightarrow \mu[\kappa]$ has expected tilt $u$ if it satisfies
\begin{equation}
\label{tilt1}
\int\mu^{\mbox{av}}(\ormd\eta)\eta(b)= \int\left(\int\mu[\kappa](\ormd\eta)\eta(b)\right)\d\P(\kappa)=\langle u,y_b-x_b \rangle, ~\mbox{for all bonds}~b=(x_b,y_b)\in(\Z^d)^*.
\end{equation}
\end{definition}
 In order to characterise the law of $\extmu$ conditional on the $\kappa$, we need to use a uniqueness result for $\kappa$-disordered gradient Gibbs measures with given distribution on the $\kappa$. The setting we will focus on is with the $V_\kappa''$ uniformly bounded above, which is precisely the main result in \cite[Theorem 1.10]{MR3383338},  stated in Theorem \ref{t:CK}(b). The theorem follows by means of \cite[Theorem 4.1 and Theorem 4.5]{MR3383338}, stated in Theorem \ref{t:CK}(a). 

\begin{theorem}
\label{t:CK}
Fix $d\ge 1$. Let $\P$ be a shift-invariant ergodic probability measure on $(K^{\bzd},\Kcal^{\otimes\bzd})$ such that $\P$-a.s. $\kappa(b)=\kappa(-b)$ for all $b\in\bzd$. Assume that for all $b\in (\zd)^*$, the map $(\kappa(b),s)\mapsto V_{\kappa(b)}(s)$ is jointly measurable, and for each $\kappa(b)$ the function $V_{\kappa(b)}(\cdot)$ is in $C^2(\R)$ and even. Assume also that there exist $c_V, C_V>0$ such that $C_V\ge V''_{\kappa(b)}\ge c_V>0$ for all $b\in (\zd)^*$ and $\kappa\in K^\bzd$.
\begin{itemize}
\item [(a)] Let for all $j\in\{1,2,\ldots,d\}$ 
\begin{equation}
\label{directiontilt}
E_j:=\{ \eta~|~\lim_{n\rightarrow\infty}\frac{1}{|\Lambda_n|}\sum_{x\in\Lambda_n}\eta(b_{x, j})=0\},
\end{equation}
along the sequence of volumes $\Lambda_n:=[-n,n]^d\cap\zd$, where $b_{x,j}:=(x, x+e_j)\in (\zd)^*$. Then there exists \textbf{exactly one} (up to $\mathbb{P}$-null sets) shift-covariant gradient Gibbs measure $\kappa\rightarrow \mu[\kappa]$ with associated family of Hamiltonians $(\kappa\to H_{\Lambda,\xi}[\kappa])_{\Lambda\subset\zd,\xi\in\chi}$ defined as in \eqref{e:disordered_H}, such that $\mu[\kappa]$ satisfies
\begin{equation}
\label{directiontilteq}
\mu[\kappa](E_j)=1,~j\in\{1,2,\ldots,d\}
\end{equation}
for $\P$-a.s. $\kappa$, and the annealed measure $\mu^{\text{av}}(\d\eta):=\left(\int\d\P(\kappa)\mu[\kappa]\right)(\ormd\eta)$ satisfies the integrability condition 
\begin{equation}
\label{intcond1a}
\sup_{b\in (\zd)^*}\int\eta^2(b)\mu^{\text{av}}(\d\eta)<\infty.
\end{equation}
Additionally, $\mu^{\text{av}}$ can be shown to be ergodic under the shifts.
\item [(b)]
Then there exists \textbf{exactly one} (up to $\mathbb{P}$-null sets) shift-covariant gradient Gibbs measure $\kappa\rightarrow \mu[\kappa]$ with associated family of Hamiltonians $(\kappa\to H_{\Lambda,\xi}[\kappa])_{\Lambda\subset\zd,\xi\in\chi}$ defined as in \eqref{e:disordered_H}, such that the annealed measure $\mu^{\text{av}}$ has zero tilt, is ergodic under the shifts and satisfies the integrability condition \eqref{intcond1a}.
\end{itemize}
\end{theorem}

We note here that a more detailed characterisation of the disordered gradient Gibbs measures appearing in Theorem \ref{t:CK} will be given in Lemma \ref{c:unique_disordered_Gibbs} (with an upgraded exponential-moment bound instead of \eqref{intcond1a}) and in Section \ref{s:characterisation} below.  Here we only explain how we apply Theorem \ref{t:CK}(b) to the setting of our extended gradient Gibbs measure $\tilde\mu$. More precisely, using Theorem \ref{t:CK}(b), we can state the counterpart to Lemma \ref{l:extension_tinv_erg} that gives a characterisation of $\extmu$ conditional on the $\kappa$ that is unique almost surely.

\begin{lemma}\label{l:extension_disordered}
Suppose that $V\colon\R\to\R$ is a strongly log-concave log-mixture of quadratic growth. Suppose that $\mu$ is a shift-invariant tempered ergodic gradient Gibbs measure with zero tilt, and let $\extmu$ be its extension with associated Hamiltonian $(H_{\Lambda, \xi, \bar\kappa})_{\Lambda\subset\zd, \xi\in\chi, \bar\kappa\in K^{\bzd}}$. Then $\kappa\to\extmu(\d\eta\mid\Fkappa)(\kappa)$ is a (unique up to $\bar\mu$ null-sets) disordered gradient Gibbs measure with associated family of Hamiltonians $\kappa\to (H_{\Lambda, \xi}[\kappa])_{\Lambda\subset\zd, \xi\in\chi}$ such that its annealed measure (equal to $\mu$) is ergodic, has zero tilt and is tempered. Furthermore, $\tilde\mu$-a.s. we have $\E_\extmu(F(\eta)\mid\Fkappa)=0$ for any odd function $F\in L^1(\mu[\kappa])$, and in particular $\E_\extmu(\eta(b)\mid\Fkappa)=0$ for all $b\in\bzd$.
\end{lemma}

\begin{remark}
Lemma \ref{l:extension_disordered} in particular says that each shift-invariant ergodic tempered gradient Gibbs measure with zero tilt is already determined by the $\kappa$-marginal of its extension.
\end{remark}

\begin{proof}[Proof of Lemma \ref{l:extension_disordered}]
Consider the family $\mu[\kappa]$ denoted by $\mu[\kappa](\d\eta)=\extmu(\d\eta\mid\Fkappa)(\kappa)$, which defines a regular conditional probability measure (see \cite[Section 1.2]{MR2807681} for more in-depth properties). 
We recall here that, as a regular conditional probability of $\tilde\mu$ given $\Fkappa$, then $\mu[\cdot](\cdot)$ is a probability kernel from $(K^{\bzd},\Fkappa)$ to $(\chi,\Feta)$, such that $\tilde\mu$-a.s. we have
\begin{equation}
\label{regprob}
\mu[\cdot](F)=\E_{\tilde\mu}(F\mid\Fkappa)(\cdot),~\mbox{for every}~\Feta-\mbox{measurable}~F~\mbox{which is}~\mu[\kappa]-\mbox{integrable}.
\end{equation}

We first check that $\kappa\to\mu[\kappa]$ is almost surely a $\kappa$-disordered gradient Gibbs measure. This is a straightforward calculation.

Firstly, $\kappa\to\mu[\kappa]$ is measurable by definition of $\extmu(\d\eta\mid\Fkappa)$ as a regular conditional expectation.

We now prove that $\mu[\kappa]$ is a gradient Gibbs measure $\tilde\mu$-a.s., that is, we will show that \eqref{dlrgrad} holds for $\mu[\kappa]$. In view of \eqref{regprob}, this is equivalent to showing that for $\tilde\mu$-a.s. $\kappa$ we have
\begin{equation}
\int_\chi\mu[\kappa](\d\xi)\int_{\chi_{\bzdlb,\xi}} \!\!\!\!\!F(\eta)\mu_{\Lambda,\xi}[\kappa](\d\eta)=\E_{\tilde\mu}(F\mid\Fkappa)(\kappa),~\forall F\in C_b(\chi),
\end{equation}
where 
\begin{equation}
\label{quenchedham}
\mu_{\Lambda, \xi}[\kappa](\d \eta)=\frac{\exp\left(-\frac{1}{2}\sum_{b\in\bzdlb} V_{\kappa(b)}(\eta(b)\right)}{Z_{\Lambda,{\xi}}[\kappa]} \d\eta_{\Lambda, \xi},
\end{equation}
with the obvious notation for the normalising constant $Z_{\Lambda,{\xi}}[\kappa]$. By the disintegration theorem, it is sufficient to show for all $A\in\Fkappa$, $B\in\Feta$, that
\[
\int_{K^{\bzd}}\1_A(\kappa)\bar\mu(\d{\kappa})\int_\chi \mu[\kappa](\d\xi)\int_{\chi_{\bzdlb,\xi}} \!\!\!\!\!\1_B(\eta)\mu_{\Lambda,\xi}[\kappa](\d\eta)=\int_{\chi\times K^{\bzd}} \tilde\mu(\d\eta, \d\kappa)\1_A(\kappa)\1_B(\eta).
\]

Firstly we claim that for $A\in\Fkappa$, $B\in\Feta$ and $\Lambda\subset\zd$ we have
\begin{equation}
\label{finvoldisint}
\int_{\chi_{\bzdlb,\xi}\times K^{\bzd}}\1_A(\tilde\kappa)\1_B(\eta)\tilde\mu_{\Lambda, \xi, 
\kappa}(\d\eta, \d\tilde\kappa)=\int_{\chi_{\bzdlb,\xi}\times K^{\bzd}}\1_A(\tilde\kappa)\tilde\mu_{\Lambda, \xi, 
\kappa}(\d\tilde\xi, \d\tilde\kappa)\int_{\chi_{\bzdlb,\tilde\xi}} \!\!\!\!\!\1_B(\eta)\mu_{\Lambda, \tilde\xi}[\tilde\kappa](\d\eta).
\end{equation}
This is a finite-volume identity, and so it can be checked using the explicit formula \eqref{eq:specification_gamma_tilde} and Definition \ref{d:finitevoldisgibbs}. When doing so, the integral over $\tilde\xi$ will precisely cancel with the partition function $Z_{\Lambda,\xi}[\tilde\kappa]$.

We now have
\begin{align*}
&\int_{K^{\bzd}}\1_A(\kappa)\bar\mu(\d{\kappa})\int_\chi \mu[\kappa](\d\xi)\int_{\chi_{\bzdlb,\xi}} \!\!\!\!\!\1_B(\eta)\mu_{\Lambda,\xi}[\kappa](\d\eta)\\
&=\int_{\chi\times K^{\bzd}}\tilde\mu(\d\xi,\d{\kappa})\1_A(\kappa)\int_{\chi_{\bzdlb,\xi}} \!\!\!\!\!\1_B(\eta)\mu_{\Lambda,\xi}[\kappa](\d\eta)\\
&=\int_{\chi\times K^{\bzd}}\tilde\mu(\d{\xi}, \d{\kappa})\int_{\chi_{\bzdlb}\times K^{\bzd}}\1_A(\tilde\kappa)\tilde\mu_{\Lambda, \xi, 
\kappa}(\d\tilde\xi, \d\tilde\kappa)\int_{\chi_{\bzdlb,\tilde\xi}} \!\!\!\!\!\1_B(\eta)\mu_{\Lambda,\tilde\xi}[\tilde\kappa](\d\eta)\\
&=\int_{\chi\times K^{\bzd}}\tilde\mu(\d{\xi}, \d{\kappa}) \int_{\chi_{\bzdlb,\xi}\times K^{\bzd}}\1_A(\tilde\kappa)\1_B(\eta)\tilde\mu_{\Lambda, \xi, 
\kappa}(\d\eta, \d\tilde\kappa)\\
&=\int_{\chi\times K^{\bzd}} \1_A(\kappa)\1_B(\eta)\tilde\mu(\d\eta, \d\kappa),
\end{align*}
where for the first equality we used that $\tilde\mu$ has marginal $\bar\mu$, for the second equality we applied that $\tilde\mu_{\Lambda, \xi, 
\kappa}$ is a specification for $\tilde\mu$, for the third equality we applied \eqref{finvoldisint}, and for the fourth equality we used again that $\tilde\mu_{\Lambda, \xi, 
\kappa}$ is a specification for $\tilde\mu$. This implies the claim.

Next, since $\tilde\mu$ is shift-invariant, we easily get that $\kappa\to\mu[\kappa]$ is shift-covariant by applying a similar reasoning as above to show that $\tilde\mu$-a.s., we have $\mu[\tau_v\kappa]=\mu[\kappa]\circ\tau_v^{-1}$, for all $v\in\zd$. 

Finally, to verify that $\kappa\to\mu[\kappa]$ has expected zero tilt, we write for all bonds $b\in \bzd$
\[
\int_{K^{\bzd}} \bar\mu(\d\kappa)\int_{\chi}\eta(b)\mu[\kappa](\d\eta)
=\int_{\chi} \mu(\d\eta)\eta(b)=0,
\]
since $\tilde\mu$ has marginal $\mu$, which has zero tilt. Additionally, the integrability assumption \eqref{intcond1a} is a direct consequence of the fact that $\mu$ is tempered, via a similar argument as above. Finally, ergodicity of the annealed measure is immediate since $\mu^{\mbox{av}}=\mu$. 

Therefore, $\mu[\kappa]$ satisfies the assumptions of Theorem \ref{t:CK}(b) so a.s. uniqueness follows. It remains to show that a.s. $\E_{\mu[\kappa]}(F(\eta))=0$ for any odd function $F\in L^1(\mu[\kappa])$. To see this, define $\mu'[\kappa](\d\eta)=\mu[\kappa](-\d\eta)$. Then the family $\mu'[\kappa]$ can also be shown to satisfy the assumptions of Theorem \ref{t:CK}. So Theorem \ref{t:CK} implies that we must have $\mu[\kappa]=\mu'[\kappa]$ for $\bar\mu$-a.e. $\kappa$.
This in turn means that we must have
\[\int \mu[\kappa](\d\eta)F(\eta)=\int \mu[\kappa](-\d\eta)\F(\eta)=-\int \mu[\kappa](\d\eta)F(\eta),\]
for $\bar\mu$-a.e. $\kappa$. Thus, $\E_{\mu[\kappa]}(F(\eta))=0$, and in particular (taking $F(\eta)=\eta(b)$)  $\E_{\mu[\kappa]}(\eta(b))=0$.
\end{proof}

\section{Brascamp-Lieb inequalities for \texorpdfstring{$\alpha$}{alpha}-monotone potentials}
\label{BLSection}

There are two related inequalities in the literature known under the name of Brascamp-Lieb inequality. The first one, stated in \eqref{BL111}, is a concentration of measure around the mean inequality, whereas the second one, stated in \eqref{poinc}, is a reinforcement of the Poincar\'e inequality. While we will only use the generalisation of \eqref{BL111} to $\alpha$-monotone potentials to obtain bounds on the moments of the gradient fields in Lemma \ref{l:expbrasc}, which are needed for the scaling limit, both inequalities are of independent interest. Among many uses, \eqref{poinc} can be used to show concentration of measure, covariance inequalities and to analyse phase transitions and uniqueness of the measure. Furthermore, for gradient models \eqref{poinc} can be seen as an instance of the Helffer-Sjöstrand representation in the particular case of the variance, and is used to control higher moments of the field (see for example \cite[Lemma 2.5]{JDDRodrig}).

We will first introduce these two inequalities in the (strongly) log-concave setting, and then we will focus on our main results in Theorems \ref{genposbeta} and \ref{genposbetapoincare}; we stress here that versions of these theorems can also be derived in the infinite volume setting of Section \ref{s:existence}. 

\subsection{Brascamp-Lieb inequalities for gradient models}

\subsubsection*{Concentration around the mean Brascamp-Lieb inequality}

This version of the inequality is the formulation from \eqref{hargebl} above (as proved in \cite[Theorem 1.1]{Harge}). Translated to the notation of gradient Gibbs measures from Section \ref{s:definitions}, and assuming that $V(s)- \alpha s^2/2, s\in\mathbb{R}$, is convex, \eqref{hargebl} gives for $v\in\mathbb{R}^{|\Lambda|}$ 
\begin{equation}\begin{split}\label{BL111}
&\nu_{\Lambda,\psi} \left(F\big(v \cdot (\phi -\nu_{\Lambda,\psi}(\phi))\big)\right)\leq \nu^{\Delta_{\alpha}}_{\Lambda,\psi}\left(F(v \cdot (\phi-\nu^{\Delta_{\alpha}}_{\Lambda,\psi}(\phi)))\right),
\end{split}
\end{equation}
where by $ \nu^{\Delta_{\alpha}}_{\Lambda,\psi}$ we denote the massless Gaussian Free Field Gibbs measure with conductance $\alpha$, that is the gradient Gibbs measure associated with $V(s)=\frac{\alpha s^2}{2}$.

In the special case $\psi=\psi_0=0$, we have for all $x\in\mathbb{Z}^d$ that 
\begin{equation}\label{lalagibbs}
\nu_{\Lambda, 0}\left(\phi(x)\right)=0.
\end{equation} Indeed, the symmetry of $V$ and the change of variables $\phi(y)\rightarrow-\phi(y)$ imply
$\nu_{\Lambda, 0}(\phi(x))=-\nu_{\Lambda, 0}(\phi(x))$.
In particular, \eqref{BL111} becomes for $\psi=0$ 
\begin{equation}\label{BL111zero}
\nu_{\Lambda, 0}\left(F\big(v \cdot \phi)\right)\leq \nu^{\Delta_{\alpha}}_{\Lambda,0}\left(F(v \cdot \phi)\right).
\end{equation}

We note that a converse of \eqref{BL111zero} is true as well: If $C_Vs^2/2-V(s)$ is convex, then
\begin{equation}\label{BL111zero_converse}
\nu^{\Delta_{C_V}}_{\Lambda,0}\left(F(v \cdot \phi)\right)\leq\nu_{\Lambda, 0}\left(F\big(v \cdot \phi)\right).
\end{equation}
Indeed, this inequality follows from \cite[Theorem 1.2]{Harge} together with the observation that every log-convex function is convex. Furthermore, there are even sharper versions of \eqref{BL111zero} and \eqref{BL111zero_converse}, as shown in \cite[Theorem 1.1]{Hariya}.

\subsubsection*{Brascamp-Lieb inequality as reinforcement of the Poincar\'e inequality}

Consider a $C^2$ function $V:\mathbb{R}^n\to\mathbb{R}$ such that for all $x\in\mathbb{R}^n$, there holds $\mbox{Hess}\,V(x)\succ 0$ and $\int \e^{-V(z)}\rmd z<\infty.$

Let
$$\rmd\mu_V(x):=\frac{\e^{-V(x)}\rmd x}{\int \e^{-V(z)}\rmd z}.$$
Then for all $F\in C^1(\mathbb{R}^n)\cap L^2(\mu_V)$, the following inequality was first proved in Theorem 4.1 from \cite{MR0450480} (in the special case where $F$ is a linear combination of $x_i$, where $x=(x_1, \ldots, x_n)$).
\begin{equation}
\label{poinc}
{\Var}_{\mu_V}(F)\le\int \langle\nabla F, (\mbox{Hess V})^{-1}\nabla F\rangle\rmd\mu_V.
\end{equation}
For an alternative proof of the above for (gradient) Gibbs measures with uniformly strictly convex potentials $V$, by means of the Helffer-Sj\"ostrand representation, see for example Theorem 4.8 from \cite{MR2228384}.

In the next section, we will extend these two types of inequalities to $\alpha$-monotone even potentials, in the setting of (gradient) Gibbs measures.

\subsection{Brascamp-Lieb inequalities for \texorpdfstring{$\alpha$}{alpha}-monotone potentials}

\subsubsection*{Concentration around the mean Brascamp-Lieb inequality}

Defining $\xi_0$ as $\xi_0(b):=\nabla\psi_0(b)$ for all $b\in (\zd)^*$, we will prove 
\begin{theorem}
\label{genposbeta}
Let $d\ge 1$. Assume that $V\in C(\mathbb{R})$ is even and $\alpha$-monotone for some $\alpha>0$. Then for all convex functions $F:\mathbb{R}\rightarrow\mathbb{R}$ bounded below and for all $v\in\mathbb{R}^{|\Lambda|}$, respectively $v^*\in\mathbb{R}^{|\bzdlb|}$, it holds  that 
\begin{equation}
\label{BLweakupperbound}
\int_{\mathbb{R}^{|\Lambda|}}\!\!\!\!\!\!\!F(v\cdot \phi)\nu_{\Lambda,0}(\d\phi)\le  \int_{\mathbb{R}^{|\Lambda|}}\!\!\!\!\!\!\!F(v\cdot \phi)  \nu^{\Delta_{\alpha}}_{\Lambda,0} (\d\phi),~~\mbox{and}~~ \int_{\chi_{\bzdlb,\xi_0}}\!\!\!\!\!\!\!\!\!F(v^*\cdot \eta)\mu_{\Lambda, \xi_0}(\d\eta)\le  \int_{\chi_{\bzdlb,\xi_0}}\!\!\!\!\!\!\!\!\!\!F(v^*\cdot \eta)  \mu^{\Delta_{\alpha}}_{\Lambda,\xi_0} (\d\eta).
\end{equation}
\end{theorem}
\begin{proof}
We will only prove below how the estimate for the Gibbs measure in \eqref{BLweakupperbound} follows from \eqref{BL111zero}, as the argument for the gradient Gibbs measure will be analogous.

By Proposition \ref{p:trivialdecomp},  $V$ can be written as
\[
V(s)=-\log\int\exp(-V_\kappa(s))\rho(\rmd \kappa),
\]
where $V_\kappa\in C^2(I_\kappa)$ for some open interval $I_\kappa\subset\R$, $V_\kappa''(s)\ge\alpha$ for all $s\in I_\kappa$, and $V_\kappa=+\infty$ on $\R\setminus I_\kappa$.

Coupling the above with \eqref{e:disordered_H}, we write
\begin{equation}
\label{disordBL}
\begin{split}
&\int_{\mathbb{R}^{|\Lambda|}}F(v\cdot \phi)\nu_{\Lambda,0}(\d\phi)\\
&=\frac{1}{Z_{\Lambda,0}} \int_{ K^{\ubzdlb}}\bigg(\int_{\RR^\zd}F(v\cdot \phi)\exp\left\{- H_{\Lambda,0}[\kappa](\phi)\right\}\rmd\phi_\Lambda\delta_0(\d\phi_{\zd\setminus\Lambda})\bigg)\prod_{b=(x,y)\in \ubzdlb}\rho(\d\kappa(b))\\
&=\frac{1}{Z_{\Lambda,0}} \int_{ K^{\ubzdlb}}Z_{\Lambda,0}[\kappa]\bigg(\frac{1}{Z_{\Lambda,0}[\kappa]}\int_{\RR^\zd}F(v\cdot \phi)\exp\left\{-H_{\Lambda,0}[\kappa](\phi)\right\}\rmd\phi_\Lambda\delta_0(\d\phi_{\zd\setminus\Lambda})\bigg)\prod_{b=(x,y)\in \ubzdlb}\rho(\d\kappa(b))\\
&=\frac{1}{Z_{\Lambda, 0}} \int_{ K^{\ubzdlb}}Z_{\Lambda, 0}[\kappa]\mathbb{E}_{\nu_{\Lambda, 0}[\kappa]}\left(F(v\cdot \phi)\right)\prod_{b=(x,y)\in \ubzdlb}\rho(\d\kappa(b)),
\end{split}
\end{equation}
with the obvious definition for $\nu_{\Lambda, 0}[\kappa]$, and where $Z_{\Lambda,0}=\int_{ K^{\ubzdlb}} Z_{\Lambda,0}[\kappa]\prod_{b=(x,y)\in \ubzdlb}\rho(\d\kappa(b)).$
Applying now \eqref{BL111}, we have for all $\kappa$ that
\[\mathbb{E}_{\nu_{\Lambda, 0}[\kappa]}\left(F(v\cdot \phi)\right)\le \int_{\mathbb{R}^{|\Lambda|}}F(v\cdot \phi)  \nu^{\Delta_{\alpha}}_{\Lambda,0} (\d\phi).\]
Combining this with \eqref{disordBL}, we conclude the proof.
\end{proof}

\subsubsection*{Brascamp-Lieb inequality as reinforcement of the Gaussian Poincar\'e inequality}

Before we state this type of Brascamp-Lieb inequality, we will introduce the discrete Laplacian $\Delta_{\Lambda,\psi}$ for $\Lambda\subset\mathbb{Z}^d$ with boundary conditions $\psi\in\mathbb{R}^{\partial\Lambda}$ 
\begin{equation}
\label{laplacop}
\Delta_{\Lambda,\psi}\phi(x):=\sum_{y\in \bar{\Lambda}:|y-x|=1} \left((\phi\vee\psi)(y)-\phi(x)\right),~~~x\in\Lambda,
\end{equation}
where $(\phi\vee\psi)=\phi(x)$ if $x\in\Lambda$, and $(\phi\vee\psi)=\psi(x)$ if $x\in\partial\Lambda$.

For a function $F\in C^1(\chi)$ we define the derivative \[\partial_xF(\eta):=\frac{d}{dt}\Big|_{t=0}F(\eta+t\nabla\1_x)\]
and note that if $F$ is the restriction of a function $\tilde F\in C^1(\R^\bzd)$ to $\chi$, then this is the same as
\[\partial_xF(\eta):=-\sum_{\substack{b'\in\bzd\\ x_{b'}=x}}\frac{\partial}{\partial \eta(b')}\tilde F(\eta)+\sum_{\substack{b'\in\bzd\\ y_{b'}=x}}\frac{\partial}{\partial \eta(b')}\tilde F(\eta).\]
(cf. \cite[(3.1)]{MR1872740}).

The following theorem is an extension of \cite[Theorem 4.8]{MR2228384} to our setting (see also \cite[Lemma 2.3]{JDDRodrig}).
\begin{theorem}
\label{genposbetapoincare}
Let $d\ge 1$. Assume that $V\in C(\mathbb{R})$ is an $\alpha$-monotone even potential. Then for all odd functions $F\in C^1(\chi)\cap L^{2}(\mu_{\Lambda, \xi_0})$ we have that 
\begin{equation}
\Var_{\mu_{\Lambda,\xi0}}(F(\eta))\le \sum_{x, y\in\Lambda}\mathbb{E}_{\mu_{\Lambda,\xi_0}}\left(\partial_x F(\eta)\left(-\alpha\Delta_{\Lambda,0}\right)^{-1}(x,y)\partial_y F(\eta)\right),
\end{equation}
with a similar inequality holding for $\nu_{\Lambda,0}$, with $\partial_xF(\phi)$ and $\partial_y F(\phi)$ used instead in the formula.
\end{theorem}
\begin{proof}
By Proposition \ref{p:trivialdecomp},  $V$ can be written as
\[
V(s)=-\log\int\exp(-V_\kappa(s))\rho(\rmd\kappa),
\]
where $V_\kappa\in C^2(I_\kappa)$ for some open interval $I_\kappa\subset\R$, $V_\kappa''(s)\ge\alpha$ for all $s\in I_\kappa$, and $V_\kappa=+\infty$ on $\R\setminus I_\kappa$. Using the same notation as in \eqref{disordBL}, and since $F$ is an odd function, we have that 
$$\mathbb{E}_{\mu_{\Lambda,\xi_0}}(F(\eta))=0~~~~\mbox{and}~~~~\mathbb{E}_{\mu_{\Lambda, \xi_0}[\kappa]}\left(F(\eta)\right)=0,~~\mbox{for all}~~\kappa\in K^\ubzdl.$$
By making use once again of \eqref{disordBL}, we obtain
\begin{equation}
\label{disordpoinc}
\Var_{\mu_{\Lambda,\xi_0}}(F(\eta))=\frac{1}{Z_{\Lambda, 0}} \int_{ K^{\ubzdlb}}Z_{\Lambda,0}[\kappa]\Var_{\mu_{\Lambda, \xi_0}[\kappa]}\left(F(\eta)\right)\prod_{b=(x,y)\in \ubzdlb}\rho(\rmd\kappa(b)).   
\end{equation}
While our $V_\kappa$ functions do not satisfy the assumptions of the Brascamp-Lieb inequality as stated in \eqref{poinc}, they do fulfil the conditions from the version of the inequality stated in Theorem 1.2 (1) from \cite{KolesMil}, respectively in Theorem 1.5 (1) from \cite{Livsh}. Therefore, for each $\Var_{\mu_{\Lambda, 0}[\kappa]}\left(F(\eta)\right)$ we write 
\begin{equation}
\label{poinckappa}
\Var_{\mu_{\Lambda, \xi_0}[\kappa]}\left(F(\eta)\right)\le \mathbb{E}_{\mu_{\Lambda,\xi_0}[\kappa]}\left(\nabla F (\mbox{Hess}~H_{\Lambda,0}[\kappa])^{-1}\nabla F\right),
\end{equation}
where $\mbox{Hess}~H_{\Lambda,0}[\kappa]$ is positive semidefinite and therefore has a pseudo-inverse, and $\mbox{Hess}~H_{\Lambda,0}[\kappa]\ge (-\alpha\Delta_{\Lambda,0})$ in the sense of symmetric operators (similarly to the argument of Lemma 4.7 (a) and (c) from \cite{MR2228384} involving $Q_{\Lambda,0}$ and $\Delta_{\Lambda,0}$ therein). Plugging this inequality in \eqref{poinckappa} produces
$$\Var_{\mu_{\Lambda, \xi_0}[\kappa]}\left(F(\eta)\right)\le \mathbb{E}_{\mu_{\Lambda, \xi_0}[\kappa]}\left(\nabla F (-\alpha\Delta_{\Lambda,0})^{-1}\nabla F\right),$$
which coupled with \eqref{disordpoinc} gives the desired result.
\end{proof}

\subsubsection*{Some special cases}
For the proof of Theorem \ref{t:existenceGGMs} we will need some special cases of the Brascamp-Lieb inequality that we will collect in the following lemma. 

For each $x_0\in\zd$, $m\in\mathbb{N}$, let $\Lambda_m+x_0:=x_0+[-m,m]^d\cap \zd$.

\begin{lemma}
\label{l:expbrasc}
 Set $d\ge 1$.  Let $V\in C(\mathbb{R})$ be an $\alpha$-monotone even function.  Then for all $0<\delta<\frac{\alpha}{2}$, we have 
\begin{equation}
\label{logmixtureexp}
\sup_{b\in(\zd)^*, x_0\in\zd,m\in\N}\mathbb{E}_{\mu_{\Lambda_m+ x_0, \xi_0}}\exp((\alpha/2-\delta)\eta^2(b))\le\frac{\sqrt{\alpha}}{\sqrt{2\delta}}.
\end{equation}
Suppose now that $f\in C_c^\infty(\R^d,\R^d)$, fix $\epsilon>0$ and take $F_\epsilon$ as in \eqref{e:observable}.
Then for all $\lambda\in\R$, there holds for all $\Lambda_m\subset\zd$  and for some $\epsilon$-independent $C(d, f)>0$ that
\begin{equation}
\label{logmixtureexplinear}
\sup_{x_0\in\zd,m\in\N}\mathbb{E}_{\mu_{\Lambda_m+x_0,\xi_0}}\exp(\lambda F_\eps(\eta))\le\exp\left(\frac{C(d,f)\lambda^2}{\alpha}\right).
\end{equation}
\end{lemma}

For the proof we will need estimates for the Green's function of the discrete Laplacian. 

We let $G_{\Lambda_m+x_0}=(-\Delta_{\Lambda_m+x_0, 0})^{-1}$  be the Green function for the discrete Laplacian $\Delta_{\Lambda_m+x_0, 0}$ with boundary condition $0$, as defined in \eqref{laplacop}. We also let
\[\nabla G_{\Lambda_m+x_0}((x,x+e_i),y)=G_{\Lambda_m+x_0}(x+e_i,y)-G_{\Lambda_m+x_0}(x,y)\]
and
\[\nabla\nabla'G_{\Lambda_m+x_0}((x,x+e_i),(y,y+e_j))=G_{\Lambda_m+x_0}(x+e_i,y+e_j)-G_{\Lambda_m+x_0}(x+e_i,y)-G_{\Lambda_m+x_0}(x,y+e_j)+G_{\Lambda_m+x_0}(x,y)\]
be the first and mixed second derivative of $G_{\Lambda_m+x_0}$. Then we have the following estimates.
\begin{lemma}\label{l:greensfunctionestimates}
Set $d\ge 1$.
\begin{itemize}
\item [(a)]
There is $C=C(d)<\infty$ such that for all $x_0\in\zd$, $m\in\N$, $x,y\in\Lambda_m+x_0$, $i,j\in\{1,\ldots,d\}$ we have
\begin{align}
|\nabla G_{\Lambda_m+x_0}((x,x+e_i),y)|&\le \frac{C}{(1\vee|x-y|)^{d-1}}\label{e:firstder}\\
|\nabla\nabla'G_{\Lambda_m+x_0}((x,x+e_i),(y,y+e_j))|&\le \frac{C}{(1\vee|x-y|)^d}\label{e:mixedsecondder}.
\end{align}
\item [(b)] For all $x_0\in\zd$, $m\in\N$, $b\in\overline{(\Lambda_m+x_0)^*}$ we have
\begin{equation}\label{e:vareta}
|\nabla\nabla'G_{\Lambda_m+x_0}(b,b)|\le 1.
\end{equation}
\end{itemize}
\end{lemma}
\begin{proof}
Estimates \eqref{e:firstder} and \eqref{e:mixedsecondder} are folklore, but we are not aware of a precise reference in the literature, and so we provide the short proof for completeness. By translation-invariance it suffices to consider the case $x_0=(m+1,\ldots,m+1)$, so that the lower left corner of $\partial(\Lambda_m+x_0)$ becomes the origin.

We will use reflections to compare $G_{\Lambda_m+x_0}$ with a suitable torus Green's function. For that purpose, let $T_m$ be the discrete torus of side length $4m+4$, which we identify with $\{-2m-1,\ldots,2m+2\}^d\subset\zd$, and let $G_{T_m}$ be the zero-average Green's function on $T_m$, uniquely defined by $\sum_{x\in T_m}G_{T_m}(x,y)=0$ for all $y$, and 
$-\Delta_{T_m}G_{T_m}(\cdot,y)(x)=\1_{x=y}-\frac{1}{(4m+4)^d}$.
We will need the analogues of \eqref{e:firstder} and \eqref{e:mixedsecondder} for $G_{T_m}$, that is the estimates
\begin{align}
|\nabla G_{T_m}((x,x+e_i),y)|&\le \frac{C}{(1\vee|x-y|_{T_m})^{d-1}}\label{e:firstdertorus}\\
|\nabla\nabla'G_{T_m}((x,x\pm e_i),(y,y\pm e_j))|&\le \frac{C}{(1\vee|x-y|_{T_m})^d}\label{e:mixedseconddertorus}.
\end{align}
Here $|z|_{T_m}=\inf_{w\in (4m+4)\zd}|z+w|$ denotes the natural Euclidean distance on the torus. The estimates \eqref{e:firstdertorus} and \eqref{e:mixedseconddertorus} can easily be proven by using Fourier analysis on $T_m$. Alternatively, they follow as a special case of \cite[Lemma 3.3]{MR3177848} upon taking the random environment there to be deterministically equal to the identity matrix. 

Now consider for $y\in\Lambda_m+x_0$ the function 
\[\tilde G_{\Lambda_m+x_0}(\cdot,y):=\sum_{\varepsilon_1,\ldots,\varepsilon_d\in\{-1,1\}}\varepsilon_1\cdots\varepsilon_dG_{T_m}(\cdot,\epsilon\odot y)\]
on $\overline{\Lambda_m+x_0}$, where we abbreviate $\epsilon\odot y:=(\varepsilon_1y_1, \ldots, \varepsilon_d y_d)$.

By symmetry we have $\tilde G_{\Lambda_m+x_0}(x,y)=0$ on $\partial(\Lambda_m+x_0)$. For example, along the hyperplane $x_1=0$ the summands with $\pm\varepsilon_1$ cancel each other out. Moreover, we have
\[-\Delta_{\Lambda_m+x_0}\tilde G_{\Lambda_m+x_0}(\cdot,y)(x)=\1_{x=y}-\sum_{\varepsilon_1,\ldots,\varepsilon_d\in\{-1,1\}}\varepsilon_1\cdots\varepsilon_d\frac{1}{(4m+4)^d}=\1_{x=y}\]
as the sum vanishes, again by symmetry. 

But this means that $\tilde G_{\Lambda_m+x_0}(\cdot,y)$ solves the same equation with the same boundary values as $G_{\Lambda_m+x_0}(\cdot,y)$, and so we must have $\tilde G_{\Lambda_m+x_0}(\cdot,y)= G_{\Lambda_m+x_0}(\cdot,y)$. In other words, we have shown that
\begin{equation}\label{e:reptorusgreenfct}
G_{\Lambda_m+x_0}(x,y)=\sum_{\varepsilon_1,\ldots,\varepsilon_d\in\{-1,1\}}\varepsilon_1\cdots\varepsilon_dG_{T_m}(x,\epsilon\odot y),
\end{equation}
for any $x,y\in\Lambda_m+x_0$. In fact, the same equality holds for $x,y\in\overline{(\Lambda_m+x_0)}$, as both sides vanish when one of $x,y$ is on $\partial(\Lambda_m+x_0)$.

Now \eqref{e:reptorusgreenfct} implies that
\[\nabla\nabla'G_{\Lambda_m+x_0}((x,x+ e_i),(y,y+ e_j))=\sum_{\varepsilon_1,\ldots,\varepsilon_d\in\{-1,1\}}\varepsilon_1\cdots\varepsilon_d\nabla\nabla'G_{T_m}((x,x+ e_i),(\epsilon\odot y,\epsilon\odot (y+e_j))),
\]
and hence from \eqref{e:mixedseconddertorus} we deduce that
\begin{align*}
|\nabla\nabla'G_{\Lambda_m+x_0}((x,x+ e_i),(y,y+ e_j))|
&\le\sum_{\varepsilon_1,\ldots,\varepsilon_d\in\{-1,1\}}\left|\nabla\nabla'G_{T_m}((x,x+ e_i),(\epsilon\odot y,\epsilon\odot (y+e_j)))\right|\\
&\le C\sum_{\varepsilon_1,\ldots,\varepsilon_d\in\{-1,1\}}\frac{1}{(1\vee|x-\epsilon\odot y|_{T_m})^d}.
\end{align*}
It remains to observe that among the mirror image points $\epsilon\odot y$ for $\varepsilon_1,\ldots,\varepsilon_d\in\{-1,1\}$, the one closest to $x$ is necessarily $y$. So we deduce that
\[|\nabla\nabla'G_{\Lambda_m+x_0}((x,x+ e_i),(y,y+ e_j))|\le
 2^d C\frac{1}{(1\vee|x-y|_{T_m})^d}
\]
and \eqref{e:mixedsecondder} follows because we also have $|x-y|_{T_m}=|x-y|$. The argument for \eqref{e:firstder} is similar.

For \eqref{e:vareta} note that
\[
\nabla\nabla'G_{\Lambda_m+x_0}(b,b) =\mathbb{E}_{\nu^{\Delta_{1}}_{\Lambda_m+x_0, 0}}(\eta_b^2).
\]
However, this expectation is the effective resistance
between $x_b$ and $y_b$ (see eqn (2.21) in \cite{MR3616205}), and this resistance is
decreasing in the graph by Rayleigh's monotonicity principle (page 35 in \cite{MR3616205}). So we can get an upper bound by considering only the
two-vertex graph with vertices $x_b$ and $y_b$ and an edge between them, and
there the bound for $\mathbb{E}_{\nu^{\Delta_{1}}_{\Lambda_m+x_0, 0}}(\eta_b^2)$ is a trivial computation.

\end{proof}
Now we can return to the proof of Lemma \ref{l:expbrasc}.

\begin{proof}[Proof of Lemma \ref{l:expbrasc}]
By Theorem \ref{genposbeta} for the convex function $F(\eta):=\exp\left((\alpha/2-\delta)\eta^2(b)\right)$ we have that

\[
\mathbb{E}_{\mu_{\Lambda_m+x_0, \xi_0}}\left(\exp((\alpha/2-\delta)\eta^2(b))\right)\le \mathbb{E}_{\mu_{\Lambda_m+x_0, \xi_0}^{\Delta_{\alpha}}} \left(\exp((\alpha/2-\delta)\eta^2(b))\right).   
\]
Under the measure $\mu_{\Lambda_m+x_0, \xi_0}^{\Delta_{\alpha}}$, $\eta_b$ is a centred Gaussian random variable, and by \eqref{e:vareta}, its variance is at most ${\alpha}^{-1}$. So the bound \eqref{logmixtureexp} follows from an explicit computation with the one-dimensional Gaussian density.

We focus now on showing \eqref{logmixtureexplinear}. By applying once more Theorem \ref{genposbeta}, this time to the convex function $\exp(\lambda F_\epsilon(\eta))$, we have 
\[\mathbb{E}_{\mu_{\Lambda_m+x_0, \xi_0}} \left( \exp(\lambda F_\epsilon(\eta))\right)
\le\mathbb{E}_{\mu_{\Lambda_m+x_0, \xi_0}^{\Delta_{\alpha}}} \left(\exp(\lambda F_\epsilon(\eta))\right)=\exp\left(\frac12\lambda^2 \Var_{\mu^{\Delta_{\alpha}}_{\Lambda_m+x_0, \xi_0}}\left(F_\epsilon(\eta)\right)\right),\]
and so it remains to bound the variance. 

By assumption $f\in C_c^\infty(\R^d,\R^d)$, and so in particular $f$ is bounded and there is $R\ge 1$ such that $f(x)=0$ whenever $|x|\ge R$. Using summation by parts in the $y$-variable, we find that
\begin{align*}
&\Var_{\mu^{\Delta_{\alpha}}_{\Lambda_m+x_0, \xi_0}}\bigg(\eps^{d/2}\sum_{x\in\zd}\sum_{i=1}^d f_i(\eps x)\eta((x,x+e_i))\bigg)\\
&=\frac{\eps^d}{\alpha}\sum_{x, y\in\zd}\sum_{i=1}^d\sum_{j=1}^d f_i(\eps x)f_j(\eps y)\nabla\nabla'{G}_{\Lambda_m+x_0}((x,x+e_i),(y,y+e_j))\\
&\le\frac{\eps^{d+1}}{\alpha}\sum_{|x|,|y|\le (R+1)/\epsilon}\sum_{i=1}^d\sum_{j=1}^d |f_i(\eps x)|\frac{|f_j(\eps (y-e_j))-f_j(\epsilon y)|}{\eps}|\nabla{G}_{\Lambda_m+x_0}((x,x+e_i),y)|.
\end{align*}
Now we can use \eqref{e:firstder} to see that
\begin{align*}
&\Var_{\mu^{\Delta_{\alpha}}_{\Lambda_m+x_0, \xi_0}}\bigg(\eps^{d/2}\sum_{x\in \zd}\sum_{i=1}^d f_i(\eps x)\eta((x,x+e_i))\bigg)\\
&\le
C\frac{\eps^{d+1}}{\alpha}\|f\|_{L^\infty(\R^d)}\|\nabla f\|_{L^\infty(\R^d)}\sum_{|x|,|y|\le (R+1)/\epsilon}\sum_{i=1}^d\sum_{j=1}^d \frac{1}{(1\vee|x-y|)^{d-1}}\\
&\le \frac{CR^{d+1}}{\alpha}\|f\|_{L^\infty(\R^d)} \|\nabla f\|_{L^\infty(\R^d)}.
\end{align*}
Here in the last step we used that there are only $C\frac{R^d}{\eps^d}$ choices for $y$, and for each fixed $y$ the sum over $x$ is bounded by $C\frac{R}{\eps}$. This completes the proof of \eqref{logmixtureexplinear}.

\end{proof}

\subsection{Application to strengthened concentration inequalities}
We will prove in this section Theorem \ref{comparenon-convex}.

\begin{proof}[Proof of Theorem \ref{comparenon-convex}]
The idea of the proof is straightforward: We will use the decomposition from Proposition \ref{p:trivialdecomp} and combine it with standard functional inequalities for the individual potentials.

\textbf{Step 1: Proof of (a)(i) and (b)(i)}\\
We begin with part (a) (i). The proof follows by induction on $d\ge 1$: the key idea is that at each induction step, we end up with an $\alpha$-monotone potential in one dimension, so we can apply the decomposition together with the Brascamp-Lieb inequality formulas from the previous induction step. 

We start with the result for $d=1$. By definition, $\mathcal{Y}$ is an $\alpha$-monotone and even continuous function. Then by Proposition \ref{p:trivialdecomp}, there holds
\[
\mathcal{V}(x)=\e^{-\mathcal{Y}(x)}=\int \exp(-\mathcal{Y}_\kappa(x))\rho(\rmd\kappa),
\]
where $\mathcal{Y}_\kappa\in C^2(I_\kappa)$ for some open interval $I_\kappa\subset\R$, $\mathcal{Y}_\kappa''(s)\ge\alpha$ for all $s\in I_\kappa$, and $\mathcal{Y}_\kappa=+\infty$ on $\R\setminus I_\kappa$.

Coupling $\int_{\mathbb{R}} \left(v \cdot x\right) \exp(-\mathcal{Y}_\kappa(x))\rmd x=0, \forall\kappa,$ with Theorem 1.1 from Harg\'e \cite{Harge} applied to each $\rmd\mu_{\mathcal{Y}_\kappa}(x):=\exp(-\mathcal{Y}_\kappa(x))\rmd x/\int_{\mathbb{R}}\exp(-\mathcal{Y}_\kappa(x))\rmd x$ in the mixture, the rest of the proof follows by the same arguments as the proof of Theorem \ref{genposbeta} above and will be omitted. We thus get
\[
\int_{\mathbb{R}}F\left(vx\right)\rmd\mu_{\mathcal{Y}}(x)\le \int_{\mathbb{R}} F\bigg(\frac{vx}{\sqrt{\alpha}}\bigg)\rmd\Gamma_1(x).
\]
We assume now that Theorem \ref{comparenon-convex} holds for $d-1\ge 0$, and we will prove it for $d\ge 2$. For simplicity of calculations, we will assume that $\int_{\mathbb{R}^d}\mathcal{V}(x)\rmd x=1$. We then write
\begin{equation}
\begin{split}
\label{inductionstepd}
 \int_{\mathbb{R}^d}F\left(v \cdot x\right)\rmd\mu_{\mathcal{Y}}(x)
 &=\int_{\mathbb{R}}\rmd x_d\frac{\int_{\mathbb{R}^{d-1}} F\bigg(v_dx_d+\sum_{i=1}^{d-1}v_i x_i\bigg)\mathcal{V}(x)\prod_{i=1}^{d-1}\rmd x_i}{\int_{\mathbb{R}^{d-1}}\mathcal{V}(x)\prod_{i=1}^{d-1}\rmd x_i}\int_{\mathbb{R}^{d-1}}\mathcal{V}(x)\prod_{i=1}^{d-1}\rmd x_i\\
 &\le \int_{\mathbb{R}}\rmd x_d\bigg(\int_{\mathbb{R}^{d-1}} F\bigg(v_dx_d+\sum_{i=1}^{d-1}\frac{v_i}{\sqrt{\alpha}} x_i\bigg)\rmd\Gamma_{d-1}(x_1, \ldots, x_{d-1})\bigg)\int_{\mathbb{R}^{d-1}}\mathcal{V}(x)\prod_{i=1}^{d-1}\rmd x_i\\
\end{split}
\end{equation}
where for the inequality we applied the induction hypothesis
in dimension $(d-1)$. Next, the function
\[
x_d\to \int_{\mathbb{R}^{d-1}} F\bigg(v_d x_d+\sum_{i=1}^{d-1}\frac{v_i}{\sqrt{\alpha}} x_i\bigg)\rmd\Gamma_{d-1}(x_1, \ldots, x_{d-1})
\]
is a convex function in $x_d$. Furthermore, the potential $U(x_d):=-\log \int_{\mathbb{R}^{d-1}}\mathcal{Y}(x)\prod_{i=1}^{d-1}\rmd x_i$ is $\alpha$-monotone and even by hypothesis.  We can therefore apply to \eqref{inductionstepd} the induction Step 1, and \eqref{BLRn} immediately follows. 

Next, we turn to part (b)(i). By Proposition \ref{p:trivialdecomp},  we have $\mathcal{W}(|x|)=-\log\int \exp(-\mathcal{W}_\kappa(|x|))\rho(\rmd\kappa)$, where $\mathcal{W}_\kappa\in C^2(I_\kappa)$ for some open interval $I_\kappa\subset\R$, $\mathcal{W}_\kappa''(s)\ge\alpha$ for all $s\in I_\kappa$, and $\mathcal{W}_\kappa=+\infty$ on $\R\setminus I_\kappa$. Furthermore, the derivatives  of the radial function  $\mc{W}_{\kappa}$ are given for $|x|\in I_\kappa\setminus\{0\}$ by
\begin{align}
    \nabla \mc{Y}_{\kappa}(x)=\frac{x}{|x|}\mc{W}'_{\kappa}(|x|),
    \nabla^2 \mc{Y}_{\kappa}(x)=
    \frac{\mc{W}'_{\kappa}(|x|)}{|x|}
    \left(\mathrm{Id}_d-\frac{x x^\top}{|x|^2}\right)
    +\mc{W}''_{\kappa}(|x|)\frac{xx^\top}{|x|^2}.
\end{align}
Using the bounds $\mc{W}'_{\kappa}(|x|)\geq \alpha |x|$ and $\mc{W}''_{\kappa}(|x|)\geq \alpha$
and that $xx^\top$ and $\mathrm{Id}_d-|x|^{-2}x x^\top$ are positive semidefinite, we find for $|x|\in I_\kappa\setminus\{0\}$ that
\begin{align}
    \nabla^2 \mc{Y}_{\kappa}(x)\geq
    \alpha\left(\mathrm{Id}_d-\frac{x x^\top}{|x|^2}\right)+\alpha\frac{xx^\top}{|x|^2}=\alpha \cdot\mathrm{Id}_d.
\end{align}
Extending the result to $\{0\}$ by continuity, there holds that $\mathcal{Y}_\kappa(x)- \alpha |x|^2/2$ is convex for all $\kappa$. The rest of the proof follows the same arguments as for $d=1$ from (a)(i), and will be omitted. 

\textbf{Step 2: Proof of (a)(ii) and (b)(ii)}\\
We start again with part (a)(ii). In view of Proposition \ref{p:trivialdecomp} and since $F$ is an odd function in at least one coordinate, the proof for the $L^2$ inequality follows by similar arguments to \cite[Theorem 17]{BartKlar} and will be omitted. To show the $L^1$ inequality, we can use for each term in the decomposition the $L^1$ inequality. That inequality for convex potentials that may take the value $+\infty$ follows for example from \cite[Theorem 1]{FaGoPro} coupled with the $L^1$ inequality for the standard Gaussian.

To show (b)(ii), we write again as in the proof of Theorem \ref{comparenon-convex} (b)(i)
\[
\mathcal{V}(x)=\e^{-\mathcal{X}(x)-\mathcal{Y}(x)}=\int \exp(-\mathcal{X}(x)-\mathcal{Y}_\kappa(x))\rho(\rmd\kappa).
\]
Define $\rmd\mu_\kappa(x):=\exp(-\mathcal{X}(x)-\mathcal{Y}_\kappa(x))\rmd x/\int_{\mathbb{R}^d}\exp(-\mathcal{X}(x)-\mathcal{Y}_\kappa(x))\rmd x$. Since $F$ is an odd function, we have $\int_{R^d}F(x)\rmd\mu_\kappa(x)=0$, for all $\kappa$. We have $L^2$ and $L^1$-Poincaré inequalities for each $\mu_\kappa$ (as follows again from \cite[Theorem 1]{FaGoPro} and the  $L^2$ and $L^1$ inequalities for the standard Gaussian), and then the rest of the proof follows by the same arguments as the proof of Theorem \ref{genposbetapoincare} above and will be omitted.

\end{proof}

\section{Existence of ergodic gradient Gibbs measures for \texorpdfstring{$\alpha$}{alpha}-monotone potentials}\label{s:existence}

The goal of this section is to prove Theorem \ref{t:existenceGGMs}, i.e. the existence of an ergodic tempered gradient Gibbs measure with zero tilt and satisfying the a priori bound \eqref{e:Gaussianmoments}.

For that purpose we proceed step by step, and first prove the existence of a shift-invariant tempered gradient Gibbs measure with zero tilt (but which is not necessarily ergodic) in Theorem \ref{existgibbs}, and then prove in Theorem \ref{existextr} that we can extract an ergodic component from it which is still a tempered gradient Gibbs measure with zero tilt. The proof of the two theorems is based on the strategy in \cite{MR2985173}, where existence of shift-covariant gradient Gibbs measures with zero tilt was proved for shift-covariant gradient Gibbs measures with interaction potentials which are strongly log-concave log-mixtures of quadratic growth. Here we improve this result by removing the upper bound on $V_\kappa''$, which requires only minor changes.

The main difficulty in this section is to show that the measures just constructed satisfy the stronger exponential moment bound \eqref{e:Gaussianmoments}, and even more, that every ergodic gradient Gibbs measure $\mu$ with zero tilt satisfies \eqref{e:Gaussianmoments} since there may be more than one such measure. In case of Theorem \ref{existgibbs} these are a straightforward consequence of Lemma \ref{l:expbrasc}, but the case of Theorem \ref{existextr} is considerably more challenging, as it is not clear whether extracting an ergodic component preserves all estimates simultaneously. So we proceed indirectly. The key technical statement, Lemma \ref{c:unique_disordered_Gibbs}, states that for each shift-invariant probability measure $\P$ on $(K^{\bzd},\Kcal^{\otimes\bzd})$ there is a shift-covariant disordered gradient Gibbs measure $\kappa\to\mu[\kappa]$ that satisfies the bounds \eqref{e:Gaussianmoments}. Applying that result with $\P$ equal to the $\kappa$ marginal $\bar\mu$ of the measure from Theorem \ref{existextr}, the uniqueness result of Theorem \ref{t:CK} then implies that the two shift-covariant gradient Gibbs measures are the same, i.e. that the ergodic measure from Theorem \ref{existextr}, or more generally the arbitrary ergodic measures from Theorem \ref{t:existenceGGMs}(b), already satisfied the bounds \eqref{e:Gaussianmoments}.

\subsection{Existence of shift-invariant tempered gradient Gibbs measures at zero tilt}
\label{constr}

To begin with, we recall the construction of an infinite volume shift-invariant gradient Gibbs measure. Assume the boundary condition $\psi(x)=0$, with $\xi_0(b)$ being the corresponding gradient one. Consider the gradient Gibbs measure $\mu_{\Lambda, \xi_0}$ and let us define the \textit{spatially averaged} measure $\hat\mu^0_{\Lambda}$ on gradient configurations given by 
\begin{equation}
\label{heldfixed}
\hat\mu^{0}_{\Lambda}:=\frac{1}{|\Lambda|}\sum_{w\in \Lambda}\mu_{\Lambda+w, \xi_0},
\end{equation}
where we defined $\Lambda+w:=\{z+w:z\in\Lambda\}$. We note that in \eqref{heldfixed}, the volumes $\Lambda+x$ 
are shifted around. 

This is the construction of shift-invariant Gibbs measures given in \cite{MR2807681}, in formula (5.20) from Chapter $5.2$; the construction was used in \cite{MR2985173} to prove existence of shift-covariant tempered gradient Gibbs measures when $A_1s^2+B _1\le V(s)\le A_2s^2+B _2$, for all $s, B_1,B_2\in\mathbb{R}$ and $A_1,A_2>0.$ Below we state the existence result for shift-invariant measures without this quadratic growth assumption, and under the stronger needed exponential moment bounds of \eqref{expexist}. We provide a short proof sketch for completeness.

\begin{theorem}[Existence for $\alpha$-monotone potentials of shift-invariant tempered gradient Gibbs measures at zero tilt]
\label{existgibbs}

Let $d\ge 1$. Assume that $V\in C(\mathbb{R})$ is an $\alpha$-monotone even potential. Then there exists a subsequence $(m_i)_{i\ge 1}$ such that $\hat\mu^{0}_{\Lambda_{m_i}}$
converges weakly to a measure $\mu$ as $i\to \infty$, which is a shift-invariant gradient Gibbs measure. 

Fix now $0<\delta<\alpha/2$ and $\epsilon>0$ arbitrarily. Let $f\in C_c^\infty(\R^d,\R^d)$ and take $F_\epsilon$ as in \eqref{e:observable}. Then for all $\lambda\in\R$
\begin{equation}
\label{expexist}
\sup_{b\in(\zd)^*}\mathbb{E}_{\mu}(\exp((\alpha/2-\delta)\eta^2(b)))\le\frac{\sqrt{\alpha}}{\sqrt{2\delta}}~~\mbox{and}~~\mathbb{E}_{\mu}(\exp(\lambda F_\eps(\eta)))\le \e^{\frac{C(d, f)}{\alpha}\lambda^2},~~\mbox{where}~~C(d, f)>0.
\end{equation}
\end{theorem}

\begin{proof}
We will start by showing tightness of the measure $\hat\mu^{0}_{\Lambda_m}$. By \cite[Corollary 4.13]{MR2807681}, tightness will be implied by showing 
\begin{equation}
\label{tightcrit}
\limsup_{N\to\infty}\sup_{m\ge 1} {\hat\mu}^0_{\Lambda_m}\left(\eta: |\eta(b)|\ge N\right)=0.
\end{equation}
For this purpose, and since by Chebyshev's inequality we have 
\begin{equation}
\label{Chebtight}
{\hat\mu}^0_{\Lambda_m}\left(\eta: |\eta(b)|\ge N\right)\le \frac{{\hat\mu}^0_{\Lambda_m}(\eta^2(b))}{N^2},
\end{equation}
we claim that it is sufficient to find an upper bound for all $m$ of type
\begin{equation}
\label{unifbalpha}
\int \bigg(\sum_{\substack{x,y\in\mathbb{Z}^d\\|x-y|=1}}(\phi(x)-\phi(y))^2
\bigg)\mu_{\Lambda_m, \xi_0}(\d\phi)\le |\Lambda_m| C(\alpha,d).
\end{equation}
For simplicity of notation, take $j=1$ and $x=0$. Then we have
 \begin{equation}
 \label{equaltiltusqr}
 \begin{split}
{\hat\mu}^0_{\Lambda_m}\left(\left(\eta(b_{0,1})\right)^2\right)
&=\frac{1}{|\Lambda_m|}\sum_{w\in\Lambda_m}\mu_{{\Lambda}_{m}+w, \xi_0}\left(\left(\phi(e_1)-\phi(0)\right)^2\right)=\frac{1}{|\Lambda_m|}\sum_{w\in\Lambda_m}\mu_{{\Lambda}_{m}, \xi_0}\left(\left(\phi(e_1-w)-\phi(-w)\right)^2\right)\\
&=\frac{1}{|\Lambda_m|}\sum_{x\in\Lambda_m}\mu_{{\Lambda}_{m}, \xi_0}\left(\left(\phi(e_1+x)-\phi(x)\right)^2\right)\le\frac{1}{|\Lambda_m|}\mu_{{\Lambda}_{m}, \xi_0}\bigg(\sum_{x, y\in\zd, |x-y|=1}\big(\phi(x)-\phi(y)\big)^2\bigg),
\end{split}
 \end{equation}
 where for the second equality we shifted all the parameters $\phi(x)\to \phi(x-w)$.

 Considering \eqref{Chebtight} and \eqref{equaltiltusqr}, \eqref{unifbalpha} implies tightness of ${\hat\mu}^0_{\Lambda_m}$, which in turn means that there exists a subsequence $(m_i)_{i\ge 1}$ such that $\hat\mu^{0}_{\Lambda_{m_i}}$ converges as $m_i\to \infty$ weakly to $\mu$. We move now to the proof of \eqref{unifbalpha}.

In view of Theorem \ref{genposbeta}, we write
\begin{equation*}
\label{disordBLexist}
\int \bigg(\sum_{\substack{x,y\in\mathbb{Z}^d\\|x-y|=1}}(\phi(x)-\phi(y))^2
\bigg)\mu_{\Lambda_m, \xi_0}(\d\phi)\le \mathbb{E}_{\mu^{\Delta_{\alpha}}_{\Lambda_m, \xi_0}}\bigg(\sum_{\substack{x,y\in\mathbb{Z}^d\\|x-y|=1}}(\phi(x)-\phi(y))^2
\bigg)<C(\alpha, d)|\Lambda_m|,
\end{equation*}
where for the last inequality we applied \eqref{logmixtureexp}. This shows \eqref{unifbalpha}.

The shift-invariance and DLR property follow similarly to the proofs of Lemma 3.9 and Lemma 3.10 from \cite{MR3383338} and to the more involved setting from Lemma \ref{c:unique_disordered_Gibbs} below, so their proofs will be omitted. Furthermore, for all $0<\delta< \alpha/2$, we have by \eqref{logmixtureexp}
\[\sup_{b\in(\zd)^*}\mu(\exp((\alpha/2-\delta)\eta^2(b)))\le\sup_{\substack{b\in(\zd)^*\\ \Lambda_{m_i}\subset\zd}}\liminf_{i\to\infty}{\hat\mu}_{\Lambda_{m_i}}^0(\exp((\alpha/2-\delta)\eta^2(b)))\le\frac{\sqrt{\alpha}}{\sqrt{2\delta}}.\]
The second inequality in  \eqref{expexist} can be shown similarly. This concludes the claims.
\end{proof}

\begin{corollary}
\label{directilt}
Let $d\ge 1$. Assume that $V\in C(\mathbb{R})$ is an $\alpha$-monotone even potential. For all $j\in\{1,2,\ldots, d\}$ and $x\in\zd$, let $b_{x,j}:=(x, x+e_j)\in (\zd)^*$. Then $\mu$ constructed in Theorem \ref{existgibbs} has tilt zero, that is, it satisfies
\begin{equation}
\label{tiltu}
\mu(\eta(b_{x,j}))=0, ~~j=1,\ldots,d.
\end{equation}
\end{corollary}
\begin{proof}
Fix $N>0$. Since $V$ is even, we have similarly to \eqref{lalagibbs} for all sets $\Lambda_{m_i}\subset\zd$
\[
{\hat\mu}_{\Lambda_{m_i}}^0(\eta(b_{x,j}) \mathds{1}_{|\eta(b_{x,j})|<N})=0.
\]
Using for the second inequality below Chebyshev's inequality and using \eqref{expexist} for the third, we get
\begin{align*}
|\mu(\eta(b_{x,j}))|&\le |\mu(\eta(b_{x,j}) \mathds{1}_{|\eta(b_{x,j})|<N})|+|\mu(\eta(b_{x,j}) \mathds{1}_{|\eta(b_{x,j})|\ge N})|\\
&\le\lim_{i\to\infty}|{\hat\mu}_{\Lambda_{m_i}}^0(\eta(b_{x,j}) \mathds{1}_{|\eta(b_{x,j})|<N})| + \frac{\mu(\eta^2(b_{x,j}))}{N}\le\frac{C}{N},
\end{align*}
where for the second inequality we applied a continuous truncation of the indicator function. Taking now $N\to\infty$ completes the argument.
\end{proof}

\subsection{Existence of shift-invariant ergodic gradient Gibbs measure with zero tilt}
In this subsection we will upgrade the result from Theorem \ref{existgibbs} to prove existence of an ergodic gradient Gibbs measure with the same properties.

Existence of ergodic gradient Gibbs measures with given tilt is a much more subtle question than that of existence of shift-invariant gradient Gibbs measures with that tilt. It does not follow from tightness of measures, but requires more delicate estimates via the Brascamp-Lieb inequality. Before we prove the main result from Theorem \ref{existextr}, we will need the result of Theorem \ref{existergod} of existence of gradient Gibbs measures with direction-averaged tilt zero. The property of direction-averaged tilt will ensure that the constructed ergodic Gibbs measure selected in Theorem \ref{existextr} by the ergodic decomposition satisfies the correct zero tilt, which is the crucial step in the proof.
\begin{theorem}
\label{existergod}
Fix $d\ge 1$. Let for all $j\in\{1,2,\ldots,d\}$ 
$$E_j:=\{ \eta~|~\lim_{n\rightarrow\infty}\frac{1}{|\Lambda_n|}\sum_{x\in\Lambda_n}\eta(b_{x, j})=0\},$$
along the sequence of volumes $\Lambda_n$, where $b_{x,j}:=(x, x+e_j)\in (\zd)^*$. 

Assume that $V\in C(\mathbb{R})$ is an $\alpha$-monotone even potential, or more generally, that $V$ is even and satisfies a Brascamp-Lieb inequality as in \eqref{BL111}. Then $\mu$ constructed in Theorem \ref{existgibbs} has direction-averaged tilt zero, that is
\begin{equation}
\label{firstlimit11}
\mu(E_j)=1,~j\in\{1,2,\ldots,d\}.
\end{equation} 
\end{theorem}
\begin{proof}
We will use in our proof the construction of the infinite volume shift-invariant gradient Gibbs measure, as detailed in Subsection \ref{constr} above. The key idea  to show \eqref{firstlimit11} is to reduce the calculations to those for the GFF, for which purpose we will use the Brascamp-Lieb inequality. 
 
\vspace{2mm}

Fix $j\in \{1,2,\ldots, d\}$. In order to show \eqref{firstlimit11}, it suffices to show that 
\begin{equation}
\label{firstlimit1ainit}
\lim_{n\rightarrow\infty}\bigg|\frac{1}{|\Lambda_n|}\sum_{x\in\Lambda_n}\eta(b_{x, j})\bigg|=0~~~\mu ~~\mbox{a.s.}.
\end{equation}
However, for the purposes of our computations below, we will find it easier to prove the following statement which will imply \eqref{firstlimit1ainit}.
\begin{equation}
\label{firstlimit1a}
\mu\bigg(\lim_{n\rightarrow\infty}\bigg(\frac{1}{|\Lambda_n|}\sum_{x\in\Lambda_n}\eta(b_{x, j})\bigg)^2\bigg)=0.
\end{equation}
Note that $\mu$-a.s. existence of the limit $\frac{1}{|\Lambda_n|}\sum_{x\in\Lambda_n}\eta(b_{x, j})$ follows from a multiparameter version of Kingman's  subadditive ergodic theorem \cite[Theorem 2.4]{Akcoglu-Krengel}.
By Fatou's lemma, it follows that to show \eqref{firstlimit1a} it is enough to prove that
\begin{equation}
\label{firstlimit120a}
\liminf_{n\rightarrow\infty}\mu\bigg(\bigg(\frac{1}{|\Lambda_n|}\sum_{x\in\Lambda_n}\eta(b_{x,j})\bigg)^2\bigg)=0.
\end{equation}
By the lower semicontinuity of $\big|\frac{1}{|\Lambda_n|}\sum_{x\in\Lambda_n}\eta(b_{x,j})\big|$ and by the weak convergence of $\hat{\mu}_{\Lambda_{m_i}, 0}$ to $\mu$ as $i\to\infty$, for a subsequence $(m_i)_{i\ge 1}$, we then have
\begin{equation}
\label{boundliminf}
\begin{split}
\mu\bigg(\bigg(\frac{1}{|\Lambda_n|}\sum_{x\in\Lambda_n}\eta(b_{x, j})\bigg)^2\bigg)
&\le\liminf_{i\to\infty}\hat{\mu}_{\Lambda_{m_i}, 0}\bigg(\bigg(\frac{1}{|\Lambda_n|}\sum_{x\in\Lambda_n}\eta(b_{x, j})\bigg)^2\bigg)\\
&=\liminf_{i\rightarrow\infty}\frac{1}{|{\Lambda}_{m_i}|}\sum_{w\in {\Lambda}_{m_i}}\mu_{{\Lambda}_{m_i}+w, \xi_0}\bigg(\bigg(\frac{1}{|\Lambda_n|}\sum_{x\in\Lambda_n}\eta(b_{x, j})\bigg)^2\bigg),
\end{split}
\end{equation}
where we recall that $\xi_0$ was defined as $\xi_0(b):=\nabla\phi_0(b)$ for all $b\in (\zd)^*$.  

Fix $m_i,n\in\mathbb{N}$. By means of Theorem \ref{genposbeta}, we have
\begin{equation}
\label{BLsplit}
\begin{split}
&\mu_{{\Lambda}_{m_i}+w,\xi_0}\bigg(\bigg(\frac{1}{|\Lambda_n|}\sum_{x\in\Lambda_n}\eta(b_{x, j})\bigg)^2\bigg)\\
&\le\mu^{\Delta_{\alpha}}_{{\Lambda}_{m_i}+w, \xi_0}\bigg(\bigg(\frac{1}{|\Lambda_n|}\sum_{x\in\Lambda_n}\eta(b_{x,j})\bigg)^2\bigg)=\mu^{\Delta_{\alpha}}_{{\Lambda}_{m_i}, \xi_0}\bigg(\bigg(\frac{1}{|\Lambda_n|}\sum_{x\in\Lambda_n-w}\eta(b_{x,j})\bigg)^2\bigg)\\
&=\frac{1}{|\Lambda_n|^2}\sum_{\{x,y\in \left(\Lambda_n-w\right): b_{x, j}, b_{y, j}\in {\bar\Lambda}^*_{m_i}\}}\mu^{\Delta_{\alpha}}_{{\Lambda}_{m_i}, \xi_0}\left(\eta(b_{x,j})\eta(b_{y,j})\right)\\
&\le \frac{1}{\alpha|\Lambda_n|^2}\sum_{\{x,y\in \left(\Lambda_n-w\right): b_{x, j}, b_{y, j}\in {\bar\Lambda}^*_{m_i}\}}\left|\nabla\nabla' G_{\Lambda_{m_i}}((x, x+e_j), (y, y+e_j))\right|\\
&\le \frac{C(d)}{\alpha}\frac{1}{|\Lambda_n|}\int_1^n\frac{r^{d-1}}{r^{d}}\d r\le \frac{C_1(d)}{\alpha}\frac{\log n}{|\Lambda_n|},
\end{split}
\end{equation}
where for the second inequality we applied \eqref{e:mixedsecondder}, and where $C(d), C_1(d)>0$.

Plugging now \eqref{BLsplit} in \eqref{boundliminf}, we obtain
\[
\frac{1}{|\Lambda_{m_i}|}\sum_{w\in\Lambda_{m_i}}\mu_{{\Lambda}_{m_i}+w,\xi_0}\bigg(\bigg(\frac{1}{|\Lambda_n|}\sum_{x\in\Lambda_n}\eta(b_{x,j})\bigg)^2\bigg)\le \frac{C_1(d)}{\alpha}\frac{\log n}{n^d}.
\]
Taking now $i\to\infty, n\to\infty$ in the above produces the desired result.
\end{proof}

\begin{theorem}{(Existence of ergodic gradient Gibbs measure with zero tilt)}
\label{existextr}
\\
Fix $d\ge 1$. Assume that $V\in C(\mathbb{R})$ is an $\alpha$-monotone even potential, or more generally, that $V$ is even and satisfies a Brascamp-Lieb inequality as in \eqref{BL111}. Then there exists at least one shift-invariant ergodic gradient Gibbs measure $\gamma$ with 
\begin{equation}
\label{firstlimit111}
\gamma(E_j)=1,~j\in\{1,2,\ldots,d\}.
\end{equation} 
Furthermore, $\gamma$ satisfies \eqref{tiltu} and it is tempered. 
\end{theorem}
\begin{proof}
Let $\mu$ be the measure constructed in Theorem \ref{existgibbs}. By the ergodic decomposition of shift-invariant Gibbs measures, as explained in \cite[Theorem 14.17]{MR2807681}, we write
\begin{equation}
\label{ergdecomp}
\mu=\int \gamma \d{\mathcal M}(\gamma),
\end{equation}
 where ${\mathcal M}$ is a measure supported on ergodic gradient Gibbs measures. Then we have by means of Theorem \ref{existergod} that
$\mu(\cup_{j=1}^d E^c_j)=0$ and hence ${\mathcal M}(\gamma:\gamma(\cup_{j=1}^dE^c_j)\neq 0)=0$. But by ergodicity for $\mathcal{M}$-a.e. $\gamma$ we have $\gamma(\cup_{j=1}^dE^c_j)\in\{0,1\}$. 
Therefore we must have ${\mathcal M}(\gamma:\gamma(\cup_{j=1}^dE^c_j)=0)=1$. In particular the set $\{\gamma:\gamma(\cup_{j=1}^dE^c_j)=0\}$ is nonempty, so ${\mathcal M}$-a.s. every ergodic component $\gamma$ satisfies $\gamma(E_j)=1$ for all $j=1,\ldots, d$.  To show temperedness, we use that in view of \eqref{ergdecomp} and \eqref{expexist}, ${\mathcal M}$-a.s. every ergodic component $\gamma$ is tempered, simultaneously for the countable set of bonds. It follows that we can choose from the intersection of these sets a tempered $\gamma$ satisfying \eqref{firstlimit111}.

Finally, we show that \eqref{firstlimit111} implies \eqref{tiltu}. To start with, $\lim_{n\rightarrow\infty}\frac{1}{|\Lambda_n|}\sum_{x\in\Lambda_n}\eta(b_{x,j})$ can be shown to exist $\gamma$-a.s. by subadditivity arguments. Additionally, by \eqref{expexist}, $\gamma\big(\frac{1}{|\Lambda_n|}\sum_{x\in\Lambda_n}|\eta(b_{x,j})|\big)<\int\gamma(\d\eta)|\eta(b_{0,j})|<\infty$. Therefore, we have by means of  $\gamma(E_j)=1$  and of the $L^1$ multiparameter ergodic theorem that
\[
0=\gamma\bigg(\lim_{n\rightarrow\infty}\frac{1}{|\Lambda_n|}\sum_{x\in\Lambda_n}\eta(b_{x,j})\bigg)=\lim_{n\rightarrow\infty}\gamma\bigg(\frac{1}{|\Lambda_n|}\sum_{x\in\Lambda_n}\eta(b_{x,j})\bigg)=\gamma(\eta(b_{0,j})).
\]
\end{proof}

\begin{proof}[Proof of Theorem \ref{t:existenceGGMs}(a)]
From Theorem \ref{existgibbs} and Corollary \ref{directilt}, we have existence of at least one gradient Gibbs measure $\mu$ satisfying \eqref{e:Gaussianmoments}. From Theorem \ref{existextr}, we can construct an ergodic and tempered measure with zero tilt. Part (b) of the theorem is proved after Lemma \ref{c:unique_disordered_Gibbs} below, which is used in the proof.
\end{proof}
\begin{remark}
The arguments above can be adapted to also show, for $\alpha$-monotone continuous and even potentials $V$, existence of ergodic Gibbs measures $\nu$ with zero tilt in $d\ge 3$.  
\end{remark}

\subsection{Exponential bound for shift-covariant gradient Gibbs measures}

The main result of this section is the proof of Theorem \ref{t:existenceGGMs}(b) below, used in the proof of Theorem \ref{t:existenceGGMs}, and it is a crucial ingredient for the derivation of the exponential moment bound needed in the proof of the scaling limit.

To show Theorem \ref{t:existenceGGMs}(b), we will apply the following lemma of a.s. existence of shift-covariant gradient Gibbs measures satisfying an exponential-moment bound. The proof of the lemma is based on the ideas in \cite{MR2985173}, where a.s. existence of shift-covariant gradient Gibbs measures satisfying temperedness was proved under more restrictive assumptions on the potential. For clarity and since the lemma is used also in the proof of Theorem \ref{th:extremality}, we provide a sketch of the arguments below.

\begin{lemma}
\label{c:unique_disordered_Gibbs}
Let $\P$ be a shift-invariant probability measure on $(K^{\bzd},\Kcal^{\otimes\bzd})$ such that $\P$-a.s. $\kappa(b)=\kappa(-b)$  for all $b\in\bzd$. Assume that for all $b\in (\zd)^*$, the map $(\kappa(b),s)\mapsto V_{\kappa(b)}(s)$ is jointly measurable, and for each $\kappa(b)$ the function $V_{\kappa(b)}$ is in $C(\R)$ and satisfies $V_{\kappa(b)}(s)=V_{\kappa(b)}(-s)$ for all $s\in\R$. Assume also that there exists $c_V>0$ such that $s\to V_{\kappa(b)}(s)- c_Vs^2/2$ is convex, for all $b\in (\zd)^*$ and $\kappa\in K^\bzd$. 

Then there exists \textbf{at least one} (up to $\mathbb{P}$-null sets) shift-covariant disordered gradient Gibbs measure $\kappa\rightarrow\mu[\kappa]$, $\mu[\kappa] \in {\cal P}(\chi)$,  such that the annealed measure $\mu^{\text{av}}(\rmd\eta):=\int\mu[\kappa](\d\eta)\d\mathbb{P}(\kappa)$ has tilt $0$, that is, it satisfies \eqref{tiltu}. 

Recall now the definition of $E_j, j\in\{1,2,\ldots,d\}$, as given in Theorem \ref{existergod}. Then the measure $\mu[\kappa]$ satisfies for a.e. $\kappa$
\begin{equation}
\label{firstlimit12b}
\mu[\kappa](E_j)=1,~j\in\{1,2,\ldots,d\}.
\end{equation} 
Fix $0<\delta<c_V/2$ and $\eps>0$ arbitrarily. Let $f\in C_c^\infty(\R^d,\R^d)$ and take $F_\epsilon$ as in \eqref{e:observable}. Then the measure $\mu[\kappa]$ satisfies for all $\lambda\in\R$ and for $\mathbb{P}$-almost all $\kappa\in K^\bzd$, for some $\epsilon$-independent $C(d, f)>0$
\begin{equation}
\label{logmixtureexpinfrand}
\sup_{b\in (\zd)^*}\mathbb{E}_{\mu[\kappa]}(\exp((c_V/2-\delta)\eta^2(b)))\le \frac{\sqrt{c_V}}{\sqrt{2\delta}}~~\mbox{and}~~\mathbb{E}_{\mu^{\text{av}}}(\exp(\lambda F_\eps(\eta)))\le \e^{\frac{C(d, f)}{c_V}\lambda^2}.
\end{equation}

Assume additionally that $\P$ is an ergodic probability measure such that the function $V_{\kappa(b)}$ is in $C^2(\R)$ and there exists $C_V>0$ with $V''_{\kappa(b)}\le C_V$, for all $b\in (\zd)^*$ and $\kappa\in K^\bzd$. Then, by Theorem \ref{t:CK}(a), $\kappa\to\mu[\kappa]$ is a.s. unique such that it satisfies \eqref{firstlimit12b} and $\mu^{\text{av}}$ satisfies \eqref{intcond1a}. Finally, it can be shown that $\mu^{\text{av}}$ is ergodic.
\end{lemma}

\begin{proof}
Similarly to \eqref{heldfixed}, let us define for all $\xi_0$ and all $\kappa$ the measure $\hat\mu^0_{\Lambda}[\kappa]$ on gradient configurations by
\begin{equation}
\label{heldfixedrand}
\hat\mu^0_{\Lambda}[\kappa]:=\frac{1}{|\Lambda|}\sum_{x\in \Lambda}\mu_{\Lambda+x, \xi_0}[\kappa].
\end{equation}
where we recall that $(\tau_v\eta)(b)=\eta(b-v)$ for all $b\in\bzd$.

\textbf{Step 1.} We will first show that the sequence $\hat\mu^0_{\Lambda_m}[\kappa]$ is tight and that it converges weakly to a random measure $\kappa\to\mu[\kappa]$. As a first step, it can be shown similarly to the proof of \eqref{unifbalpha} from Theorem \ref{existgibbs} that there exists $C(c_V,d)>0$ such that for all $\Lambda_m\subset\zd$, $b\in (\zd)^*$ and $\kappa$, there holds

\begin{equation}
\label{unifbalpharand}
\hat\mu^0_{\Lambda_m}[\kappa]\left(\eta^2(b)\right)=\frac{1}{|\Lambda_m|}\sum_{x\in \Lambda_m}\mu_{\Lambda_m+x, \xi_0}[\kappa]\left(\eta^2(b)\right)\le C(c_V,d).
\end{equation}
We move now to the proof of tightness.  The main technical point here is to construct the limit in such a way that it is $\mathbb{P}$-a.s. measurable; this takes great care since an application of Prokhorov's theorem is not guaranteed to produce for almost all $\kappa$ the same deterministic subsequence $(m_r)_{r\in\mathbb{N}}$ along which $\hat\mu_{\Lambda_{m_r}}[\kappa]$ converges weakly to a limiting random measure. Constructing such a deterministic sequence will be crucial for establishing the measurability of the map $\kappa\to\mu[\kappa]$, and will be facilitated by \cite[Theorem 1a]{KOM}, which we state below. 
\begin {proposition}
\label{kom}
If $(\zeta_n)_{n\in\mathbb{N}}$ is a sequence of real-valued random variables with $\sup_n\E(|\zeta_n|)<\infty$, there exists a subsequence $\{\theta_n\}_{n\in\mathbb{N}}$ of the sequence $\{\zeta_n\}_{n\in\mathbb{N}}$ and an integrable random variable $\theta$ such that for any arbitrary subsequence $\{\tilde{\theta}_n\}_{n\in\mathbb{N}}$ of the sequence $\{\theta_n\}$, we have almost surely that
$$\lim_{n\rightarrow\infty}\frac{\tilde{\theta}_1+\tilde{\theta}_2+\ldots+\tilde{\theta}_n}{n}=\theta.$$
\end{proposition}

Let $(f_i)_{i\in\N}$ be a countable collection of functions in $C_b(\chi)$, such that a sequence of probability measures $\mu_n \in P(\chi)$ converges weakly to $\mu\in P(\chi)$ if and only if $\mu_n(f_i)\to \mu(f_i)$ for all $i\in\N$; note that such a countable family $(f_i)_{i\in\N}$ in $C_b(\chi)$ is explicitly given for example in the general setting of separable and complete metric spaces in Lemma 1.1 from \cite{Kal}. It will therefore suffice for our purposes to show that there exists a deterministic sequence $(m_i)_{i\in\mathbb{N}}$ in $\mathbb{N}$ such that for $\P$-almost every $\kappa$
\begin{equation}
\label{komlostight}
\tilde\mu_{\ell}[\kappa]:=\frac{1}{\ell}\sum_{r=1}^\ell {\hat\mu}^0_{\Lambda_{m_{r}}}[\kappa]
\end{equation}
satisfies $\tilde\mu_{\ell}[\kappa](f_i) \to \hat\mu[\kappa](f_i)$ almost surely for each $i\in\N$. This will imply weak convergence of $\tilde\mu_{\ell}[\kappa]$ to a measurable random gradient measure $\kappa\to\mu[\kappa]$. For all $m\in\N$ and $b\in (\Z^d)^*$, define 
\begin{equation}\begin{split}\label{heldfixed2}
&X_{m, b}[\kappa]:=\hat\mu_{\Lambda_m}^0[\kappa]\left(\left(\eta(b)\right)^2\right).
\end{split}
\end{equation}

Since \eqref{unifbalpharand} holds and since $f_i, i\in \N,$ are bounded functions, we can apply Proposition \ref{kom} to the sequences $(X_{m, b}[\kappa])_{m, b}$ and $(\hat\mu_{\ell}[\kappa](f_i))_{\ell, i}$, and argue via a Cantor diagonalisation argument as in Theorem 3.8 from \cite{MR2985173}, to show that there exist a deterministic sequence $(m_r)_{r\in\N}$ in $\N$ and random variables $(\omega_b[\kappa])_{b\in (\Z^d)^*}$ and $(\iota_i[\kappa])_{i\in\N}$ such that for $\P$-almost every $\kappa$, for all $b\in (\Z^d)^*$ and all $i\in\N$, we have
\begin{equation}
\label{limkom}
\lim_{\ell\uparrow \infty}\frac{1}{\ell}\sum_{r=1}^\ell X_{m_{r}, b}[\kappa]=\omega_{b}[\kappa]~~\mbox{and}~~\lim_{\ell\uparrow \infty}\frac{1}{\ell}\sum_{r=1}^\ell\hat\mu_{\Lambda_{m_r}}^0[\kappa](f_i)=\iota_i[\kappa],
\end{equation} 
where $\sup_{\ell}\frac{1}{\ell}\sum_{r=1}^\ell X_{m_{r}, b}[\kappa]\le C(c_V, d)<\infty$. The first limit permits to show via \eqref{tightcrit} that for $\P$-a.e. $\kappa$, there exists a (possibly) \textit{random} subsequence $(\ell'[\kappa])$ of $(\ell)_{\ell\in\N}$ such that $\big({\tilde\mu}_{\ell'[\kappa]}[\kappa]\big)_{\ell'[\kappa]}$ is tight and converges weakly to a random measure $\mu[\kappa]$. Moreover, we have ${\tilde\mu}_{\ell'[\kappa]}[\kappa](f_i)\rightarrow\mu[\kappa](f_i)$ for all $i\in\N$. The random subsequence $(\ell'[\kappa])$ is used only for tightness, and will become non-random as we return now as follows to the deterministic sequences $(\ell)_{\ell\in\N}$ and $(m_r)_{r\in\N}$.  Due to the second limit in \eqref{limkom}, and by the uniqueness of the limit point, we get $\iota_i[\kappa]=\mu[\kappa](f_i)$ for all $i$. As $\tilde\mu_{\ell}[\kappa](f_i)\rightarrow \mu[\kappa](f_i)$, it follows that $\tilde\mu_\ell[\kappa]$ converges a.s. to a random measure $\mu[\kappa]$.

\textbf{Step 2.} We show here that $\kappa\to\mu[\kappa]$ is shift-covariant, that is, for all $v\in\Z^d$ and for all $F\in C_b(\chi)$ 
\[
\mu[\kappa](F\circ\tau_v)=\mu[\tau_v\kappa](F).
\]

Fix $v\in\zd$. Then we have, similarly to the proof of Lemma 3.10 from \cite{MR2985173}
\begin{equation}
\label{siebenunddreissig1}
\begin{split}
\mu[\kappa](F\circ\tau_{v})-\mu[\tau_v\kappa](F)&=\lim_{\ell\to\infty}\frac{1}{\ell}\sum_{i=1}^\ell
\frac{1}{|\Lambda_{m_i}|}\Bigl(\sum_{x\in \Lambda_{m_i}}\mu_{\Lambda_{m_i}+x, \xi_0}[\kappa](F\circ\tau_{v})
-\sum_{x\in \Lambda_{m_i}}\mu_{\Lambda_{m_i}+x, \xi_0}[\tau_v\kappa](F)\Bigr)\\
&=\lim_{\ell\to\infty}\frac{1}{\ell}\sum_{i=1}^\ell
\frac{1}{|\Lambda_{m_i}|}\Bigl(\sum_{x\in \Lambda_{m_i}}\mu_{\Lambda_{m_i}+x, \xi_0}[\kappa](F\circ\tau_{v})
-\sum_{x\in \Lambda_{m_i}}\mu_{\Lambda_{m_i}+x-v, \xi_0}[\kappa](F\circ \tau_v)\Bigr)\\
&=\lim_{\ell\to\infty}\frac{1}{\ell}\sum_{i=1}^\ell
\frac{1}{|\Lambda_{m_i}|}\Bigl(\sum_{x\in \Lambda_{m_i}}\mu_{\Lambda_{m_i}+x, \xi_0}[\kappa](F\circ\tau_v)
-\sum_{x\in \Lambda_{m_i}-v}\mu_{\Lambda_{m_i}+x, \xi_0}[\kappa](F\circ \tau_v)\Bigr).
\end{split}
\end{equation}
The second equality in \eqref{siebenunddreissig1} holds because we can make the change of variables $\eta(b)\to\eta(b-v)$ in the second term of the first equality, without changing the value of the integral; in view of \eqref{e:disordered_H}, this allows to show that $\mu_{\Lambda_{m_i}+x, \xi_0}[\tau_v\kappa](F)=\mu_{\Lambda_{m_i}+x-v, \xi_0}[\kappa](F\circ \tau_v)$, for all $x\in\Lambda_{m_i}$. We observe now that most terms on the right-hand side of \eqref{siebenunddreissig1} cancel. Therefore, for a bounded function $F$ such that $\Vert F\Vert_{\infty} \leq C(F)$ for some $C(F)>0$, we have
\begin{equation}
\label{2}
|\mu[\kappa](F\circ\tau_{v})-\mu[\tau_v\kappa](F)|\leq 
\lim_{\ell\uparrow\infty}\frac{C(F)}{\ell}\sum_{i=1}^\ell\frac{|{\Lambda}_{m_i} \triangle ({\Lambda}_{m_i} -v)|}{|{\Lambda}_{m_i}|},
\end{equation}
where we denoted by $\Delta$ the symmetric difference of the sets $\Lambda$ and $\Lambda-v$. Since for a fixed $v$, $|\Lambda_{m_i} \triangle (\Lambda_{m_i}-v)|$  
goes to zero when divided by $|\Lambda_{m_i}|$, this implies that \eqref{2}
goes to zero also. This proves the shift-covariance.

\textbf{Step 3.} Wrapping up the proof

Next on the list is to show that $\mu[\kappa]$ is a gradient Gibbs measure for $\mathbb{P}$-almost all $\kappa$, that is, $\mu[\kappa]$ satisfies the DLR equations. The argument follows similarly to the one of Step 2 above and by the same reasoning as in Lemma 3.9 from \cite{MR2985173}, by using the DLR property in finite volume and that for each fixed
finite volume, the fraction of translated volumes whose boundary intersects it vanishes. Therefore the Gibbs kernels pass to the limit.

Next, since $V_{\kappa(b)}$ is even for all $b\in (\zd)^*$ and $\kappa\in K^\bzd,$ $\mu^{\mbox{av}}$ has expected zero tilt by a similar argument to the one from Corollary \ref{directilt}.

We will briefly turn now to the last two items on the list and give a proof of \eqref{firstlimit12b} and \eqref{logmixtureexpinfrand}.

To show \eqref{firstlimit12b}, it is sufficient to prove
\[\int\mu[\kappa](E_j)\d\P(\kappa)=1,~j\in\{1,2,\ldots,d\}.\]
Since $s\to V_{\kappa(b)}(s)- c_Vs^2/2$ is convex for all $b\in (\zd)^*$ and $\kappa\in K^\bzd$, $\mu_{\Lambda, \xi_0}[\kappa]$ satisfies the Brascamp-Lieb inequality for all $\Lambda\subset\zd$, so the proof follows by the same arguments as in Theorem \ref{existergod}, and will be omitted. 

Finally, we will show \eqref{logmixtureexpinfrand}. Recalling \eqref{komlostight}, we have for a.s $\kappa$ and $0<\delta<c_V/2$
\[
\mu[\kappa](\exp((c_V/2-\delta)\eta^2(b)))\le\lim\inf_\ell\tilde\mu_{\ell}[\kappa](\exp((c_V/2-\delta)\eta^2(b))).
\]
Applying now \eqref{logmixtureexp} to each summation term in $\tilde\mu_{\ell}[\kappa]$ with edge-dependent family $V_{\kappa(b)}$, we write for all $m_i$, $x\in\Lambda_{m_i}$, and $0<\delta<c_V/2$
\[
\mu_{\Lambda_{m_i}+x, \xi_0}[\kappa](\exp((c_V/2-\delta)\eta^2(b)))\le \frac{\sqrt{c_V}}{\sqrt{2\delta}}.
\]
This implies the first inequality in \eqref{logmixtureexpinfrand}. The second inequality in \eqref{logmixtureexpinfrand} can be argued similarly to the above, in view of \eqref{logmixtureexplinear} holding for $\mu_{\Lambda_{m_i}+x, \xi_0}[\kappa]$. This allows us to conclude the proof. 
\end{proof}

\begin{proof}[Proof of Theorem \ref{t:existenceGGMs}(b)]
 Take an arbitrary shift-invariant ergodic gradient Gibbs measure $\mu$ for $V$ with zero tilt. Assume that $\tilde\mu$ is the associated extended gradient Gibbs measure, as introduced in Section \ref{extendedgradGibbs}, and it is therefore uniquely defined in terms of $\mu$. We will only prove the first inequality as the second one follows similarly. Recalling the definition of $\kappa\to\mu[\kappa]$ from Lemma \ref{l:extension_disordered}, and since $\tilde\mu$ has $\mu$ as a marginal, we write
 \[
 \E_{\mu}\left(\exp((\alpha/2-\delta)\eta^2(b))\right)=\E_{\tilde\mu}\left(\exp((\alpha/2-\delta)\eta^2(b))\right)=\E_{\bar\mu}\E_{\mu[\kappa]}\left(\exp((\alpha/2-\delta)\eta^2(b))\right).
 \]
 We observe now that $\kappa\to\extmu(\d\eta\mid\Fkappa)(\kappa)$ plays the role of the a.s unique family from Lemma \ref{c:unique_disordered_Gibbs}, that $\bar\mu$ plays the role of $\P$ therein, and $\mu$ plays the role of $\mu^{\mbox{av}}$ and is ergodic, so in view of the ergodic theorem for co-cycles (see for example \cite[Theorem 4]{MR1101082} or \cite[Lemma 2.8]{Sepplecture}), it satisfies \eqref{firstlimit12b}. We can now apply the a.s. uniqueness statement from Lemma \ref{c:unique_disordered_Gibbs} to conclude that $\mu$ satisfies \eqref{logmixtureexpinfrand}, and therefore \eqref{e:Gaussianmoments} holds.
\end{proof}

\section{Ergodicity and extremality of \texorpdfstring{$\kappa$}{kappa}-disordered measures}\label{s:characterisation}
\subsection{Time-ergodicity for \texorpdfstring{$\kappa$}{kappa}-disordered measures}\label{ss:dynamics}

The uniqueness result from Theorem \ref{t:CK} (which is taken from \cite{MR3383338}) is proven using a key idea going back to Funaki and Spohn \cite{MR1463032}. Namely one uses the Langevin dynamics associated to gradient Gibbs measures in order to construct couplings between different measures. In order to obtain a Helffer-Sjöstrand representation, we will need an additional property of the Langevin dynamics, namely that they are ergodic in time. 

Consider a collection of independent two-sided Brownian motions $W_t(x)$ indexed by $x\in\zd$ and consider the system of stochastic differential equations
\begin{equation}\label{e:Lang_dyn}
\d\eta_t(b)=-\sum_{\substack{b'\in\bzd\\ x_b=x_{b'}}}V_{\kappa(b')}'(\eta_t(b'))\d t+\sum_{\substack{b'\in\bzd\\ y_b=x_{b'}}}V_{\kappa(b')}'(\eta_t(b'))\d t+\sqrt{2}\d W_t(b),\quad b\in\bzd
\end{equation}
where $W_t(b):=W_t(y_b)-W_t(x_b)$. 

The formal generator of this diffusion is given by
\begin{equation}
\label{e:langevin}
\Lcal_{[\kappa]}F(\eta)=\sum_{x\in\zd}\Bigg(\partial_x^2F(\eta)+\sum_{\substack{b'\in\bzd\\ x_{b'}=x}}V_{\kappa(b')}'(\eta(b'))\partial_xF(\eta) \Bigg) 
\end{equation}
where we recall that
\[\partial_xF(\eta):=\frac{d}{dt}\Big|_{t=0}F(\eta+t\nabla\1_x).\]

The first question of interest is the well-posedness of the system of SDEs \eqref{e:Lang_dyn}.
For the purpose of proving Theorem \ref{t:CK}, it would be enough to have existence of strong solutions for times $[0,\infty)$. But in the proof of Theorem \ref{t:scalinglimit}, we will need to consider the space-time ergodicity of $\eta_\cdot$, and so we need to construct solutions for times $(-\infty,\infty)$, which is slightly more subtle.

Before we state the result, we point the reader to the definitions of stationarity and reversibility, as stated for example in Definition 9.1 from \cite{MR2228384}. 

\begin{lemma}\label{l:existence_langevin}
Let $\kappa\in K^\bzd$ such that $\kappa(b)=\kappa(-b)$ for all $b\in\bzd$. Assume that for all $b\in (\zd)^*$, $V_{\kappa(b)}$ is even and in $C^2(\R)$, and there exists $0<c_V\le C_V<\infty$ such that $c_V\le V''_{\kappa(b)}\le C_V$ for all $b\in (\zd)^*$. Let $\mu[\kappa]$ be a $\kappa$-disordered Gibbs measure such that $\mu[\kappa](\chi_r)=1$.

Then there is a diffusion process $\eta_\cdot\in C(\R,\chi_r)$ with the following properties:
\begin{itemize}
    \item[(i)] $\mu[\kappa]$ is stationary and reversible with respect to $\eta_\cdot$
    \item[(ii)] For any $t_0\in\R$ the process $(\eta_{t})_{t\ge t_0}$ is a solution of \eqref{e:Lang_dyn}. This solution is a strong solution in the sense that there is a Borel-measurable map $\Phi$ (independent of $t_0$) such that $(\eta_{t+t_0})_{t\ge 0}=\Phi\big(\eta_{t_0},(W_{t+t_0})_{t\ge 0}\big)$.
    \item[(iii)] $\eta_\cdot$ is invariant in law under time-shifts in the sense that the law of $(\eta_{t_0+t})_{t\in\R}$ does not depend on $t_0$.
\end{itemize}
\end{lemma}

\begin{proof}
By assumption, the drift $V'_{\kappa(b)}$ is Lipschitz-continuous. So for any fixed $t_0\in\R$ pathwise uniqueness and existence of strong solutions on $[t_0,\infty)$ starting from any deterministic $\eta_0$ follows from standard arguments, see \cite[Lemma 2.2]{MR1463032} or \cite[Section 2.1.3]{MR1872740}. This also provides existence of solution map $\Phi_{t_0}$ (possibly depending on $t_0$) such that $(\eta_{t+t_0})_{t\ge 0}=\Phi_{t_0}\big(\eta_{t_0},(W_{t+t_0})_{t\ge 0}\big)$. Moreover $\mu[\kappa]$ is reversible for the dynamics by the argument in \cite[Proposition 3.1]{MR1463032}, which implies (i).

To extend this to times $(-\infty,\infty)$, we could construct by hand the reversed dynamics on $(-\infty,t_0]$ and glue it together with the dynamics on $[t_0,\infty)$ (both started from the stationary measure). But then it would be slightly tedious to prove time-translation invariance. A faster way is as follows. By the argument in \cite[4.15]{MR512335} there is a Markov process $\tilde\eta_\cdot$ with continuous paths on $l^2(\chi,\mathfrak{a})$ for some weight function $\mathfrak{a}$ with generator \eqref{e:langevin} that is invariant in law under time-shifts, and such that the law of $\tilde\eta_t$ for each fixed $t$ is $\mu[\kappa]$.
So by pathwise uniqueness on any interval $[t_0,\infty)$, we must actually have $\eta_\cdot=\tilde\eta_\cdot$ on any interval $[t_0,\infty)$. But that means that $\tilde\eta_\cdot$ actually is a strong solution on $[t_0,\infty)$ with continuous paths in $\chi_r$. As $\tilde\eta_\cdot$ is invariant in law under time-shifts, pathwise uniqueness then also implies that the $\Phi_{t_0}$ are actually independent of $t_0$. So renaming $\tilde\eta$ to $\eta$, (ii) and (iii) follow as well.
\end{proof}

We denote by $\mathsf{P}^\kappa_{\mu[\kappa]}$ the law of the process $\eta_\cdot$ (and by $\mathsf{E}^\kappa_{\mu[\kappa]}$ the corresponding expectation). In order to establish the Helffer-Sjöstrand representation, we need to establish time-ergodicity of $\eta_\cdot$ under $\mathsf{P}^\kappa_{\mu[\kappa]}$.

\begin{lemma}\label{l:langevin_time_mixing}
Let $\P$ be a shift-invariant ergodic probability measure on $(K^{\bzd},\Kcal^{\otimes\bzd})$ such that $\P$-a.s. $\kappa(b)=\kappa(-b)$ for all $b\in\bzd$. Assume that for all $b\in (\zd)^*$, the map $(\kappa(b),s)\mapsto V_{\kappa(b)}(s)$ is jointly measurable, and for each $\kappa(b)$ the function $V_{\kappa(b)}$ is even and in $C^2(\R)$. Assume also that there exists $0<c_V\le C_V<\infty$ such that $c_V\le V''_{\kappa(b)}\le C_V$ for all $b\in (\zd)^*$ and for $\P$-a.e. $\kappa\in K^\bzd$. 

Let $\kappa\to\mu[\kappa]$ be a shift-covariant disordered gradient Gibbs measure such that the annealed measure $\mu^{\text{av}}$ has zero tilt, is ergodic under the shifts and satisfies \eqref{intcond1a}. Then for any local function $F\in C^2_\text{loc}(\chi)\cap L^2(\chi,\mu[\kappa])$ such that $\Lcal_{[\kappa]} F\in L^2(\chi,\mu[\kappa])$ we have for $\P$-almost all $\kappa\in K^\bzd$ that
\begin{equation}\label{e:langevin_time_mixing}
\lim_{t\to\infty}\big\|\mathsf{E}^\kappa_{\mu[\kappa]}(F(\eta_t)\mid \eta_0=\cdot)-\mu[\kappa] (F)\big\|_{L^2(\mu[\kappa])}=0.
\end{equation}
In particular, if $F$ is also invariant (in the sense that $F=\mathsf{E}^\kappa_{\mu[\kappa]}(F(\eta_t)\mid \eta_0=\cdot)$ for all $t\ge0$), then $F$ must be $\mu[\kappa]$-a.e. constant. Furthermore, if $\mathsf{E}^\kappa_{\mu[\kappa]}(F)=0$ then we also have
\begin{equation}
\label{e:langevin_time_mixingcov}
\lim_{t\to\infty}\mathsf{E}^\kappa_{\mu[\kappa]}(F(\eta_t)F(\eta_0))=0.
\end{equation}
\end{lemma}
\begin{proof}
The proof is an application of the coupling arguments in \cite{MR1463032,MR1872740}. These arguments have already been adapted to the $\kappa$-disordered setting in \cite{MR3383338}, and the proof of the lemma will use a combination of these ideas.

Namely, we argue as in \cite[Lemma 3.2]{MR1872740} and first reduce the reasoning to just proving the statement for $F$ uniformly Lipschitz with respect to its finitely many arguments. We now fix such an $F$.

Consider now the coupled process $\{\eta_t, {\breve\eta}_t\}_{t\ge 0}$, where $\eta$ and $\breve\eta$ are two solutions of \eqref{e:Lang_dyn}, driven by the same Brownian motion. We take $\eta_0$ and ${\breve\eta}_0$ to be independent and distributed according to $\mu[\kappa]$. Denote the law of the coupled process by $\mathsf{Q}^\kappa_{\mu[\kappa]}$. We set $u(\cdot, t):=\mathsf{E}^\kappa_{\mu[\kappa]}(F(\eta_t)\mid \eta_0=\cdot)$. Then for all $\kappa\in K^\bzd$, we write
\[
\Var_{\mu[\kappa]}(u(\cdot,t))\le\int\int\left[u(\eta, t)-u(\breve\eta, t)\right]^2\d\mu[\kappa](\eta)\d\mu[\kappa](\breve\eta)
\le\int\left(F(\eta_t)-F({\breve\eta}_t)\right)^2\d\mathsf{Q}^\kappa_{\mu[\kappa]}\\
\]
and hence
\[
\int\Var_{\mu[\kappa]}(u(\cdot,t))\d\P(\kappa)\le C\sum_{b\in (\zd)^*\cap \supp(F)}\int\int\left(\eta_t(b)-{\breve\eta}_t(b)\right)^2\d\mathsf{Q}^\kappa_{\mu[\kappa]}\d\P(\kappa),
\]
where the sum is over a finite, independent of $\kappa$, number of bonds since $F$ is local, and $C$ is the Lipschitz constant of $F$. 

Now \cite[Lemma 4.3]{MR3383338} implies that for each fixed $b\in \bzd$ we have
\[\lim_{T\rightarrow\infty}\frac{1}{T}\int_0^T\int\int\left(\eta_t(b)-{\breve\eta}_t(b)\right)^2\d\mathsf{Q}^\kappa_{\mu[\kappa]}\d t\d\P(\kappa)=0\]
This means that also
\begin{align*}
\limsup_{T\rightarrow\infty}\int\frac{1}{T}\int_0^T \Var_{\mu[\kappa]}(u(\cdot,t))\d t\,\d\P(\kappa)
\le\lim_{T\rightarrow\infty}\frac{C\sum_{b\in (\zd)^*\cap \supp(F)}}{T}\int\int_0^T\int\left(\eta_t(b)-{\breve\eta}_t(b)\right)^2\d\mathsf{Q}^\kappa_{\mu[\kappa]}\d t\d\P(\kappa)=0.
\end{align*}

Next, by the same argument as for \cite[(3.10)]{MR1872740}, we get $\frac{\d}{\d t} \Var_{\mu[\kappa]}(u(\cdot,t))\le 0.$ Thus, $\Var_{\mu[\kappa]}(u(\cdot,t))$ is decreasing as a function of $t$ and we have
\begin{align*}
\int\lim_{T\to\infty}\Var_{\mu[\kappa]}(u(\cdot,T))\d\P(\kappa)&\le \liminf_{T\to\infty}\int\Var_{\mu[\kappa]}(u(\cdot,T))\d\P(\kappa)\\
&\le \limsup_{T\rightarrow\infty}\int\frac{1}{T}\int_0^T \Var_{\mu[\kappa]}(u(\cdot,t))\d t\,\d\P(\kappa)=0,
\end{align*}
where the first inequality follows from Fatou's lemma. The above allows us to conclude that for a.e. $\kappa$ there holds
\begin{equation}
\label{varlim}
\lim_{T\to\infty}\Var_{\mu[\kappa]}(u(\cdot,T))=0.
\end{equation}
This in turn immediately implies that for $F$ invariant, $F$ must be $\mu[\kappa]$-a.e. constant. 

We turn next to the proof of \eqref{e:langevin_time_mixing}. By Lemma \ref{l:existence_langevin}, we have that $\E_{\mu[\kappa]} \left(F(\eta_0)\right)= \E_{\mu[\kappa]} \left(F(\eta_t)\right),$
and therefore \eqref{e:langevin_time_mixing} reduces to showing that
\[\lim_{t\to\infty}\left\|\mathsf{E}^\kappa_{\mu[\kappa]}(F(\eta_t)\mid \eta_0=\cdot)-\E_{\mu[\kappa]} \left(F(\eta_t)\right)\right\|_{L^2(\mu[\kappa])}=0.\]
By conditioning and then using \eqref{varlim}, we get \eqref{e:langevin_time_mixing}. Equation \eqref{e:langevin_time_mixingcov} follows by using the Cauchy-Schwarz inequality, and then applying \eqref{e:langevin_time_mixing}.
\end{proof}

In fact, we even have the following strengthening of Lemma \ref{l:langevin_time_mixing}, which will be used in the following subsection
\begin{lemma}\label{l:langevin_time_mixing_unif}
In the setting of Lemma \ref{l:langevin_time_mixing}, for $\P$-a.e. $\kappa$ any invariant function is $\mu[\kappa]$-a.e. constant.
\end{lemma}
\begin{proof}
It suffices to prove that for $\P$-a.e. $\kappa$ we have \eqref{e:langevin_time_mixing} simultaneously for all $F\in L^2(\chi,\mu[\kappa])$. To see this, we proceed by density. Consider the set $C^2_{b,\text{loc}}(\chi)$, the space of local twice differentiable functions on $\chi$ with bounded derivatives up to order 2. There is a countable subset $S$ of $C^2_{b,\text{loc}}(\chi)$, independent of $\kappa$, which is dense in $L^2(\chi,\mu[\kappa])$ (cf. \cite[p. 396]{MR1432591}), and by Lemma \ref{l:langevin_time_mixing} for each $F\in C^2_{b,\text{loc}}(\chi)$ we have \eqref{e:langevin_time_mixing} for $\P$-a.e. $\kappa$.

By countable additivity, we then $\P$-a.s. have \eqref{e:langevin_time_mixing} simultaneously for all $F\in S$. To pass from $C^2_{b,\text{loc}}(\chi)$ to $L^2(\chi,\mu[\kappa])$ we use density and the fact that the semigroup $F\mapsto T_tF:=\mathsf{E}^\kappa_{\mu[\kappa]}(F(\eta_t)\mid \eta_0=\cdot)$ is bounded on $L^2(\chi,\mu[\kappa])$.
\end{proof}

\subsection{Extremality of \texorpdfstring{$\kappa$}{kappa}-disordered measures}

A consequence of Lemma \ref{l:langevin_time_mixing} is an extremality statement for the $\mu[\kappa]$ that we would like to record separately. 

\begin{theorem}\label{th:extremality}
Let $\P$ be a shift-invariant ergodic probability measure on $(K^{\bzd},\Kcal^{\otimes\bzd})$ such that $\P$-a.s. $\kappa(b)=\kappa(-b)$ for all $b\in\bzd$. Assume that for all $b\in (\zd)^*$, the map $(\kappa(b),s)\mapsto V_{\kappa(b)}(s)$ is jointly measurable, and for each $\kappa(b)$ the function $V_{\kappa(b)}$ is even and in $C^2(\R)$. Assume also that there exists $0<c_V\le C_V<\infty$ such that $c_V\le V''_{\kappa(b)}\le C_V$ for all $b\in (\zd)^*$ and for $\P$-a.e. $\kappa\in K^\bzd$.

Let $\kappa\to\mu[\kappa]$ be the a.s.-unique shift-covariant disordered gradient Gibbs measure such that $\mu^{\text{av}}$ has zero tilt, is ergodic under the shifts and satisfies \eqref{intcond1a}. 
Then for $\P$-a.e. $\kappa$, the measure $\mu[\kappa]$ is extremal among gradient Gibbs measures. That is, for $\P$-a.e. $\kappa$ there is no non-trivial representation
\begin{equation}\label{e:non-extremal}
\mu[\kappa]=\lambda\mu'[\kappa]+(1-\lambda)\mu''[\kappa]
\end{equation}
where $0<\lambda<1$ and $\mu'[\kappa]\neq\mu''[\kappa]$ are two different gradient Gibbs measures (not necessarily tempered, not necessarily shift-covariant) for the family of Hamiltonians $(H_{\Lambda,\xi}[\kappa])_{\Lambda\subset\zd,\xi\in\chi}$. 
\end{theorem}

Before proving this result, let us state a corollary, namely the special case where all $V_\kappa$ are equal (or in other words, $V$ is already in $C^2(\R)$ with $c_V\le V''\le C$ as in the Ginzburg-Landau model). In that case the extremality of shift-invariant ergodic gradient Gibbs measure follows from \cite[Lemma 8.7.1]{MR2251117}, but our proof provides a completely different route to this result.

\begin{corollary}
Let $d\ge 1$.  Assume that there exists $0<c_V\le C_V<\infty$ such that $c_V\le V''\le C_V$. Let $\mu$ be the unique tempered shift-invariant ergodic gradient Gibbs measure with zero tilt, as proved in \cite[Theorem 2.1]{MR1463032}. 
Then the measure $\mu$ is extremal among gradient Gibbs measures. 
\end{corollary}
We state the result in the zero-tilt case as that is what follows from Theorem \ref{th:extremality}, but in fact the same method applies for arbitrary tilt, using \cite[Lemma 3.2]{MR1872740} as the main input.

Let us now turn to the proof of Theorem \ref{th:extremality}.

\begin{proof}[Proof of Theorem \ref{th:extremality}]

We will use some abstract results from \cite[Section 5]{MR1432591}. There the results are stated in detail for the special case that (in our notation) all $V_\kappa$ are quadratic functions, but as pointed out in \cite[Remark 5.18 (ii)]{MR1432591}, everything in that section applies in our setting as long as $0<c\le V_\kappa''\le C<\infty$. By Lemma \ref{l:langevin_time_mixing_unif}, $T_t$ is ergodic, and assertion \cite[Theorem 5.15 (iv)]{MR1432591} holds. 
Our plan is to apply \cite[Theorem 5.15 (iv) $\implies$ (i)]{MR1432591}, but since \cite{MR1432591} works with exponentially integrable Gibbs measures, some care is needed.

We first note that by \eqref{logmixtureexpinfrand} in Lemma \ref{c:unique_disordered_Gibbs} for $\P$-a.e. $\kappa$ we have
\begin{equation}\label{e:expintegrable}\sup_{b\in\bzd}\E_{\mu[\kappa]}(\e^{|\eta_b|})<\infty
\end{equation}
This means that $\mu[\kappa]$ is exponentially integrable for a.e. $\kappa$, and in the set $\Mcal_{exp}$ as defined in \cite[p. 413]{MR1432591}. As $\mu[\kappa]$ is also a (gradient) Gibbs measure in the sense that it satisfies the DLR conditions, by \cite[Proposition 5.9]{MR1432591} it is in the set $\Gcal_{exp}(\Phi)=\Gcal^{b^\Phi}\cap \Mcal_{exp}$ (as again defined in \cite[p. 413]{MR1432591}). Now \cite[Theorem 5.15 (iv) $\implies$ (i)]{MR1432591} indeed implies that $\mu[\kappa]$ is extremal in $\Gcal_{exp}$. 

Suppose now for the sake of contradiction that there is a nontrivial decomposition as in \eqref{e:non-extremal}. Because
\[\E_{\mu[\kappa]}\e^{|\eta_b|}=\lambda\E_{\mu'[\kappa]}(\e^{|\eta_b|})+(1-\lambda)\E_{\mu''[\kappa]}(\e^{|\eta_b|})\]
\eqref{e:expintegrable} implies that also
\[\sup_{b\in\bzd}\E_{\mu'[\kappa]}(\e^{|\eta_b|})<\infty,\qquad\sup_{b\in\bzd}\E_{\mu''[\kappa]}(\e^{|\eta_b|})<\infty\]
So by the same reasoning as for $\mu[\kappa]$, both $\mu'[\kappa]$ and $\mu''[\kappa]$ are in $\Gcal_{exp}$. This means that $\mu[\kappa]$ is a nontrivial convex combination of two elements of $\Gcal_{exp}$, and so cannot be extremal in $\Gcal_{exp}$. This is a contradiction, and so a decomposition as in \eqref{e:non-extremal} cannot exist.
\end{proof}

\section{Helffer-Sjöstrand representation}\label{s:HS}

In order to prove Theorem \ref{t:scalinglimit}, we will derive a Helffer-Sjöstrand representation for our model. In fact, for each fixed $\kappa$, the standard Helffer-Sjöstrand representation can be used, and ultimately we will take the average of this representation over the $\kappa$.

Let us explain this set-up in detail. Let $\kappa\in K^\bzd$, and let $\mu[\kappa]$ be a $\kappa$-disordered Gibbs measure such that $\mu[\kappa](\chi_r)=1$.

Throughout this section, we will assume that the $V_\kappa''$ are not only bounded away from 0, but also bounded away from $\infty$. This assumption is satisfied in the case we are interested in, strongly log-concave log-mixtures of quadratic growth. If now for some fixed $\kappa$ the $V_\kappa''$ are both bounded away from 0 and $\infty$, then we can obtain a Helffer-Sjöstrand representation for that fixed $\kappa$ by a straightforward adaptation of the argument in \cite{MR1872740}. Eventually we will average these representations over $\kappa$ to obtain a Helffer-Sjöstrand representation for a Gibbs measure associated to a strongly log-concave log-mixture of quadratic growth.

Fix now for the moment some $\kappa\in K^{\bzd}$ with $\kappa(-b)=\kappa(b)$. We consider the Langevin dynamics associated to $\mu[\kappa]$, as introduced in \eqref{e:Lang_dyn}. Given a realisation of $\eta_\cdot$, we define $a^{[\kappa],\eta}$ on $\R\times\bzd$ by $a^{[\kappa],\eta}_t(b)=V_{\kappa(b)}''(\eta_t(b))$. 
Next, given a realisation of $\eta_\cdot$, we can construct a (continuous-time) random walk starting at some $x\in\zd$ at some $s\in\R$, and with jump rates $a^{[\kappa],\eta}_t(b)$ across the bond $b$, which we denote by $X_t^{[\kappa],\eta}$. Our assumption $\kappa(-b)=\kappa(b)$ ensures that the jump rate does not depend on the direction of $b$.

\begin{lemma}\label{l:randomwalk}
Let $\kappa\in K^\bzd$  such that $\kappa(b)=\kappa(-b)$ for all $b\in\bzd$. Assume that for all $b\in (\zd)^*$, $V_{\kappa(b)}$ is even and in $C^2(\R)$, and there exists $0<c_V\le C_V<\infty$ such that $c_V\le V''_{\kappa(b)}\le C_V$ for all $b\in (\zd)^*$. Let $\mu[\kappa]$ be a $\kappa$-disordered Gibbs measure such that $\mu[\kappa](\chi_r)=1$, and consider the Langevin dynamics $\eta_\cdot$ associated with $\mathsf{P}^\kappa_{\mu[\kappa]}$ (as defined in Subsection \ref{ss:dynamics}).

Then the random walk $X_t^{[\kappa],\eta}$ is well-defined and almost surely non-explosive.
\end{lemma}
\begin{proof}
For quadratically bounded $V$ we have a uniform upper bound on $V''_{\kappa}$, and so the $a^{[\kappa],\eta}_t$ are uniformly bounded. This means that the construction of this random walk on $[s,\infty)$ is standard.
\end{proof}

We denote the law of the random walk $X_t^{[\kappa],\eta}$ when started from $x$ by $\mathbf{P}_x^{[\kappa],\eta}$, and the corresponding expectation by $\mathbf{E}_x^{[\kappa],\eta}$. 
Finally we introduce the heat kernel of this random walk, given by
\[p^{[\kappa],\eta}(s,x,t,y):=\mathbf{P}^{[\kappa],\eta}(X_t^{[\kappa],\eta}=y\mid X_s^{[\kappa],\eta}=x).\]
and recall that $C^2_{b,\text{loc}}(\chi)$ denotes the space of local twice differentiable functions on $\chi$ with bounded derivatives up to order 2.

\begin{lemma}\label{l:HS_quenched}
Under the assumptions of Lemma \ref{l:langevin_time_mixing} for $\P$-a.e. $\kappa$ the following holds:
For any $F,G\in C^2_{b,\text{loc}}(\chi)$ we have
\begin{align*}
    \Cov_{\mathsf{P}^\kappa_{\mu[\kappa]}}\left(F(\eta_0),G(\eta_t)\right)&=\int_0^\infty \sum_{x\in\zd}\E_{\mu[\kappa]} \left(\partial_xF(\eta_0)\mathbf{E}_x^{[\kappa],\eta}\left(\partial_{X_{t+s}^{[\kappa],\eta}}G(\eta_{t+s})\right)\right)\d s\\
    &=\int_0^\infty\sum_{x,y\in\zd}\E_{\mu[\kappa]}\left(\partial_xF(\eta_0)\partial_yG(\eta_{t+s})p^{[\kappa],\eta}(0,x,t+s,y)\right)\d s
\end{align*}
\end{lemma}
\begin{proof}
This result in the shift-invariant setting is a direct consequence of \cite[Proposition 3.1]{MR1872740}. But shift-invariance only enters the proof there via the time-ergodicity of the dynamics as stated in \cite[Lemma 3.2]{MR1872740}.
But in our case we have that time-ergodicity for $\P$-a.e. $\kappa$ by \eqref{e:langevin_time_mixingcov}. Using this as input, the remainder of the proof can be done exactly as for \cite[Proposition 3.1]{MR1872740}.
\end{proof}

Our main interest is in an annealed version of this Helffer-Sjöstrand representation, where we average over the $\kappa$. This will lead to a representation as an annealed random walk in a random environment. Let us introduce the setting.

In Lemma \ref{l:randomwalk} we had constructed the random walk in the environment $a^{[\kappa],\eta}_\cdot$ for a fixed $\kappa$. We will now average this construction over $\kappa$ sampled according to $\bar\mu$ (where as in earlier sections $\mu$ is a shift-invariant ergodic tempered gradient Gibbs measure with zero tilt, $\extmu$ its extension, and $\bar\mu$ its $\kappa$-marginal).

So we have a family of random walks in random environments that depends on $\kappa$. We can equivalently view this as a single random walk in a random environment whose law is a mixture of the different environments. That is, we sample first $\kappa$ according to $\bar\mu$, and then $\eta_\cdot$ from $P^{\kappa}_{\mu[\kappa]}$. This leads to a random walk in the random environment $a^{[\kappa],\eta}_\cdot$, where $(\kappa,\eta_\cdot)$ are distributed according to $\mathsf{P}_\mu(\d\kappa,\d\eta_\cdot):=\bar\mu(\d\kappa)\mathsf{P}^\kappa_{\mu[\kappa]}(\d\eta_\cdot)$ (with associated expectation $\mathsf{E}_\mu$). Note first that according to Lemma \ref{l:randomwalk}, this random walk is still almost surely non-explosive.

As we will shortly want to use results from stochastic homogenisation, it will be crucial for us that the environment is space-time ergodic.

\begin{lemma}\label{l:jumprates_ergodic}
Let $V$ be a strongly log-concave log-mixture potential of quadratic growth. Let $\mu$ be a shift-invariant tempered ergodic gradient Gibbs measure with zero tilt. Let $\extmu$ be its extension, and $\bar\mu$ its $\kappa$-marginal (as in Lemma \ref{l:extension_tinv_erg}). Then the law of $a^{[\kappa],\eta}_\cdot$, where $(\kappa,\eta_\cdot)$ are distributed according to $\mathsf{P}_\mu$, is invariant under space-time shifts and space-time ergodic. 
\end{lemma}
Here by space-time ergodicity we mean that any bounded measurable function on path space $(\R^+)^{\bzd\times\R}$ that is invariant under shifts of space and time is $\mathsf{P}_\mu$-a.e. constant. Note that ergodicity under space-time shifts is a stronger property than ergodicity under time shifts. So the ergodicity statement here is non-trivial.
\begin{proof}
As $a^{[\kappa],\eta}_\cdot$ is a measurable function of $(\kappa,\eta_\cdot)\in K^\bzd\times (\R^\bzd)^{\R}\subset (K\times \R)^{\bzd\times\R}$, it suffices to prove that its law $\mathsf{P}_\mu$, considered as a probability measure on $(K\times \R)^{\bzd\times\R}$, is space-time shift-invariant and ergodic.

For fixed $\kappa$, the dynamics $\eta_\cdot$ have stationary measure $\mu[\kappa]$ and are invariant under time shifts. Because $\bar\mu$ is invariant under space shifts, and the law of $\mathsf{P}^\kappa_{\mu[\kappa]}$ is covariant under space-shifts (as follows from shift-invariance of the law of $W_t$), the law of $a^{[\kappa],\eta}_\cdot$ is also invariant under space shifts. So the law is indeed invariant under space-time shifts.

To prove the ergodicity, consider the process $\eta_\cdot$. This process is Markovian, and according to Lemma \ref{l:langevin_time_mixing_unif}, for a.e. $\kappa$ all invariant functions for this Markov process are constant. Now \cite[Theorem 26.11]{MR4226142} implies that if we consider $\eta_\cdot$ as process on path space $(\R^\bzd)^{\R}=\R^{\bzd\times\R}$ then the invariant sigma field is trivial, and so for a.e. $\kappa$ any measurable function that is invariant under time shifts is $\mathsf{P}^\kappa_{\mu[\kappa]}$-a.e. constant. 

Now suppose that $F\colon(K\times \R)^{\bzd\times\R}\to\R$ is bounded, measurable and invariant under space-time shifts. We need to prove that $F$ is constant $\mathsf{P}_\mu$-almost everywhere. To that end, observe first that the invariance of $F$ under time shifts implies that for a.e. $\kappa$ there is a constant $c_\kappa$ such that $F$ is $\mathsf{P}^\kappa_{\mu[\kappa]}$-a.s. equal to $c_\kappa$. This implies that $\mathsf{E}^\kappa_{\mu[\kappa]}(F(\kappa,\cdot))=c_\kappa$ for a.e. $\kappa$. The function $\mathsf{E}^\kappa_{\mu[\kappa]}(F(\kappa,\cdot))$ on the left-hand side is invariant under space shifts (because $F$ is shift-invariant and $\mathsf{P}^\kappa_{\mu[\kappa]}$ is shift-covariant), and so the ergodicity of $\bar\mu$ implies that $\bar\mu$-a.e. $c_\kappa$ is equal to some fixed constant $c$. In other words, $F$ is $\mathsf{P}_\mu$-a.e. constant, as required.
\end{proof}

Thus we now have a random walk in an ergodic random environment.
We introduce the annealed heat kernel of this random walk, given by
\[p(s,x,t,y):=\mathsf{E}_\mu p^{[\kappa],\eta}(s,x,t,y)=\E_{\bar\mu}\mathsf{E}^\kappa_{\mu[\kappa]}p^{[\kappa],\eta}(s,x,t,y).\]
According to Lemma \ref{l:jumprates_ergodic}, $p$ is invariant under shifts of space and time, i.e.  $p(s,x,t,y)=p(0,0,t-s,y-x)$.

\begin{theorem}\label{t:HS_annealed}
Let $V$ be a strongly log-concave log-mixture potential of quadratic growth. Let $\mu$ be a shift-invariant tempered ergodic gradient Gibbs measure with zero tilt. Let $\extmu$ be its extension, and $\bar\mu$ its $\kappa$-marginal (as in Lemma \ref{l:extension_tinv_erg}).
Then for any $F,G\in C^2_{b,\text{loc}}(\chi)$ such that at least one of them is an odd function, and any $t\ge0$ we have 
\begin{equation}\label{e:HS_annealed}
\begin{split}
    \Cov_{\mathsf{P}_\mu}\left(F(\eta_0),G(\eta_t)\right)&=\int_0^\infty \sum_{x\in\zd}\mathsf{E}_\mu \left(\partial_xF(\eta_0)\mathbf{E}_x^{[\kappa],\eta}\left(\partial_{X_{t+s}^{[\kappa],\eta}}G(\eta_{t+s})\right)\right)\d s\\
    &=\int_0^\infty\sum_{x,y\in\zd}\mathsf{E}_\mu\left(\partial_xF(\eta_0)\partial_yG(\eta_{t+s})p^{[\kappa],\eta}(0,x,t+s,y)\right)\d s
\end{split}
\end{equation}

In particular, if $F(\eta)=\sum_{b\in\bzd}f(b)\eta(b)$ and $G(\eta)=\sum_{b\in\bzd}g(b)\eta(b)$ for some $f,g\colon\bzd\to\R$ with compact support, then
\begin{equation}\label{e:HS_annealed_cov}
\Cov_{\mathsf{P}_\mu}\left(F(\eta_0),G(\eta_0)\right)=\int_0^\infty\sum_{b,b'\in\bzd}f(b)g(b')\nabla\nabla' p(0,b,s,b')\d s
\end{equation}
where
\[\nabla\nabla' p(0,b,s,b'):=p(0,y_b,s,y_{b'})-p(0,y_b,s,x_{b'})-p(0,x_b,s,y_{b'})+p(0,x_b,s,x_{b'})\]
\end{theorem}
Note that in comparison to Lemma \ref{l:HS_quenched} we have the additional assumption that at least one of $F,G$ is odd. We need it to ensure that its expectation over $\mu[\kappa]$ vanishes, which will imply that $ \Cov_{\mathsf{P}_\mu}\left(F(\eta_0),G(\eta_t)\right)=\E_{\bar\mu}\left(\Cov_{\mathsf{P}^\kappa_{\mu[\kappa]}}\left(F(\eta_0),G(\eta_t)\right)\right)$.

\begin{proof}
Assume w.l.o.g. that $F$ is odd. By Lemma \ref{l:extension_disordered}, we have $\E_{\mu[\kappa]}F(\eta)=0$ for $\bar\mu$-a.e. $\kappa$, and thus also $\mathsf{E}_\mu F(\eta)=\E_\mu F(\eta)=\E_{\bar\mu}\E_{\mu[\kappa]}F(\eta)=0$. This implies that
\[\Cov_{\mathsf{P}_\mu}\left(F(\eta_0),G(\eta_t)\right)=\mathsf{E}_\mu\left(F(\eta_0)G(\eta_t)\right)=\E_{\bar\mu}\mathsf{E}^\kappa_{\mu[\kappa]}\left(F(\eta_0)G(\eta_t)\right)=\E_{\bar\mu}\Cov_{\mathsf{P}^\kappa_{\mu[\kappa]}}\left(F(\eta_0),G(\eta_t)\right)\]

So \eqref{e:HS_annealed} follows from Lemma \ref{l:HS_quenched}, provided that we can interchange $\E_{\bar\mu}$ with the integral. We cannot naively apply Fubini's theorem here, as the heat kernel $p^{[\kappa],\eta}$ does not decay fast enough (at least when $d=2$). So we need to rewrite the integrand first in such a way that only the derivatives of the heat kernel remain (which have better decay).

For that purpose, let $\tilde F,\tilde G$ be extensions of $F,G$ to $C^1(\R^\bzd)$ that are still local functions. Under our assumptions on $F,G$ the functions $\partial_x F$ and $\partial_yG$ vanish except for finitely many $x,y$, and the same holds true for $\partial_x \tilde F$ and $\partial_y \tilde G$

So we can calculate
\begin{align*}
&\sum_{x,y\in\zd}(\partial_xF(\eta_0)\partial_yG(\eta_{t+s})p^{[\kappa],\eta}(0,x,t+s,y)\\
&=\sum_{x,y\in\zd}\bigg(-\sum_{\substack{b\in\bzd\\ x_{b}=x}}\frac{\partial}{\partial \eta(b)}\tilde F(\eta_0)+\sum_{\substack{b\in\bzd\\ y_{b}=x}}\frac{\partial}{\partial \eta(b)}\tilde F(\eta_0)\bigg)\bigg(-\sum_{\substack{b'\in\bzd\\ x_{b'}=y}}\frac{\partial}{\partial \eta(b')}\tilde G(\eta_{t+s})+\sum_{\substack{b'\in\bzd\\ y_{b'}=y}}\frac{\partial}{\partial \eta(b')}\tilde G(\eta_{t+s})\bigg)\\
&\qquad\qquad\times p^{[\kappa],\eta}(0,x,t+s,y)\\
&=\sum_{b,b'\in\bzd}\frac{\partial}{\partial \eta(b)}\tilde F(\eta_0)\frac{\partial}{\partial \eta(b')}\tilde G(\eta_{t+s})\\
&\qquad\qquad\times \left(p^{[\kappa],\eta}(0,y_b,t+s,y_{b'})-p^{[\kappa],\eta}(0,y_b,t+s,x_{b'})-p^{[\kappa],\eta}(0,x_b,t+s,y_{b'})+p^{[\kappa],\eta}(0,x_b,t+s,x_{b'})\right)\\
&=\sum_{b,b'\in\bzd}\frac{\partial}{\partial \eta(b)}\tilde F(\eta_0)\frac{\partial}{\partial \eta(b')}\tilde G(\eta_{t+s})\nabla\nabla'p^{[\kappa],\eta}(0,b,t+s,b')
\end{align*}
After this rewriting, we can note that the annealed heat kernel bounds \eqref{e:heatkernelsecder} in Theorem \ref{t:heatkernel} imply that
\[\int_0^\infty\mathsf{E}_\mu\left|\sum_{b,b'\in\bzd}\frac{\partial}{\partial \eta(b)}\tilde F(\eta_0)\frac{\partial}{\partial \eta(b')}\tilde G(\eta_{t+s})\nabla\nabla'p^{[\kappa],\eta}(0,b,t+s,b')\right|\d s<\infty\]
and now \eqref{e:HS_annealed} follows from Fubini's theorem.

For \eqref{e:HS_annealed_cov}, we cannot directly apply \eqref{e:HS_annealed}, as $F,G$ are not in $C^2_{b,\text{loc}}(\chi)$. But we can use a straightforward truncation argument. Given $M>0$, let $\theta_M(x)\in C^2(\R)$ be an odd function such that $\theta_M(x)=x$ for $|x|\le M$, $|\theta_M(x)|\le C|x|$, $|\theta_M'(x)|\le C$, $|\theta_M''(x)|\le C$ for all $x$ (for instance, let $\theta\in C^2(\R)$ be an odd function such that $\theta(x)=x$ for $|x|\le1$, $\theta(x)=0$ for $|x|\ge2$ and take $\theta_M(x)=M\theta\left(\frac{x}{M}\right)$). Consider $F_M(\eta)=\theta_M(F(\eta))$ and $G_M(\eta)=\theta_M(G(\eta))$. Then $F_M,G_M\in C^2_{b,\text{loc}}(\chi)$, and from \eqref{e:HS_annealed} we conclude
\begin{align*}
\Cov_{\mathsf{P}_\mu}\left(F_M(\eta_0),G_M(\eta_0)\right)&=\int_0^\infty\sum_{x,y\in\zd}\mathsf{E}_\mu\left(\partial_xF_M(\eta_0)\partial_yG_M(\eta_{s})p^{[\kappa],\eta}(0,x,s,y)\right)\d s\\
&=\int_0^\infty\sum_{x,y\in\zd}\mathsf{E}_\mu\left(\theta_M'(F(\eta_0))\theta_M'(G(\eta_s))\partial_xF(\eta_0)\partial_yG(\eta_{s})p^{[\kappa],\eta}(0,x,s,y)\right)\d s\\
&=\int_0^\infty \sum_{b,b'\in\bzd}f(b)g(b')\mathsf{E}_\mu\left(\theta_M'(F(\eta_0))\theta_M'(G(\eta_s))\nabla\nabla' p^{[\kappa],\eta}(0,b,s,b')\right)\d s
\end{align*}
We can take the limit $M\to\infty$ by dominated convergence and obtain \eqref{e:HS_annealed_cov}. On the left-hand side the necessary bound follows from temperedness of $\mu$, on the right-hand side again from the heat kernel bound \eqref{e:heatkernelsecder} and the fact that $\theta_M'$ is uniformly bounded. 
\end{proof}

The proof of the theorem required annealed bounds for the heat kernel.
Because our jump rates $a^{[\kappa],\eta}_\cdot$ are bounded above and below, we can use for this purpose the classical Delmotte-Deuschel estimates \cite{MR2198017}.

\begin{theorem}\label{t:heatkernel}
In the setting of Theorem \ref{t:HS_annealed}, we have the uniform time-decay bounds
\begin{align}
\left(\mathsf{E}_\mu|\nabla p^{[\kappa],\eta}(0,b,t,y)|^2\right)^{1/2}&\le \frac{C(d,V)}{(1\vee t)^{(d+1)/2}}\label{e:heatkernelfirstder}\\
\mathsf{E}_\mu|\nabla\nabla' p^{[\kappa],\eta}(0,b,t,b')|&\le \frac{C(d,V)}{(1\vee t)^{(d+2)/2}}\label{e:heatkernelsecder}
\end{align}
for all $b,b'\in\bzd$ and all $t>0$.

We also have the off-diagonal bound
\begin{equation}\label{e:heatkernelsecderoffdiag}
|\nabla\nabla' p(0,b,t,b')|\le \frac{C(d,V)}{1\vee t}p^*(c(d,V)t,x_b-x_{b'})
\end{equation}
where $p^*(t,x)$ is the heat kernel of continuous time simple random walk on $\zd$. In all three estimates the constants $C(d,V),c(d,V)$ depend on $d$, and on $V$ via the ellipticity parameters $\alpha,C_V$.
\end{theorem}
\begin{proof}  
The jump rates $a^{[\kappa],\eta}_\cdot$ are invariant under space-time shifts and space-time ergodic, and bounded away from 0 and infinity. 
So 
\eqref{e:heatkernelsecderoffdiag} follows from \cite[(1.5b)]{MR2198017}, while for \eqref{e:heatkernelfirstder} and \eqref{e:heatkernelsecder} one combines \cite[(1.4) and (1.5a)]{MR2198017} with the trivial bound $p^*(t,x)\le\frac{C}{(1\vee t)^{d/2}}$.

\end{proof}

We have not only these annealed heat kernel bounds, but also an ``invariance principle in probability'' (using the language of \cite{biskup}), a small strengthening of an (annealed) central limit theorem. In our setting (and in fact, in more general settings), even a quenched central limit theorem (a much stronger result) is known \cite{ACDSquenched}, and the most convenient way to obtain the invariance principle in probability we need is to use that quenched result.

\begin{theorem}\label{t:clt}
In the setting of Theorem \ref{t:HS_annealed} there is a (deterministic) symmetric positive definite matrix $q\in\R^{d\times d}$ such that the annealed probability measure $\mathsf{E}_\mu\mathbf{P}_0^{[\kappa],\eta}\left(\eps X_{t\eps^{-2}}^{[\kappa],\eta}\in\cdot\right)$ converges weakly as $\eps\to0$ to the centred Gaussian $\Ncal(0,2tq)$.

Moreover, for any $F\in C_b(\R^d)$, $\mathbf{E}_0^{[\kappa],\eta}\left(F(\eps X_{t\eps^{-2}}^{[\kappa],\eta})\right)$ converges in $\mathsf{P}_\mu$-probability to $\E (F(Y_t))$, where $Y_t\sim \Ncal(0,2tq)$.
\end{theorem}
The factor 2 might seem slightly artificial here, but it arises from the fact that continuous time simple random walk with jump rates 1 along each edge of $\zd$ scales to $\sqrt{2}B_t$ under parabolic rescaling of space and time. So our normalisation ensures that in the standard GFF case (where $V(s)=\frac{s^2}{2}$) we obtain $q=\mathrm{Id}$, as expected.

\begin{proof}
According to \cite[Theorem 1.7]{ACDSquenched}, there is a deterministic symmetric positive definite matrix $\Sigma\in\R^{d\times d}$ such that for $\mathsf{P}_\mu$-a.e. environment the law of the random walk converges weakly in the Skorokhod space to that of a Brownian motion $Y_t$ with covariance matrix $\Sigma$. In particular, for any $F\in C_b(\R^d)$, $\mathbf{E}_0^{[\kappa],\eta}\left(F(\eps X_{t\eps^{-2}}^{[\kappa],\eta})\right)$ converges $\mathsf{P}_\mu$-almost surely to $\E (F(Y_t))$. This implies directly the claimed convergence in probability and in distribution, when setting $q=\frac12\Sigma$.
\end{proof}

\section{Applications of the Helffer-Sjöstrand representation}\label{s:applications}

In this section, we will use the properties of the extended gradient Gibbs measure from Section \ref{extendedgradGibbs}, and the Helffer-Sjöstrand representation from Section \ref{s:HS} to show the scaling limit result from Theorem \ref{t:scalinglimit} above and the decay of covariances result from Theorem \ref{t:decaycovariances} below.

\subsection{Scaling limit}
\label{SectScaling}

Using Theorem \ref{t:HS_annealed}, Theorem \ref{t:heatkernel} and Theorem \ref{t:clt}, we can now turn to the proof of our main result on the scaling limit of the interfaces. The main step will be to show that the ($\kappa$-dependent) quadratic form arising from the Helffer-Sjöstrand representation converges in $L^p(\mathsf{P}_\mu)$ for some $p>1$ (we will actually use $p=2$ right away). It would be easier to prove convergence in $L^1(\mathsf{P}_\mu)$ (as is done in \cite[Proposition 4.4]{MR2778801}), but in our setting this would be too weak to deduce Gaussianity of the scaling limit.

As we want to obtain convergence in $L^2(\mathsf{P}_\mu)$, we need an $L^2$-bound for the heat kernel. This means that we cannot use the $L^1$-bound \eqref{e:heatkernelsecder} for the second derivative, but need to use the $L^2$-bound \eqref{e:heatkernelfirstder} for the first derivative. This has the side effect that we lose a factor $\frac{1}{t^{1/2}}$ in the proof of the lemma, and for that reason it only applies for $d\ge2$. Fortunately for us, the scaling limit in $d=1$ is true by Donsker's principle (cf. Section \ref{sec:intro_interface}), so it is no problem that our approach does not cover that case.
\begin{lemma}\label{l:scalinglimitvariances}
Let $d\ge1$. Suppose that $V$ is a strongly log-concave log-mixture of quadratic growth, and let $\mu$ be a shift-invariant ergodic tempered gradient Gibbs measure for $V$ with zero tilt. Take $f=(f_1,\ldots,f_d)\in C_c^\infty(\R^d,\R^d)$. Recall the notation of $F_\eps(\eta)$ and $\mathfrak{Q}_f$ from Theorem \ref{t:scalinglimit}, and the definition and properties of the extended measure $\tilde\mu$ from Section \ref{extendedgradGibbs}. Then we have
\begin{equation}\label{e:scalinglimitvariancefixedt}
\lim_{\eps\to0}\mathsf{E}_\mu\bigg(\eps^{d-2}\sum_{x,y\in\zd}\sum_{i=1}^d\sum_{j=1}^df_i(\eps x)f_j(\eps y) \nabla\nabla' p^{[\kappa],\eta}(0,(x,x+e_i),s\eps^{-2},(y,y+e_j))-(\nabla\cdot f,\e^{s Q}\nabla \cdot f)_{L^2(\R^d)}\bigg)^2=0
\end{equation}
for each $s>0$. Moreover, for $d\ge2$ we have that
\begin{equation}\label{e:scalinglimitvarianceintegrated}
\lim_{\eps\to0}\mathsf{E}_\mu\bigg(\int_0^\infty\eps^{d-2}\sum_{x,y\in\zd}\sum_{i=1}^d\sum_{j=1}^df_i(\eps x)f_j(\eps y)\nabla\nabla' p^{[\kappa],\eta}(0,(x,x+e_i),s\eps^{-2},(y,y+e_j))\d s-\mathfrak{Q}_f\bigg)^2=0
\end{equation}
\end{lemma}
\begin{proof}
We follow the proof of \cite[Proposition 4.4 and Lemma 4.5]{MR2778801}, with some changes because we are using a slightly different observable, and because we aim for $L^2$-convergence in \eqref{e:scalinglimitvarianceintegrated}. Namely, the annealed invariance principle \cite[Lemma 3.5]{MR2778801} is not strong enough to deduce our desired conclusion, and instead one needs to use an invariance principle in probability, just as we will do here. 

\textbf{Step 1: Proof of \eqref{e:scalinglimitvarianceintegrated} assuming \eqref{e:scalinglimitvariancefixedt}}\\
We begin by observing that \eqref{e:scalinglimitvarianceintegrated} is a consequence of \eqref{e:scalinglimitvariancefixedt} and the heat kernel bounds of Theorem \ref{t:heatkernel}. Indeed, we note that
\[\mathfrak{Q}_f=(\nabla\cdot f,(-Q)^{-1}\nabla\cdot f)_{L^2(\R^d)}=\int_0^\infty(\nabla\cdot f,\e^{s Q}\nabla \cdot f)_{L^2(\R^d)}\d s.\]
Plugging the above into \eqref{e:scalinglimitvarianceintegrated}, and using the abbreviation 
\[\nabla^2p:=\nabla\nabla' p^{[\kappa],\eta}(0,(x,x+e_i),s\eps^{-2},(y,y+e_j)),\]
where we omit for simplicity of notation the dependence on $x,y, \kappa, \eta$, we write
\begin{equation}\label{e:scalinglimitvariance1}
\begin{split}
&\mathsf{E}_\mu\bigg(\int_0^\infty\eps^{d-2}\sum_{x,y\in\zd}\sum_{i=1}^d\sum_{j=1}^df_i(\eps x)f_j(\eps y)\nabla^2p\d s-\mathfrak{Q}_f\bigg)^2\\
&=\mathsf{E}_\mu\bigg(\int_0^\infty\bigg(\eps^{d-2}\sum_{x,y\in\zd}\sum_{i=1}^d\sum_{j=1}^df_i(\eps x)f_j(\eps y)\nabla^2p-(\nabla\cdot f,\e^{s Q}\nabla \cdot f)_{L^2(\R^d)}\bigg)\d s\bigg)^2\\
&\le\left(\int_0^\infty \frac{1}{(1\vee s)^{(d+1)/2}}\d s\right)\mathsf{E}_\mu\int_0^\infty (1\vee s)^{(d+1)/2}\\
&\qquad\qquad\times\bigg(\eps^{d-2}\sum_{x,y\in\zd}\sum_{i=1}^d\sum_{j=1}^df_i(\eps x)f_j(\eps y)\nabla^2p-(\nabla\cdot f,\e^{s Q}\nabla \cdot f)_{L^2(\R^d)}\bigg)^2\d s\\
&\le C(d)\int_0^\infty (1\vee s)^{(d+1)/2}\\
&\qquad\times\mathsf{E}_\mu\bigg(\eps^{d-2}\sum_{x,y\in\zd}\sum_{i=1}^d\sum_{j=1}^df_i(\eps x)f_j(\eps y)\nabla^2p-(\nabla\cdot f,\e^{s Q}\nabla \cdot f)_{L^2(\R^d)}\bigg)^2\d s,
\end{split}
\end{equation}
for some $C(d)>0$. For the first inequality we multiplied and divided by $(1\vee s)^{(d+1)/2}$ in the term after the equality sign, and then we applied the Cauchy-Schwarz inequality. For the second inequality, we interchanged the order of integration in the second term, and bounded the first term by $C(d)$.
By \eqref{e:scalinglimitvariancefixedt}, the integrand on the right-hand side of \eqref{e:scalinglimitvariance1} goes to 0 for each fixed $s$, and so \eqref{e:scalinglimitvarianceintegrated} follows from \eqref{e:scalinglimitvariance1} and dominated convergence provided that we can find an integrable majorant for the integrand. 

By assumption $f\in C_c^\infty(\R^d,\R^d)$, and so in particular all derivatives of $f$ are bounded. Moreover, we can find some $R\ge 1$ such that $f(x)=0$ whenever $|x|\ge R$. Using below summation by parts for the equality (in the $y$-variable only), as well as \eqref{e:heatkernelfirstder} for the second inequality, we can estimate
\begin{equation}\label{e:scalinglimitvariance2}
\begin{split}
&\mathsf{E}_\mu\bigg(\eps^{d-2}\sum_{x,y\in\zd}\sum_{i=1}^d\sum_{j=1}^df_i(\eps x)f_j(\eps y)\nabla\nabla' p^{[\kappa],\eta}(0,(x,x+e_i),s\eps^{-2},(y,y+e_j))\bigg)^2\\
&=\mathsf{E}_\mu\bigg(\eps^{d-1}\sum_{\substack{x,y\in\zd\\|x|,|y|<(R+1)/\eps}}\sum_{i=1}^d\sum_{j=1}^df_i(\eps x)\frac{f_j(\eps y)-f_j(\eps (y-e_j))}{\eps}\nabla p^{[\kappa],\eta}(0,(x,x+e_i),s\eps^{-2},y)\bigg)^2\\
&\le\mathsf{E}_\mu\eps^{2(d-1)}\times C\frac{R^{2d}}{\eps^{2d}}\\
&\qquad\times\sum_{\substack{x,y\in\zd\\|x|,|y|<(R+1)/\eps}}\sum_{i=1}^d\sum_{j=1}^d(f_i(\eps x))^2\left(\frac{f_j(\eps y)-f_j(\eps (y-e_j))}{\eps}\right)^2(\nabla p^{[\kappa],\eta}(0,(x,x+e_i),s\eps^{-2},y))^2\\
&\le C\eps^{2(d-1)}\frac{R^{2d}}{\eps^{2d}}\|f\|_{L^\infty}^2\|\nabla f\|_{L^\infty}^2\frac{C}{(1\vee s\eps^{-2})^{d+1}}\frac{R^{2d}}{\eps^{2d}}\le C\frac{R^{4d}}{(\eps^2\vee s)^{d+1}}\|f\|_{L^\infty}^2\|\nabla f\|_{L^\infty}^2.
\end{split}
\end{equation}
This estimate is useful for $s$ large. For small $s$ we can use the rather crude estimate 
\begin{align*}
&\left|\eps^{d-2}\sum_{x,y\in\zd}\sum_{i=1}^d\sum_{j=1}^df_i(\eps x)f_j(\eps y)\nabla\nabla' p^{[\kappa],\eta}(0,(x,x+e_i),s\eps^{-2},(y,y+e_j))\right|\\
&=\left|\eps^d\sum_{x,y\in\zd}\sum_{i=1}^d\sum_{j=1}^d\frac{f_i(\eps x)-f_i(\eps (x-e_i))}{\eps}\frac{f_j(\eps y)-f_j(\eps (y-e_j))}{\eps}p^{[\kappa],\eta}(0,x,s\eps^{-2},y)\right|\\
&\le\|\nabla f\|_{L^\infty(\R^d)}^2\eps^d\sum_{\substack{x,y\in\zd\\|x|,|y|<(R+1)/\eps}}p^{[\kappa],\eta}(0,x,s\eps^{-2},y)
\le CR^d\|\nabla f\|_{L^\infty(\R^d)}^2,
\end{align*}
where for the first equality we passed from the gradient on the $\nabla\nabla' p^{[\kappa],\eta}$ term to the difference terms on $f$ by opening up the sums and then exploiting the gradient nature of $\nabla\nabla' p^{[\kappa],\eta}$. In the last step we used that there are only $C\frac{R^d}{\varepsilon^d}$ choices of $x$, and for each fixed $x$ the sum over the $y$ is at most 1. 

In combination with \eqref{e:scalinglimitvariance2} this implies that
\begin{equation}\label{e:scalinglimitvariance3}
\begin{split}
&\mathsf{E}_\mu\bigg(\eps^{d-2}\sum_{x,y\in\zd}\sum_{i=1}^d\sum_{j=1}^df_i(\eps x)f_j(\eps y)\nabla\nabla' p^{[\kappa],\eta}(0,(x,x+e_i),s\eps^{-2},(y,y+e_j))\bigg)^2\\
&\le C\frac{R^{4d}}{(1\vee\eps^2\vee s)^{d+1}}\left(\|f\|_{L^\infty(\R^d)}^4+\|\nabla f\|_{L^\infty(\R^d)}^4\right)
\end{split}
\end{equation}

For the other term in the integrand in \eqref{e:scalinglimitvariance1} we can argue similarly. Indeed, if we denote the heat kernel associated with the (constant-coefficient, elliptic) operator $Q$ by $p_Q$, then it (and its derivatives) have the same asymptotics as the heat kernel of the ordinary Laplacian. In particular we have the estimates
\[
(\nabla\cdot f,\e^{s Q}\nabla \cdot f)_{L^2(\R^d)}=\int_{\R^d}\int_{\R^d} \nabla\cdot f(x)\nabla\cdot f(y)p_Q(0,x,s,y)\d x\d y\le C\|\nabla f\|_{L^2(\R^d)}^2\le CR^d\|\nabla f\|_{L^\infty(\R^d)}^2\]
and also
\[
(\nabla\cdot f,\e^{s Q}\nabla \cdot f)_{L^2(\R^d)}=\int_{\R^d}\int_{\R^d}\sum_{i,j=1}^df_i(x)f_j(y)\nabla_i\nabla'_jp_Q(0,x,s,y)\d x\d y\le C\frac{R^{2d}}{s^{d/2+1}}\|f\|_{L^\infty(\R^d)}^2,\]
which can be combined into 
\begin{equation}\label{e:scalinglimitvariance4}
\begin{split}
(\nabla\cdot f,\e^{s Q}\nabla \cdot f)_{L^2(\R^d)}^2\le C\frac{R^{4d}}{(1\vee s)^{d+2}}\left(\|f\|_{L^\infty(\R^d)}^4+\|\nabla f\|_{L^\infty(\R^d)}^4\right).
\end{split}
\end{equation}
The combination of \eqref{e:scalinglimitvariance3} and \eqref{e:scalinglimitvariance4} now shows that the integrand in \eqref{e:scalinglimitvariance1} is bounded above (for $\eps\le1$, say) by
\[C\frac{(1\vee s)^{(d+1)/2}}{(1\vee s)^{d+1}}R^{4d}(\|f\|_{L^\infty(\R^d)}^4+\|\nabla f\|_{L^\infty(\R^d)}^4)\le C\frac{1}{(1\vee s)^{(d+1)/2}}R^{4d}\left(\|f\|_{L^\infty(\R^d)}^4+\|\nabla f\|_{L^\infty(\R^d)}^4\right)\]
which for $d\ge2$ is integrable. So we can pass to the limit $\eps\to0$ in \eqref{e:scalinglimitvariance1}, and \eqref{e:scalinglimitvarianceintegrated} follows. 

\textbf{Step 2: Proof of \eqref{e:scalinglimitvariancefixedt}}\\
It remains to prove \eqref{e:scalinglimitvariancefixedt}. Here it is convenient to rewrite things in terms of the random walk $X_{s\eps^{-2}}^{[\kappa],\eta}$. For that purpose we will use $p^{[\kappa],\eta}(0,x,t,y)=\mathbf{P}_x^{[\kappa],\eta} \left(X_t^{[\kappa],\eta}=y\right)=\mathbf{E}_x^{[\kappa],\eta} \left(\1_{X_t^{[\kappa],\eta}=y}\right)$. 
Using also summation by parts, we can calculate that (for $\eps<1$, say) we have 
\begin{equation}\label{e:scalinglimitvariance5}
\begin{split}
&\eps^{d-2}\sum_{x,y\in\zd}\sum_{i=1}^d\sum_{j=1}^d f_i(\eps x)f_j(\eps y)\nabla\nabla' p^{[\kappa],\eta}(0,(x,x+e_i),s\eps^{-2},(y,y+e_j))\\
&=\eps^d\sum_{x,y\in\zd}\sum_{i=1}^d\sum_{j=1}^d \frac{(f_i(\eps x)-f_i(\eps (x-e_i)))}{\eps}\frac{(f_j(\eps y)-f_j(\eps (y-e_j)))}{\eps}p^{[\kappa],\eta}(0,x,s\eps^{-2},y)\\
&=\eps^d\sum_{\substack{x\in\zd\\|x|<(R+1)/\eps}}\sum_{i=1}^d\sum_{j=1}^d \frac{(f_i(\eps x)-f_i(\eps (x-e_i)))}{\eps}\mathbf{E}_x^{[\kappa],\eta}\left(\frac{f_j(\eps  X_{s\eps^{-2}}^{[\kappa],\eta})-f_j(\eps X_{s\eps^{-2}}^{[\kappa],\eta}-\eps e_j)}{\eps}\right)\\
&=\eps^d\sum_{\substack{x\in\zd\\|x|<(R+1)/\eps}}\sum_{i=1}^d\sum_{j=1}^d \left(\partial_if_i(\epsilon x)+O(\eps\|\nabla^2f\|_{L^\infty})\right)\mathbf{E}_x^{[\kappa],\eta}\left(\partial_jf_j(\eps  X_{s\eps^{-2}}^{[\kappa],\eta})+O(\eps\|\nabla^2f\|_{L^\infty})\right)\\
&=\eps^d\sum_{\substack{x\in\zd\\|x|<(R+1)/\eps}}\left( \nabla\cdot f(\epsilon x)\mathbf{E}_x^{[\kappa],\eta}\nabla\cdot f(\eps  X_{s\eps^{-2}}^{[\kappa],\eta})+O(\eps\|\nabla f\|_{L^\infty}\|\nabla^2f\|_{L^\infty}+\eps^2\|\nabla^2f\|_{L^\infty}^2)\right),
\end{split}
\end{equation}
where for the last equality we used Taylor expansion to approximate finite differences by gradients.

To proceed, we need to find a much finer estimate for $(\nabla\cdot f,\e^{s Q}\nabla\cdot f)_{L^2(\R^d)}$ than the one derived in Step 1 above. We observe that $Q$ is the generator of a Brownian motion with covariance $2q$. We denote it by $Y_s$, and we denote the corresponding expectation when started from $y\in\R^d$ by $\mathbf{E}_y$. 
Then we have that 
\begin{equation}\label{e:scalinglimitvariance6}
\begin{split}
(\nabla\cdot f,\e^{s Q}\nabla\cdot f)_{L^2(\R^d)}&=\int_{\{y\colon|y|<R\}}\nabla \cdot f(y)\mathbf{E}_y \left(\nabla\cdot f(Y_s)\right)\d y\\
&=\sum_{\substack{x\in\zd\\|x|<(R+1)/\eps}}\int_{(0,\eps)^d}\nabla \cdot f(\eps x+y)\mathbf{E}_{\eps x+y}\left(\nabla\cdot f(Y_s)\right)\d y\\
&=\sum_{\substack{x\in\zd\\|x|<(R+1)/\eps}}\int_{(0,\eps)^d}\nabla \cdot f(\eps x+y)\mathbf{E}_{\eps x}\left(\nabla\cdot f(Y_s+y)\right)\d y\\
&=\sum_{\substack{x\in\zd\\|x|<(R+1)/\eps}}\int_{(0,\eps)^d}\left(\nabla \cdot f(\eps x)+O(\eps\|\nabla^2f\|_{L^\infty})\right)\left(\mathbf{E}_{\eps x} \nabla\cdot f(Y_s)+O(\eps\|\nabla^2f\|_{L^\infty}\right)\d y\\
&=\eps^d\sum_{\substack{x\in\zd\\|x|<(R+1)/\eps}}\left(\nabla \cdot f(\eps x)\mathbf{E}_{\eps x} \left(\nabla\cdot f(Y_s)\right)+O(\eps\|\nabla f\|_{L^\infty}\|\nabla^2f\|_{L^\infty}+\eps^2\|\nabla^2f\|_{L^\infty}^2)\right),
\end{split}
\end{equation}
where for the fourth equality we applied Taylor expansion to each of the two terms in the third equality, and for the fifth equality we rearranged the terms around.

Combining \eqref{e:scalinglimitvariance5} and \eqref{e:scalinglimitvariance6}, we find that
\begin{equation}\label{e:scalinglimitvariance7}
\begin{split}
&\mathsf{E}_\mu\bigg(\eps^{d-2}\sum_{x,y\in\zd}\sum_{i=1}^d\sum_{j=1}^df_i(\eps x)f_j(\eps y) \nabla\nabla' p^{[\kappa],\eta}(0,(x,x+e_i),s\eps^{-2},(y,y+e_j))-(\nabla\cdot f,\e^{s Q}\nabla \cdot f)_{L^2(\R^d)}\bigg)^2\\
&=\mathsf{E}_\mu\Bigg[\Bigg(\eps^d\sum_{\substack{x\in\zd\\|x|<(R+1)/\eps}}\Big( \nabla\cdot f(\epsilon x)\left(\mathbf{E}_x^{[\kappa],\eta}\nabla\cdot f(\eps  X_{s\eps^{-2}}^{[\kappa],\eta})-\mathbf{E}_{\eps x} \left(\nabla\cdot f(Y_s)\right)\right)\Big)\Bigg)\\
&\qquad\qquad+O\left(\eps^d\cdot\frac{R^d}{\eps^d}\left(\eps\|\nabla f\|_{L^\infty}\|\nabla^2f\|_{L^\infty}+\eps^2\|\nabla^2f\|_{L^\infty}^2\right)\right)\Bigg]^2\\
&\le2\mathsf{E}_\mu\bigg(\eps^d\sum_{\substack{x\in\zd\\|x|<(R+1)/\eps}}\Big( \nabla\cdot f(\epsilon x)\left(\mathbf{E}_x^{[\kappa],\eta}\nabla\cdot f(\eps  X_{s\eps^{-2}}^{[\kappa],\eta})-\mathbf{E}_{\eps x} \left(\nabla\cdot f(Y_s)\right)\right)\Big)\bigg)^2\\
&\qquad\qquad+C\left(R^d\left(\eps\|\nabla f\|_{L^\infty}\|\nabla^2f\|_{L^\infty}+\eps^2\|\nabla^2f\|_{L^\infty}^2)\right)\right)^2\\
&\le2\bigg(\eps^d\sum_{\substack{x\in\zd\\|x|<(R+1)/\eps}} |\nabla\cdot f(\epsilon x)|^2\bigg)\mathsf{E}_\mu\bigg(\eps^d\sum_{\substack{x\in\zd\\|x|<(R+1)/\eps}}\left(\mathbf{E}_x^{[\kappa],\eta}\nabla\cdot f(\eps  X_{s\eps^{-2}}^{[\kappa],\eta})-\mathbf{E}_{\eps x} \left(\nabla\cdot f(Y_s)\right)\right)^2\bigg)\\
&\qquad\qquad+C\left(R^{2d}\left(\eps^2\|\nabla f\|_{L^\infty}^2\|\nabla^2f\|_{L^\infty}^2+\eps^4\|\nabla^2f\|_{L^\infty}^4)\right)\right)\\
&\le CR^d\|\nabla f\|_{L^\infty}^2 \mathsf{E}_\mu\bigg(\eps^d\sum_{\substack{x\in\zd\\|x|<(R+1)/\eps}}\left(\mathbf{E}_x^{[\kappa],\eta}\nabla\cdot f(\eps  X_{s\eps^{-2}}^{[\kappa],\eta})-\mathbf{E}_{\eps x} \left(\nabla\cdot f(Y_s)\right)\right)^2\bigg)\\
&\qquad\qquad+C\left(R^{2d}\left(\eps^2\|\nabla f\|_{L^\infty}^2\|\nabla^2f\|_{L^\infty}^2+\eps^4\|\nabla^2f\|_{L^\infty}^4)\right)\right),\\
\end{split}
\end{equation}
where for the second inequality in the above, we moved the $\nabla\cdot f(\epsilon x)$ terms out of the expectation by taking a rough bound involving summation after all $|x|<(R+1)/\epsilon$, and for the last inequality we used that $|x|<(R+1)/\epsilon$ to simplify that bound.
We claim that
\begin{equation}\label{e:scalinglimitvariance8}
\lim_{\eps\to0}\mathsf{E}_\mu\bigg(\eps^d\sum_{\substack{x\in\zd\\|x|<(R+1)/\eps}}\left(\mathbf{E}_x^{[\kappa],\eta}\nabla\cdot f(\eps  X_{s\eps^{-2}}^{[\kappa],\eta})-\mathbf{E}_{\eps x} \left(\nabla\cdot f(Y_s)\right)\right)^2\bigg)=0.
\end{equation}
Once we have shown this, the right-hand side of \eqref{e:scalinglimitvariance7} tends to 0 as $\eps\to0$, which implies \eqref{e:scalinglimitvariancefixedt}.

To see why \eqref{e:scalinglimitvariance8} holds, recall that by Theorem \ref{t:clt}, $\mathbf{E}_0^{[\kappa],\eta}\left(F(\eps X_{s\eps^{-2}}^{[\kappa],\eta})\right)$ converges in $\mathsf{P}_\mu$-probability to $\mathbf{E}_0 (F(Y_s))$, whenever $F\in C_b(\R^d)$. As $F$ is bounded, we can upgrade the convergence in probability to $L^2$-convergence, that is we have
\[\lim_{\eps\to0}\mathsf{E}_\mu\left(\mathbf{E}_0^{[\kappa],\eta}\left(F(\eps X_{s\eps^{-2}}^{[\kappa],\eta})\right)-\mathbf{E}_0 (F(Y_s))\right)^2\to0.\]
Actually this convergence holds uniformly for $F\in\Fcal$, where $\Fcal\subset C_b(\R^d)$ is a set of functions that are equicontinuous, uniformly bounded, and supported in a fixed compact set (as follows easily from the Arzelà-Ascoli theorem and a finite-net argument). The family $\Fcal=\{(\nabla \cdot f)(y+\cdot)\colon |y|<(R+1)\}$ has these properties, and hence (specializing to $y=\eps x$ for $x\in\Z^d$) we find that
\[\lim_{\eps\to0}\sup_{\substack{x\in\zd\\|x|<(R+1)/\eps}}\mathsf{E}_\mu\left(\left(\mathbf{E}_0^{[\kappa],\eta}\nabla\cdot f(\eps x+\eps  X_{s\eps^{-2}}^{[\kappa],\eta})-\mathbf{E}_{0} \left(\nabla\cdot f(\eps x+Y_s)\right)\right)^2\right)=0.\]
Using the shift-invariance of $\mathsf{P}_\mu$, this also implies that
\[\lim_{\eps\to0}\sup_{\substack{x\in\zd\\|x|<(R+1)/\eps}}\mathsf{E}_\mu\left(\left(\mathbf{E}_x^{[\kappa],\eta}\nabla\cdot f(\eps  X_{s\eps^{-2}}^{[\kappa],\eta})-\mathbf{E}_{\eps x} \left(\nabla\cdot f(Y_s)\right)\right)^2\right)=0.\]
This in turn immediately implies \eqref{e:scalinglimitvariance8}.
\end{proof}
Using Lemma \ref{l:scalinglimitvariances}, Theorem \ref{t:scalinglimit} follows from the Helffer-Sjöstrand representation in Theorem \ref{t:HS_annealed} and the fact that we have an a priori bound on the exponential moment-generating function (which follows from \eqref{e:Gaussianmoments}). The argument is based on \cite[Proof of Theorem A]{MR1461951}, but with some details filled in.
\begin{proof}[Proof of Theorem \ref{t:scalinglimit}]
Let $\Lambda_\eps(\lambda):=\E_\mu(\exp(\lambda F_\eps(\eta)))$ be the moment-generating function of $F_\eps(\eta)$. By the Brascamp-Lieb exponential moment inequality from Theorem \ref{t:existenceGGMs}(b), we have
\begin{equation}\label{e:scalinglimit2}
\Lambda_\eps(\lambda)\le\exp(C\lambda^2)
\end{equation}
for any $\lambda\in\R$ with an $\eps$-independent constant $C$ (and in particular, $\Lambda_\eps(\lambda)$ is well-defined and even real-analytic). Moreover, we can compute that
\[\Lambda_\eps'(\lambda)=\E_\mu(F_\eps(\eta)\exp(\lambda F_\eps(\eta)))\]
We want to apply Theorem \ref{t:HS_annealed} to the right-hand side in the above (noting that $F_\eps(\eta)$ is an odd function). Some care is needed, though, as neither $F_\eps$ nor $\exp(\lambda F_\eps)$ are bounded functions. So we will use a truncation argument similar to the one used for \eqref{e:HS_annealed_cov}. As in the proof of \eqref{e:HS_annealed_cov}, we let $\theta_M(x)\in C^2(\R)$ be an odd function such that $\theta_M(x)=x$ for $|x|\le M$, $|\theta_M(x)|\le C|x|$, $|\theta_M'(x)|\le C$, $|\theta_M''(x)|\le C$ for all $x$.
Then both $\theta_M(F_\eps)$ and $\theta_M(\exp(\lambda F_\eps))$ are in $C^2_{b,\text{loc}}(\chi)$, and $\theta_M(F_\eps)$ is still odd. So Theorem \ref{t:HS_annealed} implies that
\begin{equation}\label{e:scalinglimit1}
    \begin{split}
&\Cov_{\mathsf{P}_\mu}\big(\theta_M(\exp(\lambda F_\eps(\eta))),\theta_M(F_\eps(\eta))\big)\\
&=\int_0^\infty\sum_{x,y\in\zd}\mathsf{E}_\mu\left(\partial_x\big(\theta_M(\exp(\lambda F_\eps(\eta_0)))\big)\partial_y \theta_M(F_\eps(\eta_s))p^{[\kappa],\eta}(0,x,s,y)\right)\d s\\
&=\lambda\int_0^\infty\eps^d\sum_{x,y\in\zd}\sum_{i=1}^d\sum_{j=1}^d\left(f_i(\eps(x-e_i))-f_i(\eps x)\right)\left(f_j(\eps(y-e_j))-f_j(\eps y)\right)\\
&\qquad\qquad\times\mathsf{E}_\mu\left(\theta_M'(\exp(\lambda (F_{\eps}(\eta_0))))\exp(\lambda (F_{\eps}(\eta_0)))\theta_M'(F_\eps(\eta_s))
p^{[\kappa],\eta}(0,x,s,y)\right)\d s\\
&=\lambda\mathsf{E}_\mu\Bigg(\int_0^\infty\eps^d\sum_{x,y\in\zd}\sum_{i=1}^d\sum_{j=1}^df_i(\eps x)\left(f_j(\eps(y-e_j))-f_j(\eps y)\right)\\
&\qquad\qquad\times\theta_M'(\exp(\lambda (F_{\eps}(\eta_0))))\exp(\lambda (F_{\eps}(\eta_0)))\theta_M'(F_\eps(\eta_s)) \nabla p^{[\kappa],\eta}(0,(x,x+e_i),s,y)\d s\Bigg)
\end{split}
\end{equation}
In the last step we have used summation by parts in the $x$ variable only, and also Fubini's theorem. The latter is justified by noting that\[\big|\theta_M'(\exp(\lambda (F_{\eps}(\eta_0))))\exp(\lambda (F_{\eps}(\eta_0)))\theta_M'(F_\eps(\eta_s))\big|\le C\exp(\lambda (F_{\eps}(\eta_0)))\]
and that, if $f(x)=0$ for $|x|\ge R$, say, we have
\begin{equation}\label{e:scalinglimit5}
    \begin{split}
&\int_0^\infty\eps^d\Big|\sum_{x,y\in\zd}\sum_{i=1}^d\sum_{j=1}^d|f_i(\eps x)|\left|f_j(\eps(y-e_j))-f_j(\eps y)\right|\mathsf{E}_\mu\left(\exp(\lambda (F_{\eps}(\eta_0)))\nabla p^{[\kappa],\eta}(0,(x,x+e_i),s,y)\right)\Big|\d s\\
&\le C\frac{R^{2d}}{\eps^d}\|f\|_{L^\infty}^2\int_0^\infty\big(\mathsf{E}_\mu(\exp(2\lambda (F_{\eps}(\eta_0))))\big)^{1/2}\sup_{x,y\in\zd}\big(\mathsf{E}_\mu(\nabla p^{[\kappa],\eta}(0,(x,x+e_i),s,y))^2\big)^{1/2}\d s
\end{split}
\end{equation}
where the right-hand side is finite according to \eqref{e:heatkernelfirstder} and \eqref{e:scalinglimit2}. We note that for this calculation it was crucial that we used summation by parts only in one variable, as we needed to be able to use an $L^2$-bound for the heat kernel (the same phenomenon already occurred for \eqref{e:scalinglimitvariance2}).

Next, we want to use dominated convergence to take the limit $M\to\infty$ in \eqref{e:scalinglimit1}. On the right-hand side, the justification follows directly from \eqref{e:scalinglimit5}, while on the left-hand side it can be justified by noting that $\big|\theta_M(\exp(\lambda F_\eps(\eta)))\theta_M(F_\eps(\eta))\big|\le C\big|\exp(\lambda F_\eps(\eta))|F_\eps(\eta)|$ which is integrable by \eqref{e:scalinglimit2}. 

Accordingly we obtain that
\begin{align*}
&\Lambda_\eps'(\lambda)=\Cov_{\mathsf{P}_\mu}(\exp(\lambda F_\eps(\eta)),F_\eps(\eta))\\
&=\lambda\mathsf{E}_\mu\Bigg(\int_0^\infty\eps^d\sum_{x,y\in\zd}\sum_{i=1}^d\sum_{j=1}^df_i(\eps x)\left(f_j(\eps(y-e_j))-f_j(\eps y)\right)\exp(\lambda (F_{\eps}(\eta_0))) \nabla p^{[\kappa],\eta}(0,(x,x+e_i),s,y)\d s\Bigg)
\end{align*}
We can now apply summation by parts also in the $y$ variable and change the time variable from $s$ to $s\eps^{-2}$ to conclude that
\begin{equation}\label{e:scalinglimit6}
\begin{split}
&\Lambda_\eps'(\lambda)\\
&=\lambda\mathsf{E}_\mu\Bigg(\int_0^\infty\eps^{d-2}\sum_{x,y\in\zd}\sum_{i=1}^d\sum_{j=1}^df_i(\eps x)f_j(\eps y)\exp(\lambda (F_{\eps}(\eta_0))) \nabla \nabla'p^{[\kappa],\eta}(0,(x,x+e_i),s\varepsilon^{-2},(y,y+e_j))\d s\Bigg)\\
&=\lambda\mathfrak{Q}_f \Lambda_\eps(\lambda)+R(\lambda,\eps),
\end{split}
\end{equation}
where
\begin{equation}\label{e:scalinglimit3}
    \begin{split}
&R(\lambda,\eps)\\
&=\lambda\mathsf{E}_\mu\Bigg(\int_0^\infty\eps^{d-2}\sum_{x,y\in\zd}\sum_{i=1}^d\sum_{j=1}^df_i(\eps x)f_j(\eps y)\exp(\lambda F_\eps(\eta_0)) \nabla\nabla' p^{[\kappa],\eta}(0,(x,x+e_i),s\varepsilon^{-2},(y,y+e_j))\d s\\
&\qquad-\mathfrak{Q}_f\E_\mu(\exp(\lambda F_\eps(\eta)))\Bigg)\\
&=\lambda\mathsf{E}_\mu\Bigg(\exp(\lambda F_\eps(\eta_0))\\
&\qquad\times\Big(\int_0^\infty\eps^{d-2}\sum_{x,y\in\zd}\sum_{i=1}^d\sum_{j=1}^d\left(f_i(\eps x)f_j(\eps y) \nabla\nabla' p^{[\kappa],\eta}(0,(x,x+e_i),s\varepsilon^{-2},(y,y+e_j))\right)\d s-\mathfrak{Q}_f\Big)\Bigg).
\end{split}
\end{equation}
The right-hand side is bounded in absolute value using Cauchy-Schwarz by
\[|\lambda| (\Lambda_\eps(2\lambda))^{1/2}\Bigg(\mathsf{E}_\mu\bigg(\bigg(\int_0^\infty\eps^{d-2}\sum_{x,y\in\zd}\sum_{i=1}^d\sum_{j=1}f_i(\eps x)f_j(\eps y)\nabla\nabla' p^{[\kappa],\eta}(0,(x,x+e_i),s\varepsilon^{-2},(y,y+e_j))\d s-\mathfrak{Q}_f\bigg)^2\bigg)\Bigg)^{1/2},\]
which by Lemma \ref{l:scalinglimitvariances} and \eqref{e:scalinglimit2} tends to 0 as $\eps\to0$, uniformly for $\lambda$ in any compact set. So from \eqref{e:scalinglimit3} we learn that $R(\lambda,\eps)$ tends to 0 as $\eps\to0$, uniformly for $\lambda$ in any compact set.

Now for any fixed $\lambda$ we have that
\begin{align*}
\Lambda_\eps(\lambda)\exp\left(-\frac{\lambda^2}{2}\mathfrak{Q}_f\right)&=1+\int_0^\lambda \left(\Lambda_\eps'(\theta)-\theta\mathfrak{Q}_f \Lambda_\eps(\theta)\right)\exp\left(-\frac{\theta^2}{2}\mathfrak{Q}_f\right)\d \theta\\
&=1+\int_0^\lambda R(\theta,\eps)\exp\left(-\frac{\theta^2}{2}\mathfrak{Q}_f\right)\d \theta.
\end{align*}
The locally uniform convergence of $R(\theta,\eps)$ implies that the integral on the right-hand side goes to 0 as $\eps\to0$. So we obtain \eqref{e:scalinglimit}, which completes the proof of the theorem.
\end{proof}

\subsection{Decay of covariances}

\begin{theorem}
\label{t:decaycovariances}
Let $d\ge 2$. Let $V$ be a strongly log-concave log-mixture potential of quadratic growth, and let $\mu$ be a corresponding shift-invariant ergodic gradient Gibbs measure with zero tilt. Then $\mu$ satisfies for all $f,g\colon\bzd\to\R$ with compact support
\[\bigg|\Cov_{\mu}\bigg(\sum_{b\in\bzd}f(b)\eta(b),\sum_{b'\in\bzd}g(b')\eta(b')
\bigg)\bigg|\le C\|f\|_\infty\|g\|_\infty\sum_{\substack{b\in \bzd\cap\supp(f)\\b'\in \bzd\cap\supp(g)}}\frac{1}{1\vee|x_b-x_{b'}|^d},\]
for some $C>0$, where the sum is over the finite number of terms $b,b'$ in $f$ and $g$. 

More generally, for any $F,G\in C^2_{b,\text{loc}}(\chi)$ that are restrictions of functions $\tilde F,\tilde G\in C^1(\R^\bzd)$ to $\chi$ and such that at least one of $F,G$ is an odd function, we have that
\[\bigg|\Cov_{\mu}\big(F(\eta),G(\eta)\big)\bigg|\le C(d,V)\sum_{b,b'\in \bzd}\frac{||\partial_b \tilde F||_\infty||\partial_{b'} \tilde G||_\infty}{(1\vee |x_b-x_{b'}|)^d},\]
for some $C(d,V)>0$, where we used the notation
$\partial_b\tilde F(\eta):=\frac{d}{d\eta(b)}\tilde F(\eta)$.
\end{theorem}

\begin{proof}
We will prove only the first inequality, as the second one follows similarly. As shown in Theorem \ref{t:HS_annealed}, we have
\begin{equation}\label{e:rep_covariances}
\left|\Cov_{\mu}\big(F(\eta),G(\eta)\big)\right|\le\sum_{b,b'\in\bzd}|f(b)g(b')|\left|\int_0^\infty\nabla\nabla' p(0,b,s,b')\d s\right|.
\end{equation}
Now \eqref{e:heatkernelsecderoffdiag} implies that
\[\left|\int_0^\infty\nabla\nabla' p(0,b,s,b')\d s\right|\le C\int_0^\infty \frac{1}{(1\vee s)}p^*(cs,x_b-x_{b'})\d s.\]
For $p^*$ we have the bound
\[p^*(t,x)\le \frac{C}{(1\vee t)^{d/2}}\exp\left(-\frac{|x|}{(1\vee t)^{1/2}}\right)\]
as follows for example from \cite[Proposition B.3]{MR1872740}, and so we obtain 
\[\left|\int_0^\infty\nabla\nabla' p(0,b,s,b')\d s\right|\le C\int_0^\infty \frac{1}{(1\vee cs)^{(d+2)/2}}\exp\left(-\frac{|x_b-x_{b'}|}{(1\vee cs)^{1/2}}\right)\d s\le \frac{C}{(1\vee |x_b-x_{b'}|)^d},\]
where the last estimate is a straightforward computation. Inserting this into \eqref{e:rep_covariances}, the theorem immediately follows.
\end{proof}

\appendix

\paragraph{Acknowledgements}
The authors would like to thank Sebastian Andres, Oleksandra Antoniouk, Marek Biskup and Michael Röckner for helpful correspondence.

The authors would like to thank the Isaac Newton Institute for Mathematical Sciences, Cambridge, for support and hospitality during the programme 'Stochastic systems for anomalous diffusion', where part of the work on this paper was undertaken. This work was supported by EPSRC grant EP/Z000580/1. CC
has been partly supported by the EPSRC grant EP/M027694/1, and by a Lise Meitner Visiting Professorship held at Lund University.


\end{document}